%% file: main.tex
\documentclass[a4paper]{article}
\usepackage[utf8]{inputenc}
\usepackage[margin=1in]{geometry}
\usepackage{amsthm}
\usepackage{amsmath}
\usepackage{amssymb}
\usepackage{longtable}
\usepackage[colorlinks]{hyperref}
\usepackage{cleveref}
\crefformat{footnote}{#2\footnotemark[#1]#3}
\usepackage[cmyk]{xcolor}
\usepackage{listings} 
\definecolor{mycyan}{cmyk}{1,.2,0,0} 
\definecolor{mymagenta}{cmyk}{0,1,0,0}
\definecolor{myyellow}{cmyk}{0,0,1,.1}
\usepackage{enumerate}
\usepackage{tikz, tikz-cd, tkz-graph}
\usetikzlibrary{positioning,calc,arrows}
\usepackage{mathtools}
\usepackage{subcaption}
\usepackage{mleftright} 
\usepackage{footnote}
\makesavenoteenv{tabular}
\makesavenoteenv{table}

\newtheorem{theorem}{Theorem}[section]
\newtheorem{corollary}[theorem]{Corollary}
\newtheorem{proposition}[theorem]{Proposition}
\newtheorem{lemma}[theorem]{Lemma}
\newtheorem{definition}[theorem]{Definition}
\newtheorem{conjecture}[theorem]{Conjecture}

\theoremstyle{definition}
\newtheorem{example}[theorem]{Example}
\newtheorem{remark}[theorem]{Remark}
\newtheorem{question}[theorem]{Question}
\newtheorem{setup}[theorem]{Setup}
\newtheorem{claim}[theorem]{Claim}

\renewcommand{\L}{\operatorname{L}} 
\DeclareMathOperator{\Hstar}{H^\ast}
\DeclareMathOperator{\Hstari}{H^\ast_i}
\DeclareMathOperator{\Hstarimo}{H^\ast_{i-1}}

\DeclareMathOperator{\Hstartil}{H}
\newcommand{\hstar}{h^\ast}

\newcommand{\bzero}{\mathbf 0}
\newcommand{\be}{\mathbf e}
\newcommand{\bs}{\mathbf s}
\newcommand{\bu}{\mathbf u}
\newcommand{\bv}{\mathbf v}
\newcommand{\bw}{\mathbf w}
\newcommand{\bx}{\mathbf x}
\newcommand{\by}{\mathbf y}
\newcommand{\bz}{\mathbf z}
\newcommand{\beu}{\mathbf{e}_u}
\newcommand{\bev}{\mathbf{e}_v}
\newcommand{\bei}{\mathbf{e}_i}
\newcommand{\beii}{\mathbf{e}_{i+1}}
\newcommand{\bej}{\mathbf{e}_j}
\newcommand{\besi}{\mathbf{e}_{\sigma_i}}
\newcommand{\besj}{\mathbf{e}_{\sigma_j}}

\newcommand{\chitriv}{\chi_{\text{triv}}}
\newcommand{\chialt}{\chi_{\text{alt}}}

\newcommand{\C}{\mathbb{C}}
\newcommand{\Q}{\mathbb{Q}}
\newcommand{\R}{\mathbb{R}}
\newcommand{\Z}{\mathbb{Z}}

\newcommand{\up}{\operatorname{UP}}
\newcommand{\ehr}{\operatorname{L}}
\newcommand{\Ehr}{\operatorname{Ehr}} 
\newcommand{\EE}{\operatorname{EE}} 
\newcommand{\id}{\mathrm{id}}
\newcommand{\Sym}{\operatorname{Sym}} 
\newcommand{\Aut}{\operatorname{Aut}} 
\newcommand{\inc}{\operatorname{inc}} 
\newcommand{\vol}{\operatorname{vol}} 
\newcommand{\verts}{\operatorname{vert}} 
\DeclareMathOperator{\conv}{conv}     
\DeclareMathOperator{\aff}{aff} 
\newcommand{\val}{\operatorname{val}_2} 
\newcommand{\Cone}{\operatorname{Cone}} 
\newcommand{\trace}{\operatorname{trace}} 
\newcommand{\sym}{\operatorname{Sym}} 
\newcommand{\indeg}{\operatorname{indeg}} 
\newcommand{\Bx}{\operatorname{Box}} 
\newcommand\hobox{ 
    \unitlength .23 mm 
    \begin{picture}(10,10)(0,0)
    \linethickness{0.2mm}
    \put(0,-0.4){\line(0,1){10.83}}
    \linethickness{0.2mm}
    \put(-0.4,0){\line(1,0){10.83}}
    \linethickness{0.2mm}
    \multiput(10,0.11)(0,1.885){6}{\line(0,1){0.9}}
    \linethickness{0.2mm}
    \multiput(0.11,10)(1.885,0){6}{\line(1,0){0.9}}
    \end{picture}
    \,
}

\renewcommand{\path}{\mathsf{Path}}

\newcommand*\concat{\cdot} 
\newcommand*\nconcat{\odot} 

\newcommand*\ao{\mathfrak{o}} 

\title{Techniques in Equivariant Ehrhart Theory}
\author{Sophia Elia, Donghyun Kim, Mariel Supina}
\date{\today}

\begin{document}
\maketitle

\input{abstract}

\input{1_introduction}

\input{2_background}

\input{3_techniques}

\input{3.1_zonotopal_decompositions}

\input{3.2_invariant_decompositions}

\input{3.3_breaking_hstar}

\input{3.4_describing_invariant_hypersurfaces}

\input{4_further_questions}

\input{5_acknowledgements}

\input{6_appendix_Sagemath}

\pagebreak
\bibliographystyle{amsalpha}
\bibliography{sources.bib}

\end{document}

%% file: abstract.tex
\begin{abstract}
Equivariant Ehrhart theory generalizes the study of lattice point enumeration to also account for the symmetries of a polytope under a linear group action.
We present a catalogue of techniques with applications in this field, including zonotopal decompositions, symmetric triangulations, combinatorial interpretation of the $h^\ast$-polynomial, and certificates for the (non)existence of invariant non-degenerate hypersurfaces.
We apply these methods to several families of examples including hypersimplices, orbit polytopes, and graphic zonotopes, expanding the library of polytopes for which their equivariant Ehrhart theory is known.

\end{abstract}

%% file: 1_introduction.tex
    


\section{Introduction}
    
    
    Ehrhart theory studies the enumeration of lattice points in polytopes via the Ehrhart counting function $\L(P;t):= |tP\cap M'|$ for positive integers $t$ and a lattice $M'$.
    If $P$ is a lattice polytope, then $\L(P;t)$ agrees with a polynomial in $t$ which is known as the \emph{Ehrhart polynomial} \cite{Ehrhart1962}.
    When a lattice polytope $P$ is invariant under the linear action of a group $G$ on $M'$ with representation $\rho:G\to \operatorname{GL}(M')$, then we may study its \emph{equivariant Ehrhart theory}.
    This concept was introduced by Stapledon \cite{Stapledon} with motivation from toric geometry, representation theory, and mirror symmetry.
    The equivariant analogue of the Ehrhart polynomial is the character $\chi_{tP}$, defined as the character of the complex permutation representation on the lattice points in $tP$.
    The \emph{equivariant Ehrhart series} is then given as follows: 
    $$\EE(P;z) = \sum_{t \geq 0} \chi_{tP} z^t = \frac{\Hstar(P;z)}{\det (I - z \cdot \rho)} $$
    Evaluating this series at the identity element of the group returns the usual Ehrhart series of $P$.
    The numerator $\Hstar(P;z)$ is a formal power series with coefficients in $R(G)$, the character ring of $G$.
    It is an open problem to determine for which polytopes and group actions $\Hstar(P;z)$ is a polynomial and effective.
    Stapledon provided many conjectures around this topic, including:
    
    \begin{conjecture}[{\cite[Conjecture 12.1]{Stapledon}}]\label{conj:Stapledon}
    The following are equivalent:
        \begin{enumerate}[(i)]
        \itemsep 0em
            \item \label{item:nondegen} The toric variety of $P$ admits a $G$-invariant nondegenerate hypersurface with Newton polytope $P$. 
            
            \item \label{item:effective} The $\Hstar$-series is effective.
            \item \label{item:polynomial}The $\Hstar$-series is a polynomial.
        \end{enumerate}
    \end{conjecture}
    
    
    Stapledon further showed that \eqref{item:nondegen} $\implies$ \eqref{item:effective} $\implies$ \eqref{item:polynomial}; see Sections \ref{subsec:EEseries} and \ref{subsec:nondegenerate} for details.
    In email correspondence that they shared with the authors in 2019, Santos and Stapledon presented a counterexample to \eqref{item:polynomial} $\implies$ \eqref{item:nondegen}, which we include here with their permission.
    
    \begin{theorem}[\cite{SSemail}, Counterexample to \Cref{conj:Stapledon}]\label{thm:counterex}
    Let $P=[0,1]^3\subset\R^3$ be the $3$-cube and let $G=(\Z/2\Z)^2=\{\id, \sigma,\tau,\sigma\tau=\tau\sigma\}$ act on $P$ by $180$ degree rotations as described in \Cref{fig:cubeaction}.
    Then $\Hstar(P;z)$ is polynomial and effective, but the toric variety of $P$ does not admit a $G$-invariant nondegenerate hypersurface.
    \end{theorem}
    \begin{proof}
    By computing the Ehrhart series of the rational fixed polytope of the cube under each of these rotations, and then comparing these to a character table for $G$, one can find the $\Hstar$-series of $P$ under this action to be $\chi_{\text{triv}} + \chi_{\text{reg}}z + \chi_{\text{triv}}z^2$ where these characters indicate the trivial and regular representations of $G$.
    So the $\Hstar$-series is a polynomial and is effective.
    
    However, the toric variety of $P$ does not admit a $G$-invariant nondegenerate hypersurface.
    \Cref{fig:cubeaction} shows the two orbits of the vertices of $P$ under the action of $G$.
    So a $G$-invariant hypersurface in the toric variety of $P$ has the equation
        \[
        0 = (1+z+xy+xyz) + c(x+y+xz+yz) = (1+z)(1+xy+c(x+y))
        \]
    for some complex parameter $c$.
    Regardless of the value of $c$, such a hypersurface is singular when $z=-1$ and $1+xy+c(x+y)=0$.
    \end{proof}
    
    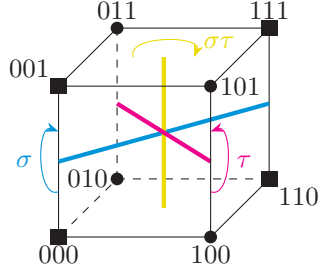
\begin{figure}
        \centering
        \input{img/cubeaction.tikz}
        \caption{Action of $(\Z/2\Z)^2$ on the $3$-cube. Each of $\sigma$, $\tau$, and $\sigma\tau$ acts by a rotation of $180$ degrees around a corresponding axis. The vertices are divided into two orbits under this action, which are indicated by their shapes.}
        \label{fig:cubeaction}
    \end{figure}
    
    In this article we present a variety of techniques for computing and studying the equivariant Ehrhart series.
    Our aim is to develop equivariant Ehrhart theory as its own branch of discrete geometry, and to investigate Stapledon's conjectures for a number of families of polytopes.
    We demonstrate how to use several tools to describe the equivariant Ehrhart series, including zonotopal decompositions, symmetric triangulations, combinatorial interpretations of the $\hstar$-polynomial, and computational methods using \texttt{Sagemath}.
    These techniques serve as a guidebook for future progress toward resolving the many open questions and conjectures in this field.
    
    We first provide more details about our setup in \Cref{sec:background}, as well as presenting the necessary background in equivariant Ehrhart theory and representation theory.
    We then arrive to our catalog of techniques in \Cref{sec:techniques}.
    
    \Cref{sec:zonotopes} investigates the case when $P$ is a graphic zonotope and $G$ is the automorphism group of the graph.
    This section generalizes past work on the permutahedron \cite{ArdilaSchindlerVindas, ASV}, and serves as a blueprint for studying the equivariant Ehrhart theory of general zonotopes.
    The main results of this section are \Cref{cor:graphic zonotope ehrhart series} which describes the equivariant Ehrhart series of a graphic zonotope, and \Cref{thm:graphic zonotope polynomial H*} in which we characterize the conditions under which the $\Hstar$-series of a graphic zonotope is a polynomial.

    In \Cref{sec:invariant decomps}, we use symmetric triangulations to calculate the equivariant Ehrhart series.
    The symmetric triangulations in question are called \emph{$G$-invariant half-open decompositions} and are closely related to partitionability of posets.  
    In \Cref{thm:permrep}, we generalize a theorem of Stapledon for the equivariant Ehrhart series of a simplex to lattice polytopes with $G$-invariant half-open decompositions.
    In \Cref{thm:eeseries2}, we show that $\chi_{tP}$ is polynomial in $t$ for polytopes with $G$-invariant half-open decompositions that satisfy an additional property on the group action, and give a formula for the equivariant Ehrhart series. This section appears in \cite{Elia2022}.
    
    \Cref{sec:breaking Hstar} deals with instances where we can break down the $\Hstar$-series using a combinatorial interpretation of its coefficients.
    In particular, we study hypersimplices under the action of the cyclic group, as well as permutahedra in prime dimensions.
    \Cref{thm: hypersimplex cyclic group} uses decorated ordered set partitions to describe the $\Hstar$-series of the hypersimplex, building off of \cite{Kim}.
    In \Cref{thm:cyclicaction_trivialfixed}, we give an explicit formula for the $\Hstar$-series of a polytope under a cyclic group action with trivial fixed subpolytopes. This is specified to the case of prime permutahedra in \Cref{cor:Hstarprime}, which also shows the $\Hstar$-series is polynomial and effective.
    
    In \Cref{sec:hypersurfaces} we focus on criteria \eqref{item:nondegen} from \Cref{conj:Stapledon}, demonstrating techniques for proving the existence or non-existence of $G$-invariant nondegenerate hypersurfaces.
    The main results of this section are \Cref{thm: op no hypersurface} and \Cref{cor:oddrectface} which characterize a family of $S_n$-orbit polytopes for which no such hypersurface can exist, and \Cref{thm: hypersimplex hypersurface} regarding hypersimplices, another subfamily of orbit polytopes which do exhibit these hypersurfaces.

    Finally, in \Cref{sec:further questions}, we collect the many open questions that arose during our work on this project, and in \Cref{appendix:sage}, we demonstrate the calculation of the equivariant Ehrhart series using our implemented code which is staged for release with \texttt{Sagemath} version 9.6.

%% file: img/cubeaction.tikz
\begin{tikzpicture}[scale=2,>=stealth']
\newcommand\Square[1]{+(-#1,-#1) rectangle +(#1,#1)}


\coordinate[label=below:000] (000) at (0,0,1);
\coordinate[label=below:100] (100) at (1,0,1);
\coordinate[label=left:010] (010) at (0,0,0);
\coordinate[label=below right:110] (110) at (1,0,0);
\coordinate[label=above left:001] (001) at (0,1,1);
\coordinate[label=right:101] (101) at (1,1,1);
\coordinate[label=above:011] (011) at (0,1,0);
\coordinate[label=above:111] (111) at (1,1,0);

\coordinate (yellow) at ($ .25*(001) + .25*(011) + .25*(111) + .25*(101) $);
\coordinate (magenta) at ($ .5*(100) + .5*(101) $);
\coordinate (cyan) at ($ .5*(000) + .5*(001) $);

\draw (000) -- (100) -- (110);
\draw[dashed] (110) -- (010) -- (000);
\draw (000) -- (001);
\draw (110) -- (111);
\draw[dashed] (010) -- (011);
\draw[ultra thick,mycyan] (cyan) -- ($ .5*(110) + .5*(111) $);
\draw[ultra thick,myyellow] (yellow) -- ($ .25*(000) + .25*(010) + .25*(110) + .25*(100) $);
\draw[ultra thick,mymagenta]  ($ .5*(010) + .5*(011) $) -- (magenta);
\draw (100) -- (101);
\draw (001) -- (101) -- (111) -- (011) -- cycle;



\fill (000) \Square{1.5pt};
\fill (100) circle(1.2pt);
\fill (010) circle(1.2pt);
\fill (001) \Square{1.5pt};
\fill (110) \Square{1.5pt};
\fill (101) circle(1.2pt);
\fill (011) circle(1.2pt);
\fill (111) \Square{1.5pt};

\draw[myyellow,-{Stealth[scale=1]}] ($ (yellow) - (.2,0,0) $) to [out=110,in=70] ($ (yellow) + (.2,0,0) $) node[above right]{$\sigma\tau$};
\draw[mycyan,-Stealth] ($ (cyan) - (0,.2,0) $) to [out=200,in=160] node[left]{$\sigma$} ($ (cyan) + (0,.2,0) $);
\draw[mymagenta,-Stealth] ($ (magenta) - (0,.2,0) $) to [out=-20,in=20] node[right]{$\tau$} ($ (magenta) + (0,.2,0) $);

\end{tikzpicture}

%% file: 2_background.tex
\section{Background}\label{sec:background} 
    
We give relevant background on Ehrhart theory in \Cref{subsec:Ehrahrt_Theory}. In \Cref{subsec:equisetup}, we present the general setup for equivariant Ehrhart theory that is used throughout. \Cref{subsec:reptheory} presents background on representation theory, and \Cref{subsec:EEseries} presents background on the equivariant Ehrhart series, in particular providing rational generating functions. In \Cref{subsec:nondegenerate}, we discuss $G$-invariant nondegenerate hypersurfaces and a useful corollary for their detection. In \Cref{sec:restricted_reps}, we present some motivating lemmas on equivariant Ehrhart theory for restricted representations.   

\subsection{Ehrhart Theory}\label{subsec:Ehrahrt_Theory}
    
The main reference we follow for Ehrhart theory is \cite{beck_robins}.
Fix a lattice $M\subset \R^n$.
The \emph{Ehrhart counting function} of a polytope
 $P \subset \R^n$, written $\ehr(P;t)$, gives the number of lattice points in the $t$-th dilate of $P$ for $t \in \Z _{\geq 1}$: $$\ehr(P;t)= |tP \cap M|.  $$
Ehrhart's theorem \cite{Ehrhart1962} says that if $P$ is a lattice polytope (i.e. the vertices of $P$ are contained in $M$), then for positive integers, $\ehr(P;t)$ agrees with a polynomial in $t$ of degree equal to the dimension of $P$. Furthermore, the constant term of this polynomial is equal to 1 and the coefficient of the leading term is equal to the Euclidean volume of $P$ within its affine span. 
The interpretation of other coefficients of the Ehrhart polynomial is an active direction of research, see for example \cite{ferroni2021ehrhart}. 

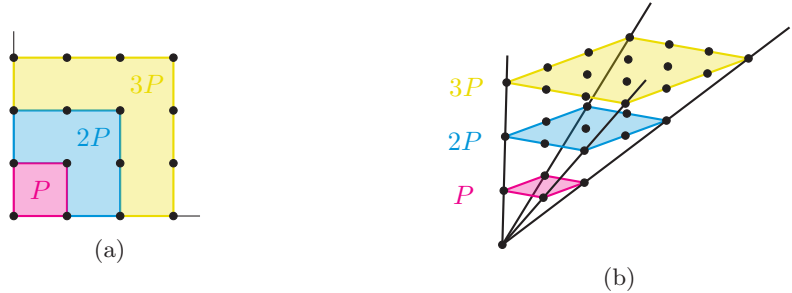
\begin{figure}
    \centering
        \begin{subfigure}{.4\linewidth}
        \centering
        \input{img/ehrhart_example.tikz}
        \caption{}
        \label{subfig:ehrhart poly}
        \end{subfigure}
    \hspace{1mm}
        \begin{subfigure}{.4\linewidth}
        \centering
        \input{img/cone_example.tikz}
        \caption{}
        \label{subfig:ehrhart series}
        \end{subfigure}
    \caption{\eqref{subfig:ehrhart poly} The first three dilates of $P$ and \eqref{subfig:ehrhart series} the cone over $P$.}
    \label{fig:ehrhart}
\end{figure}

Generating functions play a central role in Ehrhart theory. 
The \emph{Ehrhart series} $\Ehr(P;z)$ of a polytope$~P$ is the formal power series given by $$\Ehr(P;z) = 1 + \sum_{t \geq 1} \ehr(P;t)z^t. $$ 
Let $\Cone(P) \coloneqq \{\bx \in \R^{n+1} : \bx = \lambda(\by,1), \text{ where }\by \in P,\, \lambda \geq 0 \}$ as in \Cref{fig:ehrhart}.
We may view the coefficient of $z^t$ in the Ehrhart series as counting the number of lattice points in $\Cone(P)$ at height $t$. This viewpoint allows us to show that if
 $P$ is a $d$-dimensional lattice polytope, then the Ehrhart series has the rational generating function
$$\Ehr(P;z)=1 + \sum_{t \geq 1}\ehr(P;t)z^t = \frac {\hstar(P;z)}{(1-z)^{d+1}} ,$$
where $\hstar(P;z) = \sum_{i=0}^d \hstar_i z ^i $ is a polynomial in $z$ of degree at most $d$, called the \emph{$\hstar$-polynomial}. 
Furthermore, each $\hstar_i$ is a non-negative integer \cite{stanley1980decompositions}. 
The coefficients of the $\hstar$-polynomial form the \emph{$\hstar$-vector}: $(\hstar_0,\hstar_1,\dots, \hstar_d)$.
%
The Ehrhart polynomial may be recovered from the $\hstar$-vector through the transformation \begin{equation}\label{eq:hstartoehr} \ehr_P(t) = \sum_{i=0}^{d}\hstar_i \binom{t+d -i}{d}.\end{equation}

Let $P \subseteq \R^d$ be a rational $d$-polytope with \emph{denominator} $k$, i.e., $k$ is the smallest positive integer such that $kP$ is a lattice polytope.
Then $\ehr(P;t)$ is a \emph{quasipolynomial} with period dividing $k$, i.e., of the form
$\ehr( P;t) = c_d(t) t^d + \dots + c_1(t) t + c_0(t)$ where $c_0(t), c_1(t),
\dots, c_d(t)$ are periodic functions. 
In this case, the Ehrhart series has the rational generating function
\begin{equation}\label{eq:ehrgenfct1}
 \Ehr(P;z) \ := \sum_{t\in \Z_{\geq0}} \ehr(P;t) \, z^t \ 
 = \ \frac{\hstar(P;z)}{(1-z^k)^{d+1}},
\end{equation}
where $\hstar(P;z) \in \Z[z]$ has degree $< k (d+1)$.
The choice of denominator is no longer canonical.

     \subsection{The Equivariant Setup}\label{subsec:equisetup}
     For a lattice and $\Z$-module $M$, we write $M_{\R}$ for $M \otimes_{\Z} \R$.

    \begin{setup}\label{setup}
    Let $G$ be a finite group acting linearly on a lattice $M' \cong M \times \Z$ of rank $d+1$ such that the $\Z$-coordinate of the lattice points in $M'$ is preserved under the action of $G$. 
    Let $P \subset  M'_\R $ be a $d$-dimensional $G$-invariant polytope with vertices in $M \times \{1\}$, where \emph{$G$-invariant} means that as a set, $g(P) = P$, for all $g \in G$.
    \end{setup}
    
    We assume $M = \Z^d$ when convenient.
    The \emph{exponent} of a finite group $G$ is the smallest positive integer $N$ such that $g^N = \id_G$ for all $g \in G$.
    For an element $g \in G$, the subset of $P$ fixed by $g$, denoted $P^g$ and called the \emph{fixed subpolytope}, is the convex hull of the barycenters of orbits of vertices of $P$  under the action of $g$ \cite[Lemma 5.4]{Stapledon}.
    This implies that any fixed subpolytope $P^g$ for $g \in G$ is a rational polytope with denominator dividing $N$.

    \subsection{Representation Theory}\label{subsec:reptheory}
    For a nice introduction to representation theory of finite groups, see \cite{Sagan} or \cite{Serre}.
    Let $V$ be an $n$-dimensional vector space over a field $K$ of characteristic $0$. 
    A \emph{representation} $\rho$ of a group $G$ on $V$ is a group homomorphism $\rho: G \rightarrow GL(V)$ from $G$ to the group of invertible $K$-linear transformations of~$V$. Equivalently $\rho$ may be seen as a group homomorphism from $G$ to the group $GL_n(V)$ of $n\times n$ invertible matrices with entries in $K$.
    A subspace $W \subseteq V$ is called \emph{ $G$-invariant} if $W = g(W)$ as a set, for all $g \in G$.
    A representation is \emph{irreducible} if there are no nontrivial, proper invariant subspaces of $V$ under the action of $G$.

    Let $\rho: G \rightarrow GL_n(V)$ be a representation of $G$. 
    The \emph{character} of $\rho$, written $\chi_\rho$, is the function $G \rightarrow \C$ such that $\chi_\rho(g) := \trace(\rho(g))$;
    the trace of a matrix is the sum of its diagonal entries.
    The characters of irreducible representations are referred to as \emph{irreducible characters}.
    A group $G$ always has a \emph{trivial representation} on a vector space $V$ which sends each group element to the identity map on $V$.
    The character of the trivial representation is denoted by $\chitriv$ throughout.
    Characters are \emph{class functions}, functions from the group to the complex numbers that take the same value on every conjugacy class.
    In fact, a class function is a character of a representation if and only if it can be written as a nonnegative integral linear combination of the irreducible characters of $G$. 
    We write $R^+(G)$ for the set of the class functions that are characters and refer to them as \emph{effective characters}.
    The group generated by $R^+(G)$ is denoted $R(G)$ and called the \emph{character ring}; it consists of all class functions that can be written as linear integral combinations of the irreducible characters. 
    Elements of $R(G)$ are referred to as \emph{virtual characters}.
    The character ring $R(G)$ is a subring of the $\C$-vector space $F_{\C}(G)$ of class functions on $G$ with values in $\C$, called the \emph{ring of class functions}.
    There is an inner product on $F_{\C}(G)$ such that for two class functions $\phi,\chi$ of $G$,
    $\langle \phi, \chi \rangle = \frac{1}{|G|}\sum_{g \in G} \phi(g) \overline{\chi(g)}$, where~$\overline{\cdot}$ denotes the complex conjugate.
    
    Let a group $G$ act on a finite (ordered) set $S$. The \emph{permutation representation} of $G$ on $S$ is the group homomorphism that sends every element of $G$ to the corresponding permutation matrix for the action of $G$ on $S$. The character of the permutation representation is called the \emph{permutation character}.

    \subsection{Equivariant Ehrhart Theory}\label{subsec:EEseries}

    Let $\chi_{tP}$ be the complex permutation character induced by the action of $G$ on the lattice points in $tP \cap M'$. 
    
    \begin{definition}\label{def:equivehrseries}
    The \emph{equivariant Ehrhart series}, $\EE(P;z)$, is the formal power series in $R(G)[[z]]$, such that the coefficient of $z^t$ for $t \in \Z_{\geq 0}$ is the character $\chi_{tP}$: 
        $$ \EE(P;z) = \sum_{t \geq 0} \chi_{tP} z^t.$$ 
    \end{definition}
    
    Evaluating the series at $g \in G$, $\EE(P;z)(g)$, gives the Ehrhart series of the fixed subpolytope $P^g$ {\cite[Lemma 5.2]{Stapledon}}.

\begin{theorem}[{\cite[Theorem 5.7]{Stapledon}}]\label{thm:chiquasi}
Let $P$ be a lattice $d$-polytope invariant under the action of a group $G$ as in the Setup \ref{setup} and exponent $N \geq 1$.
As a function of $t$, $\chi_{tP}$ is quasipolynomial, that is,
$$
\chi_{tP} = f_0(t)t^0 + f_{1}(t)t^1 + \dots + f_d(t) t^d,
$$
where $f_0(t), f_1(t), \dots, f_d(t) \in F_{\C}(G)$ are periodic functions in $t$ with period dividing $N$. 
\end{theorem}
\begin{proof}
For all $g \in G$, the Ehrhart counting function $\ehr(P^g;z)$ can be expressed as a quasipolynomial of period $N$ and degree equal to $\dim(P^g) \leq d$.
Therefore,
\begin{align*}
\EE(P;z)(g) = \sum_{t \geq 0}\chi_{tP}(g)z^t &= \sum_{a \geq 0}\sum_{j = 0}^{N-1}\chi_{(aN+j)P}(g)z^{aN+j} \\
& = \sum_{a \geq 0}\sum_{j = 0}^{N-1} \ehr(P^g;aN+j)z^{aN+j}\\
& = \sum_{a \geq 0}\sum_{j = 0}^{N-1} \left( \sum_{i = 0}^{d}c_{j,i}^g (aN+j)^i \right)z^{aN+j},
\end{align*}
where $c_{j,i}^g \in \Q$ is the coefficient of the degree $i$ term in the $j$-th constituent of the Ehrhart quasipolynomial of $P^g$. 
Let $\{g_1, \dots, g_k \}$ be conjugacy class representatives of $G$.
Define class functions
$f_{j,i}[g_\ell] \coloneqq c_{j,i}^{g_\ell}$ for all $j\in[N]$, $\ell \in [k],$ and $i \in \{0,1,\dots,d \}$.
Then for $j \in [N]$, and $t \equiv j \mod N$, $\chi_{tP} = {f_{j,0}t^0 + \dots + f_{j,d} t^d}$.
\end{proof}
\begin{corollary}\label{cor:quasiirreds}
Fix $j \in [N]$. For all $t \equiv j \mod N$, $\chi_{tP}$ may be expressed as a linear combination of the irreducible characters of $G$ with coefficients in $\Q[t]$.
\end{corollary}
\begin{corollary}
Let $P$ be a rational polytope with denominator $S \in \Z_{> 0}$ such that~$P$ is invariant under the linear action of a group $G$. 
Then $\chi_{tP}$ is quasipolynomial in $t$ with period dividing $S N$.  
\end{corollary}
\begin{theorem}\label{thm:eerational}
Let $P$ be a $G$-invariant lattice polytope of dimension $d$ as in the Setup \ref{setup}, where $G$ has exponent $N$ and irreducible characters $\{\chi_1, \dots, \chi_k \}$.
The equivariant Ehrhart series has the following rational expression:
$$
\EE(P;z) = \sum_{t \geq 0}\chi_{tP}z^t = \frac{\tilde{\Hstartil}(P;z)}{\left(1-z^N\right)^{d+1}},
$$
where $\tilde{\Hstartil}(P;z)$ is a polynomial with coefficients in $R(G) \otimes_{\Z} \Q $ of degree ${\leq N(d+1) -1 }$.
\end{theorem}

    
    
    We therefore have two different rational expressions for the equivariant Ehrhart series:
    $$
     \EE(P;z) = \sum_{t \geq 0} \chi_{tP} z^t = \frac{\tilde{\Hstartil}(P;z)}{\left(1-z^N\right)^{d+1}} = \frac{\Hstar(P;z)}{\det(I - z \cdot \rho )} \, ,
    $$
    where $\Hstar(P;z) \in R(G)[[z]]$.
    The $\Hstar$-series is \emph{effective} if the coefficient of $z^i$ is an effective character for all $i$.
    
    \begin{remark}
    In \cite{Stapledon}, the $\Hstar$-series is denoted by $\phi$, and $\rho$ is the representation of $G$ acting on $M_\R$. Accordingly, the denominator in \cite{Stapledon} is 
    $
    (1-z)\det(I - z \cdot \rho).
    $
    \end{remark}

    Understanding the $\Hstar$-series is the main goal of Stapledon's  \Cref{conj:Stapledon}, and he stated that \eqref{item:effective} implies \eqref{item:polynomial}. We provide a short proof:

    \begin{lemma}\label{lem:effectiveimpliespoly}
    If $\Hstar(P;z)$ is effective then $\Hstar(P;z)$ is a polynomial.
    \end{lemma}
    \begin{proof}
    The $\Hstar$-series is a priori an infinite formal sum:
    $\Hstar(P;z) = \sum_{i \geq 0} \Hstari z^i$.
    On the level of series, we have:
    $$
     \frac{\Hstar(P;z)(\id_G)}{\det(I - z \cdot \rho(\id_G) )} = \EE(P;z)(\id_G) = \Ehr(P;z) = \frac{h^*(P;z)}{(1-z)^{d+1}}.
    $$
    As $\det (I - z \cdot \rho(\id_G)) = (1-z)^{d+1}$, 
    $\Hstar(P;z)(\id_G) = h^*(P;z)$.
    Let $\{\chi_1 , \dots, \chi_k \}$ be the irreducible characters of $G$. 
    Since $\Hstar(P;z)$ is effective, $\Hstari = \sum_{j = 1}^{k } c_{i,j}\chi_{j}$, with $c_{i,j} \in \Z_{\geq 0}$ for all $i,j$. As $\chi_j(\id_G) > 0$ for all $j$, no cancellation can occur and $\Hstar(P;z)$ must be polynomial.
    \end{proof}

    \subsection{Nondegenerate Hypersurfaces}\label{subsec:nondegenerate}
    
    One motivation for considering the $\Hstar$-series is its connection to toric geometry, see \cite[Sections 7 and 8]{Stapledon}. 
    A lattice polytope $P \subset \R^d$ defines a projective toric variety $X_P$. 
    A hypersurface in $X_P$ is given by the vanishing set of $f = \sum_{\bv \in P \cap \Z^d} c_\bv x ^\bv $, where $c_\bv \in \C$ and $x^\bv$ denotes the monomial $x_1^{v_1}x_2^{v_2}\cdots x_n^{v_n}$.
    The Newton polytope of $f$ is the convex hull of the lattice points corresponding to monomials of $f$ that have nonzero coefficients $c_\bv$.
    A hypersurface has Newton polytope $P$ if $c_\bv \neq 0$ for all vertices $\bv$ of $P$.
    Suppose that a lattice polytope $P$ is invariant under the linear action of a group $G$.
    Then a hypersurface in $X_P$ is said to be \emph{$G$-invariant} if $c_\bv = c_\bw$ for all lattice points $\bv,\bw$ in the same $G$-orbit of $P\cap\Z^d$.
    The hypersurface is \emph{smooth} if the gradient vector $(\frac{\partial f}{\partial x_1}, \dots, \frac{\partial f}{\partial x_n})$ is never zero when $x_1, \dots, x_n \in \C^*$.
    The hypersurface is \emph{nondegenerate} if~$f |_{Q} = \sum_{\bv \in Q\cap \Z^d} c_\bv x^\bv $ is smooth for all faces $Q$ of $P$.
    

\begin{theorem}[{\cite[Theorem 7.7]{Stapledon}}]
If there exists a $G$-invariant nondegenerate hypersurface with Newton polytope $P$, then $\Hstar (P;z)$ is effective. In particular, this assumption holds if the linear system $\Gamma(X_P,L)^G$ of $G$-invariant global sections on the corresponding line bundle $L$ is base point free.
\end{theorem}

Recall that a subspace $W \subseteq \Gamma (X,L)$ is \emph{basepoint free} if for every $p \in X$, there exists $s \in W$ with $s(p) \neq 0$ \cite[Section 6.0]{cox2011toric}. For a face $Q \subset P$ let $G_Q$ denote the stabilizer of $Q$. By \cite[Remark 7.8]{Stapledon}, the linear system of $G$-invariant global sections $\Gamma(X_P,L)^G$ on $X_P$ is basepoint free if and only if for each face $Q \subset P$, the linear system
\begin{equation}\label{eq:linearsystem}
    \left \{ \sum_{\bu \in Q \cap M} c_\bu x^ \bu \quad \bigg \lvert \quad c_\bu = c_{\bu '} \text{ if $\bu,\bu '$  lie in the same $G_Q$-orbit }    \right \} 
\end{equation}
on the torus $T$ is basepoint free.
Additionally, this condition is automatically satisfied for faces of dimension $\le 1$.

Suppose that every face $Q$ of $P$ with $\dim (Q) >1$ contains a lattice point $\bu_Q$ that is $G_Q$-fixed. 
To show that the linear system $\eqref{eq:linearsystem}$ is basepoint free on $T$, we need to show that for every $x \in T$, there exists a polynomial in the system that doesn't vanish at $x$. 
In fact, we have something stronger, since the polynomial $x^{\bu _Q}$ is an element of the linear system and doesn't vanish at \emph{any} $x \in T$. 
Thus we recover the following combinatorial criterion that certifies the existence of a $G$-invariant nondegenerate hypersurface with Newton polytope $P$. 

\begin{corollary}[{\cite[Corollary 7.10]{Stapledon}}, reworded]\label{cor:seventen}
If every face $Q$ of $P$ with $\dim(Q) >1 $ contains a lattice point that is $G_Q$-fixed, where $G_Q$ denotes the stabilizer of $Q$, then the toric variety of $P$ admits a $G$-invariant nondegenerate hypersurface with Newton polytope $P$. 
\end{corollary}

    \subsection{Equivariant Ehrhart Theory of Restricted Representations}\label{sec:restricted_reps}

    Let $P$ be a polytope invariant under the action of $G$ as in the Setup \ref{setup}, and let $\mathcal H$ be a subgroup of $G$. The action of $G$ on $P$ induces an action of $\mathcal H$ on $P$ by restriction. Let $\Hstar_G(P;z)$ and $\Hstar_{\mathcal H}(P;z)$ be their respective equivariant $\Hstar$-series.
    
    \begin{theorem}\label{thm:subrephypersurface}
    If the toric variety of $P$ admits a $G$-invariant nondegenerate hypersurface with Newton polytope $P$, then the same hypersurface is a nondegenerate $\mathcal H$-invariant hypersurface.
    \end{theorem}
    \begin{proof}
    Let $\sum_{\mathbf v \in P \cap \Z^d} c_\bv x^\bv = 0$ be a $G$-invariant nondegenerate hypersurface. Every $\mathcal H$-orbit of lattice points in $P$ is contained in a $G$-orbit of lattice points. Thus $c_\bv = c_\bw$ for all $\mathbf v, \mathbf w$ in an orbit of $\mathcal H$, and the hypersurface is $\mathcal H$-invariant. The non-degeneracy condition is independent of the group.
    \end{proof}
        
    \begin{theorem}
    If $\Hstar_G(P;z)$ is effective then $\Hstar_{\mathcal H}(P;z)$ is effective.
    \end{theorem}
    \begin{proof}
    The coefficient of each $z^i$ in $\Hstar_G(P;z)$ is an effective character.
    Thus for each coefficient there is a corresponding representation of $G$  on a vector space which is a direct sum of irreducible representations of $G$. Restricting to the action of $\mathcal H$ decomposes each summand further into a finite number of irreducible representations of $\mathcal H$.
    \end{proof}
    
    \begin{theorem}
    If $\Hstar_G(P;z)$ is polynomial, then $\Hstar_{\mathcal H}(P;z)$ is polynomial.
    \end{theorem}
    \begin{proof}
    Let $\{\chi_1, \dots, \chi _m\}$ be the irreducible representations of $G$, and let $\{\mu_1, \dots, \mu_k\}$ be the irreducible representations of $\mathcal H$. 
    Suppose $\Hstar_G(P;z)$ is polynomial so that 
    $$
    \Hstar_G(P;z) = \sum_{i=0}^{d} \left (\sum_{j=1}^{m} c_{i,j}\chi_j \right)z^i = \sum_{j=1}^{m}  \left ( \sum_{i=0}^{d} c_{i,j}z^i\right)\chi_j, 
    $$
    where $c_{i,j} \in \Z$ for all $i \in \{ 0 , \dots, d\}$ and $j \in [m]$.
    Then the coefficient of $\mu_k$ in $\Hstar_{\mathcal H}(P;z)$ is 
    $$
    \sum_{j=0}^{m}(c_{0,j}z^0 + \dots + c_{d,j}z^d)\langle \chi_i |_{\mathcal H} ,  \mu_k \rangle_{\mathcal H},
    $$
    where $\langle \Phi ,\Psi \rangle _{\mathcal H} \in \mathbb C$ denotes the inner product between characters $\Phi, \Psi$ of $\mathcal H$, and $\chi_i |_{\mathcal H}$ is the representation restricted to $\mathcal H$.
    \end{proof}
        
    \begin{example}
    The cyclic group $\Z/n\Z \subseteq S_n$ acting on the permutahedron $\Pi_n$ gives an interesting example of the failure of the other direction. 
    Namely, $\Hstar_{S_n}(\Pi_n;z)$ is non-polynomial for $n \geq 4$ \cite{ASV}, but $\Hstar_{\Z/n\Z}(\Pi_n;z)$ is polynomial for all $n$ (see \Cref{lem:cyclicactionpipolynomial}).
    Here we show how the rational function collapses to a polynomial for the case $\mathcal H = \Z/4\Z \subseteq S_4$:
    We implemented functionality in \texttt{Sagemath} for the calculation of the $\Hstar$-series, and we use it to calculate $\Hstar(\Pi_4, ;z)$ below, with characters given in the following tables. See \Cref{appendix:sage} for details.  
    \begin{align*}
     \Hstar_{S_4}(\Pi_4;z)  = & 
    \frac{(\chi_2 + \chi_3 + \chi_4)z^4}{z + 1} + \frac {(\chi_0 + 5\chi_1 + 6\chi_2 + 9\chi_3 + 6\chi_4)z^3}{z + 1} + \frac{(\chi_0 + 7\chi_1 + 8\chi_2 + 14\chi_3 + 9\chi_4)z^2}{z + 1} \\
    & + \frac{(\chi_0 + 3\chi_1 + 3\chi_2 + 5\chi_3 + 4\chi_4)z}{z + 1} + \frac{\chi_4}{z + 1} 
    \end{align*}        
    $$
    \resizebox{\linewidth}{!}{
    \begin{tabular}{c | c c c c c}
    $S_4$ & () & (12) & (12)(34) & (123) & (1234) \\
    \hline \\
     $\chi_0$ & 1 &-1 & 1 & 1 &-1 \\
     $\chi_1$ & 3 &-1 &-1 & 0 & 1 \\
     $\chi_2$ & 2 & 0 & 2 &-1 & 0 \\
     $\chi_3$ & 3 & 1 &-1 & 0 &-1 \\
     $\chi_4$ & 1 & 1 & 1 & 1 & 1 \\
    \end{tabular}
    \,\,\,\,
    \begin{tabular}{c | c c c c c}
    $\mathcal H$   & () & (1234) & (12)(34) & (1432) \\
    \hline \\
     $\mu_0$   & 1  & 1      & 1        & 1      \\
     $\mu_1$   & 1  & -1     & 1        & -1     \\
     $\mu_2$   & 1  & i      & -1       & -i     \\
     $\mu_3$   & 1  & -i     & -1       & i      \\
     & & & & \\
    \end{tabular}
    }$$
    We compute $\langle \chi_i |_{\mathcal H} , \mu _j \rangle _{\mathcal H}$ for all pairs $i,j$ in order to 
    restrict $\Hstar_{S_4}(\Pi_4;z)$ to $\Hstar_{\mathcal H}(\Pi_4;z)$. A factor of $1+z$ appears in the numerator:
    
    \begin{align*}
    \Hstar_{\mathcal H}(\Pi_4;z) = & \frac{(2\mu_0+2\mu_1+\mu_2+\mu_3)z^4+(17\mu_0+16\mu_1+14\mu_2+14\mu_3)z^3 + (24\mu_0+23\mu_1+21\mu_2+21\mu_3)z^2}{1+z}\\
    & +\frac{(10\mu_0+9\mu_1+8\mu_2+8\mu_3)z+\mu_0}{1+z} \\
    = & 
     (2 \mu_{0} + 2 \mu_{1} + \mu_{2} + \mu_{3}) z^{3} + (15 \mu_{0} + 14 \mu_{1} + 13 \mu_{2} + 13 \mu_{3}) z^{2} + (9 \mu_{0} + 9 \mu_{1} + 8 \mu_{2} + 8 \mu_{3}) z + \mu_{0}.
    \end{align*}
    
     \end{example}

%% file: img/ehrhart_example.tikz
\begin{tikzpicture}%
	[scale=.7,
	edge/.style={thick},
	facet/.style={fill=mymagenta,fill opacity=0.300000}]
%
%
\draw[black] (0,3.5) -- (0,0) -- (3.5,0);

\draw[edge,color=mymagenta] (0,0) -- (1,0) -- (1,1) -- (0,1) -- (0,0);
\fill[facet] (0,0) -- (1,0) -- (1,1) -- (0,1) -- (0,0);
\node at (.5,.5) {\color{mymagenta}$P$};
\draw[edge,color=mycyan] (1,0) -- (2,0) -- (2,2) -- (0,2) -- (0,1);
\fill[color=mycyan, fill opacity=.3] (1,0) -- (2,0) -- (2,2) -- (0,2) -- (0,1) --(1,1) -- (1,0);
\node at (1.5,1.5) {\color{mycyan} $2P$};
\draw[edge,color=myyellow, fill opacity=.3] (2,0) -- (3,0) -- (3,3) -- (0,3) -- (0,2);
\fill[color=myyellow, fill opacity=.3](2,0) -- (3,0) -- (3,3) -- (0,3) -- (0,2) -- (2,2) -- (2,0);
\node at (2.5,2.5) {\color{myyellow}$3P$};

\foreach \i in {0,...,3}{
    \foreach \j in {0,...,3}{
        \draw[black, fill=black] (\i,\j) circle (2pt);
    }
}
\end{tikzpicture}

%% file: img/cone_example.tikz
\begin{tikzpicture}[scale=.7,
	edge/.style={thick},
	facet/.style={fill opacity=0.300000},
	rotate around x = -94, 
	rotate around y=0,
	rotate around z=-25]

%

\draw[edge, color=black] (0,0,0) -- (0,3.5,3.5);

\draw[edge,color=mymagenta] (0,0,1)--(0,1,1)--(1,1,1)--(1,0,1)--cycle;
\fill[facet,fill=mymagenta] (0,0,1)--(0,1,1)--(1,1,1)--(1,0,1)--cycle;

\draw[edge, color=mycyan] (0,0,2)--(0,2,2)--(2,2,2)--(2,0,2)--cycle;
\fill[facet, fill=mycyan] (0,0,2)--(0,2,2)--(2,2,2)--(2,0,2)--cycle;

\draw[edge, color=myyellow] (0,0,3)--(0,3,3)--(3,3,3)--(3,0,3)--cycle;
\fill[facet, fill=myyellow] (0,0,3)--(0,3,3)--(3,3,3)--(3,0,3)--cycle;

\draw[edge, color=black] (0,0,0) -- (0,0,3.5);
\draw[edge, color=black] (0,0,0) -- (3.5,0,3.5);
\draw[edge, color=black] (0,0,0) -- (3.5,3.5,3.5);

\foreach \p in { (0,0,0),(0,0,1), (0,0,2),(0,0,3),(1,0,1),(1,0,2),(1,0,3),(0,1,1),(0,1,2),(0,1,3),(1,1,1),(1,1,2),(1,1,3),(0,2,2),(0,2,3),(2,0,2),(2,0,3),(2,2,2),(2,2,3),(2,1,2),(2,1,3),(1,2,2),(1,2,3),(0,3,3),(1,3,3),(2,3,3),(3,3,3),(3,2,3),(3,1,3),(3,0,3)}{
	\draw[black,fill=black] \p circle (2pt);
}

\node at (-.5,-.5,1) {\color{mymagenta}$P$};
\node at (-.5,-.5,2) {\color{mycyan}$2P$};
\node at (-.5,-.5,3) {\color{myyellow}$3P$};

\end{tikzpicture}

%% file: 3_techniques.tex
\section{Techniques}\label{sec:techniques}

In this section we introduce our main results, exploring four approaches toward equivariant Ehrhart theory and applying these methods to a variety of examples.

%% file: 3.1_zonotopal_decompositions.tex
\subsection{Zonotopal Decompositions}\label{sec:zonotopes} 
    
    Zonotopes are a family of polytopes which can be decomposed into parallelotopes, and from these decompositions their Ehrhart polynomials are easily computed.
    This approach also lends itself nicely to equivariant Ehrhart theory, and has previously been used to study the equivariant Ehrhart theory of the permutahedron under the action of the symmetric group \cite{ArdilaSchindlerVindas,ASV}.
    Here we adapt this technique to all zonotopes of the Type A root system, also known as graphic zonotopes.
    We expect this to generalize further to other root systems (see \cite{ArdilaBeckMcWhirter} for progress in this direction), as well as to general families of zonotopes.
    We begin with a brief overview of zonotopes; for a more in-depth introduction to the subject, see \cite[Chapter 9]{beck_robins}.
    
    Let $S$ be a finite set of vectors.
    The \emph{zonotope} of $S$ is denoted $Z(S)$ and is the Minkowski sum of the line segments connecting $\bzero$ to $\bs$ for each $\bs\in S$, or any translation of this polytope.
    It was shown by Shephard that any lattice zonotope $Z(S)$ can be partitioned into half-open parallelotopes that are in bijection with the linearly independent subsets of $S$ \cite[Theorem 54]{Shephard}.
    For a linearly independent subset $T\subseteq S$, the corresponding half-open parallelotope $\hobox{T}$ is, up to translation, the zonotope $Z(T)$ where each line segment in the Minkowski sum is open at the endpoint $\bzero$.
    Half-open lattice parallelotopes have two nice properties: their volumes can be easily computed from matrix minors, and the number of lattice points they contain is exactly equal to their volume.
    Stanley used this to give a formula for the Ehrhart polynomial of a zonotope.
    
    \begin{theorem}[{\cite[Theorem~2.2]{Stanleyzonotope}}]\label{thm:zonotope_ehrhart}
    The Ehrhart polynomial of the zonotope $Z(S)$ is
        \[
        \L\big(Z(S);t\big) = \sum_{\substack{T\subseteq S\\\text{lin. indep.}}} \vol{\hobox{T}}\cdot t^{|T|}.
        \]
    \end{theorem}
    
    Let $\Gamma=(V,E)$ be a simple graph, where $V$ is a finite set of vertices, and $E$ is the set of undirected edges.
    We write $E\subseteq\binom{V}{2}$ where $\binom{V}{2}$ is the collection of all $2$-element subsets of $V$, and we denote by $\R V$ a real vector space with a basis $\{\bev:v\in V\}$ indexed by the elements of $V$.
    The \emph{graphic zonotope} $Z_\Gamma$ is given by taking the Minkowski sum of line segments in $\R V$
        \[
        Z_\Gamma:= \sum_{\{u,v\}\in E} [\beu,\bev].
        \]
    The dimension of $Z_\Gamma$ is $|V|$ minus the number of connected components of $\Gamma$.
    
    \begin{example}[Path graph]\label{ex:path graph}
    When $V=[n]$ and $E=\{\{1,2\},\{2,3\},\dots,\{n-1,n\}\}$, we call the graph $(V,E)$ the \emph{path graph} or $\path_n$.
    The graphic zonotope $Z_{\path_n}$ is the $(n-1)$-dimensional parallelotope $\sum_{i=1}^{n-1}[\bei,\beii]$; \Cref{fig:path_graph_zonotope} shows $Z_{\path_4}$.
    \end{example}
    
    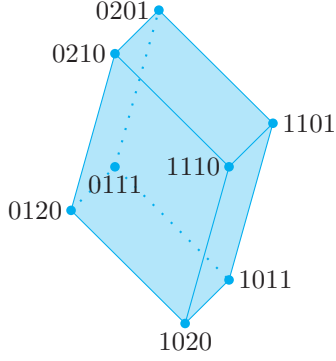
\begin{figure}
        \centering
        \input{img/path_graph_zonotope.tikz}
        \caption{The graphic zonotope $Z_{\path_4}$ is a parallelopiped.}
        \label{fig:path_graph_zonotope}
    \end{figure}
    
    \begin{example}[Complete graph]
    The graph with $V=[n]$ and $E=\binom{[n]}{2}$ is called the \emph{complete graph}.
    The corresponding graphic zonotope is $\sum_{1\le i<j\le n}[\bei,\bej]$, also known as the \emph{permutahedron}.
    \end{example}
    
    The following characterizations of the graphic zonotope are well-known; see \cite[Section~2]{Stanley_hyperplane} and \cite[Section~13]{AguiarArdila} for more information.
    
    \begin{proposition}\label{prop:graphic zonotope desc}
    The graphic zonotope $Z_\Gamma$ has....
        \begin{itemize}
            \item ...the half-space description 
                \[Z_\Gamma = \left\{ \bx\in\R V : \sum_{v\in V} x_v = |E|, \; \forall S\subset V \, \sum_{v\in S} x_v \le \inc_\Gamma(S) \right\},\]
            where for any subset $S\subseteq V$, we define $\inc_\Gamma(S)$ to be the number of edges of $\Gamma$ that have an endpoint in $S$ (are ``incident to $S$'').
            \item ...the vertex description
                \[\conv\left\{\bv_\ao:=\sum_{v\in V} \indeg_{\mathfrak{o}}(v)\bev : \mathfrak{o} \text{ is an acyclic orientation of }\Gamma\right\},\]
            where $\indeg_{\mathfrak{o}}(v)$ is the number of edges pointing into the vertex $v$ in the acyclic orientation $\mathfrak{o}$ of $\Gamma$.
            Furthermore, $\bv_\ao$ is a vertex of $Z_\Gamma$ for all acyclic orientations $\ao$ of $\Gamma$.
        \end{itemize}
    \end{proposition}
    
    The symmetric group $\Sym(V)$ acts on the set of vertices $V$, also inducing an action on $\binom{V}{2}$.
    The \emph{automorphism group} $\Aut(\Gamma)$ of $\Gamma$ is the subgroup of $\Sym(V)$ under which the edge set $E$ is invariant.
    Just as $\Aut(\Gamma)$ acts on $V$ and holds $\Gamma$ invariant, likewise $\Aut(\Gamma)$ acts linearly on $\R V$ and holds $Z_\Gamma$ invariant.
    Hence $\R V$ is a representation of $\Aut(\Gamma)$, and we can consider the equivariant Ehrhart theory of $Z_\Gamma$ under the action of $\Aut(\Gamma)$ (or any subgroup thereof).
    
    For the setup, we will view $\Aut(\Gamma)$ as a subgroup of $\Sym(V)$, so every element $\sigma\in\Aut(\Gamma)$ can be written as a product of $m\le |V|$ disjoint cycles $\sigma_1,\dots,\sigma_m$ partitioning $V$.
    We denote the cycle lengths by $\ell_1,\dots,\ell_m$, respectively.
    For each $\sigma\in\Aut(\Gamma)$, we define the \emph{$\sigma$-connectivity graph} $C_\Gamma(\sigma)$ to be the simple graph with vertex set $\{\sigma_1,\dots,\sigma_m\}$ and an edge between two distinct cycles $\sigma_i$ and $\sigma_j$ whenever some element of the cycle $\sigma_i$ is connected by an edge in $\Gamma$ to some element of $\sigma_j$.

    \begin{example}
    Let $\Gamma=\path_4$, the path graph with vertex set $[4]$, and let $\sigma=(14)(23)$.
    Then $C_\Gamma(\sigma)$ has vertex set $\{14,23\}$ and an edge connecting $14$ to $23$, as seen in \Cref{fig:connectivity_graph}.
    \end{example}
    
    \begin{figure}[h]
        \centering
            \begin{subfigure}{.4\linewidth}
            \centering
            \input{img/path_graph.tikz}
            \caption{}
            \label{fig:path4}
            \end{subfigure}
            \hspace{1mm}
            \begin{subfigure}{.4\linewidth}
            \centering
            \input{img/path_graph_connectivity.tikz}
            \caption{}
            \label{fig:pathconnectivity}
            \end{subfigure}
        \caption{The orbits of the vertices of $\path_4$ under the automorphism $(14)(23)$ (\ref{fig:path4}), and the $(14)(23)$-connectivity graph of $\path_4$ (\ref{fig:pathconnectivity}).}
        \label{fig:connectivity_graph}
    \end{figure}
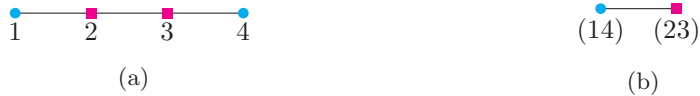
    
    \begin{lemma}\label{prop:degree}
    Let $\sigma\in \Aut(\Gamma)$ and let $\sigma_i,\sigma_j$ be (not necessarily distinct) cycles of $\sigma$.
    Then for all $v\in\sigma_i$, the number of edges connecting $v$ to the elements of $\sigma_j$ is constant.
    \end{lemma}
    
    \begin{proof}
    Suppose that $u,v\in \sigma_i$ have different numbers of edges connecting them to $\sigma_j$.
    Since $u$ and $v$ are in the same cycle of $\sigma$, there exists $r$ such that $\sigma^r(u)=v$.
    But $\sigma^r$ is then not an automorphism of $\Gamma$, a contradiction.
    \end{proof}
    
    We denote the number described in \Cref{prop:degree} by $\deg(\sigma_i,\sigma_j)$.
    (Warning: It is NOT true in general that $\deg(\sigma_i,\sigma_j)=\deg(\sigma_j,\sigma_i)$; see \Cref{ex: sigma degree}.)
    We will repeatedly make use of the following identity for distinct cycles $\sigma_i$ and $\sigma_j$:
        \begin{equation}\label{eq: edges}
            E_{ij}:= \ell_i\deg(\sigma_i,\sigma_j) = \ell_j\deg(\sigma_j,\sigma_i) = \text{\# edges connecting }\sigma_i \text{ and } \sigma_j \text{ in } \Gamma.
        \end{equation}
    We also write the number of edges within a given cycle $\sigma_i$ as $E_{ii}$ and note that $E_{ii} = \frac{1}{2}\ell_i\deg(\sigma_i,\sigma_i)$.
    Observe also that the notation $E_{ij}$ \emph{is} symmetric in $i$ and $j$.
        
    \begin{example}\label{ex: sigma degree}
    Let $\Gamma$ be the complete graph on the vertex set $[3]$ and let $\sigma = (12)(3)\in\Aut(\Gamma)$.
    Then $\deg((12),(3)) = 1$ since the vertices $1$ and $2$ are each connected to $3$ by one edge, but $\deg((3),(12)) = 2$ since the vertex $3$ has two edges connecting it to vertices in the cycle $(12)$.
    Moreover, $\deg((12),(12)) = 1$ and $\deg((3),(3)) = 0$.
    \end{example}
    
    We use $\besi$ to denote the vector $\sum_{v\in\sigma_i}\bev\in\R V$.
    \Cref{thm:gzfixedpoly} and its proof closely follow the description of the fixed polytopes of the permutahedron from Theorem 2.12 of \cite{ArdilaSchindlerVindas}.
    Our result generalizes theirs: where they used permutations, we use acyclic orientations of $\Gamma$, and we also make use of identity \eqref{eq: edges}.
    
    \begin{theorem}[Description of graphic zonotope fixed polytopes]\label{thm:gzfixedpoly}
    Let $\sigma\in\Aut(\Gamma)$.
    Then the following descriptions of the fixed polytope $Z_\Gamma^\sigma$ are equivalent:
        \begin{enumerate}[(i)]
            \item\label{item:fixed polytope} It is the subset of $Z_\Gamma$ that is fixed by the action of $\sigma$.
            
            \item\label{item:vertex description} It is the convex hull of the points $\bw_\ao$, where $\ao$ ranges over acyclic orientations of $C_\Gamma(\sigma)$, and 
                \[
                \bw_\ao:=\sum_{i=1}^m\left( \frac{1}{2}\deg(\sigma_i,\sigma_i) + \sum_{\substack{j:\sigma_j\to\sigma_i\\\text{in }\ao}} \deg(\sigma_i,\sigma_j) \right)\besi.
                \]
            Furthermore, this gives a bijection between acyclic orientations of $C_\Gamma(\sigma)$ and vertices of $Z_\Gamma^\sigma$.
            
            \item\label{item:zonotope description} It is the Minkowski sum 
            \begin{equation}\label{eq:gzmink}
                \sum_{\substack{ \{\sigma_i,\sigma_j\} \text{ is an} \\ \text{edge of } C_\Gamma(\sigma)\\ i<j}} [\deg(\sigma_j,\sigma_i)\besj,\, \deg(\sigma_i,\sigma_j)\besi] + \sum_{j=1}^m \frac{1}{2}\deg(\sigma_j,\sigma_j)\besj.
            \end{equation}
        \end{enumerate}
    \end{theorem}
    \begin{proof}
    \eqref{item:vertex description} $\subseteq$ \eqref{item:fixed polytope}:
    Let $\hat{\ao}$ be any orientation of $\Gamma$ ``extending'' $\ao$, meaning that if $\sigma_i\to\sigma_j$ in $\ao$, and $v\in\sigma_i$ and $v'\in\sigma_j$ are joined by an edge of $\Gamma$, then $v\to v'$ in $\hat{\ao}$.
    Then $\bw_\ao$ is the average of the $\sigma$-orbit of the vertex of $Z_\Gamma$ corresponding to $\hat{\ao}$, so $\bw_\ao\in Z_\Gamma$.
    By definition, the coordinates of $\bw_\ao$ are constant over the cycles of $\sigma$, so $\bw_\ao\in Z_\Gamma^\sigma$ for all acyclic orientations $\ao$.\\
    
    \eqref{item:zonotope description} $\subseteq$ \eqref{item:vertex description}:
    Let $\bw$ be a vertex of \eqref{eq:gzmink} and let $\by=(y_1,\dots,y_n)\in(\R^n)^\ast$ be such that $\bw$ is the $\by$-maximal face of \eqref{eq:gzmink}.
    For $1\le i\le m$, define $y_{\sigma_i}$ to be the average of $y_k$ for $k\in\sigma_i$; so $y_{\sigma_i} = \frac{1}{\ell_i}\sum_{k\in\sigma_i}y_k$.
    Now, since $\by$ is maximal at a vertex, we know that for each edge $\{\sigma_i, \sigma_j\}$ of $C_\Gamma(\sigma)$, we must have 
        \[\by\left(\deg(\sigma_j,\sigma_i)\besj\right)\neq \by\left(\deg(\sigma_i,\sigma_j)\besi\right).\]
    Therefore it holds that
        \begin{align*}
            \deg(\sigma_j,\sigma_i)\sum_{k\in\sigma_j}y_k &\neq \deg(\sigma_i,\sigma_j)\sum_{k\in\sigma_i}y_k\\
            \ell_j\deg(\sigma_j,\sigma_i)y_{\sigma_j} &\neq \ell_i \deg(\sigma_i,\sigma_j) y_{\sigma_i}
        \end{align*}
    Using identity \eqref{eq: edges}, we find that $y_{\sigma_j}\neq y_{\sigma_i}$.
    For each edge $\{\sigma_i,\sigma_j\}$ in $C_\Gamma(\sigma)$, if $y_{\sigma_i}<y_{\sigma_j}$, assign the direction pointing from $\sigma_i$ to $\sigma_j$.
    This gives an acyclic orientation $\ao$ of $C_\Gamma(\sigma)$, and $\bw=\bw_\ao$.\\
    
    \eqref{item:fixed polytope} $\subseteq$ \eqref{item:zonotope description}:
    According to \cite[Lemma 5.4]{Stapledon}, the fixed polytope $Z_\Gamma^\sigma$ is the convex hull of the averages over the $\sigma$-orbits of the vertices of $Z_\Gamma$.
    So we need only show that for any vertex $\bv_\ao$ of $Z_\Gamma$ corresponding to the acyclic orientation $\ao$ of $\Gamma$,  the average $\overline{\bv_\ao}$ of the $\sigma$-orbit of $\bv_\ao$ is in the Minkowski sum \eqref{eq:gzmink}.
    To do this, we will fix an orientation $\ao'$ of $\Gamma$, show that $\overline{\bv_{\ao'}}$ is in the Minkowski sum, and then show that we can find a path from $\overline{\bv_{\ao'}}$ to $\overline{\bv_{\ao}}$ within the sum.
    
    We construct $\ao'$ as follows:
    Choose any ordering of the $n$ vertices of $\Gamma$ such that all the vertices in the cycle $\sigma_i$ are less than all the vertices in $\sigma_{i+1}$ for $1\le i<m$, and let $\ao'$ be the acyclic orientation of $\Gamma$ induced by this ordering of the vertices.
    Then the vertex $\bv_{\ao'}$ of $Z_\Gamma$ is defined as in \Cref{prop:graphic zonotope desc}, and the average of the $\sigma$-orbit of $\bv_{\ao'}$ is
        \[
        \overline{\bv_{\ao'}} = \sum_{\substack{\{\sigma_i,\sigma_j\} \text{ is an}\\ \text{edge of } C_\Gamma(\sigma)\\i<j}} \deg(\sigma_j,\sigma_i)\be_{\sigma_j} + \sum_{j=1}^m \frac{1}{2}\deg(\sigma_j,\sigma_j)\be_{\sigma_j},
        \]
    which is clearly seen to be in the sum \eqref{eq:gzmink}.
    
    Now consider any acyclic orientation $\ao$ of $\Gamma$.
    Let $\ao_0:=\ao',\ao_1,\ldots,\ao_\ell:=\ao$ be a sequence of acyclic orientations of $\Gamma$ where $\ao_k$ is obtained from $\ao_{k-1}$ by reversing the orientation of one edge for $1\le k\le \ell$.
    For each $k$, if the edge of $\ao_{k-1}$ that is reversed to obtain $\ao_k$ points from a vertex in $\sigma_i$ to a vertex in $\sigma_j$, then we have that
        \[
        \overline{\bv_{\ao_k}} - \overline{\bv_{\ao_{k-1}}} =
        \frac{1}{\ell_i}\be_{\sigma_i} - \frac{1}{\ell_j}\be_{\sigma_j} =
        \frac{1}{E_{ij}}\left(\deg(\sigma_i,\sigma_j)\be_{\sigma_i} - \deg(\sigma_j,\sigma_i)
        \be_{\sigma_j}\right).
        \]
    Note that if the edge that is reversed connects two vertices of $\Gamma$ that are in the same cycle of $\sigma$, then there is no change in the $\sigma$-average.
    Now for each pair of distinct cycles $\sigma_i$ and $\sigma_j$, let $D_{ij}$ be the number of edges between $\sigma_i$ and $\sigma_j$ that have different orientations in $\ao'$ and $\ao$.
    Then we get
        \begin{align*}
        \overline{\bv_{\ao}} - \overline{\bv_{\ao'}} &=
        \left(\overline{\bv_{\ao_\ell}} - \overline{\bv_{\ao_{\ell-1}}}\right) + \cdots + \left(\overline{\bv_{\ao_2}} - \overline{\bv_{\ao_1}} \right) + \left( \overline{\bv_{\ao_1}} - \overline{\bv_{\ao_0}} \right)\\
        &= \sum_{1\le i<j\le m} \frac{D_{ij}}{E_{ij}}\left( \deg(\sigma_i,\sigma_j)\be_{\sigma_i} - \deg(\sigma_j,\sigma_i)\be_{\sigma_j} \right)
        \end{align*}
    If there are no edges between $\sigma_i$ and $\sigma_j$ in $\Gamma$, then $D_{ij}$ must be $0$, so $\deg(\sigma_i,\sigma_j)\be_{\sigma_i} - \deg(\sigma_j,\sigma_i)\be_{\sigma_j}$ may only contribute to this sum if there is an edge between $\sigma_i$ and $\sigma_j$ in $C_\Gamma(\sigma)$.
    Furthermore, the maximum possible number of edges between $\sigma_i$ and $\sigma_j$ is $E_{ij}$, so $0\le \frac{D_{ij}}{E_{ij}}\le 1$.
    Hence adding our earlier expression for $\overline{\bv_{\ao'}}$ produces a point in the zonotope \eqref{eq:gzmink}.
    
    Finally, it remains to prove that the acyclic orientations of $C_\Gamma(\sigma)$ are in bijection with the vertices of $Z_\Gamma^\sigma$.
    By definition, each acyclic orientation $\ao$ results in a distinct point $\bw_\ao$.
    We already showed that every vertex of $Z_\Gamma^\sigma$ has the form $\bw_\ao$ for some acyclic orientation $\ao$.
    For any $\bw_\ao$, let $\by = \sum_{i=1}^m y_i ( \sum_{k\in\sigma_i}\be_k^\ast)$ be some linear functional in $(\R^n)^\ast$ such that $y_i<y_j$ whenever $\sigma_i\to\sigma_j$ in $\ao$.
    Then $\by(\deg(\sigma_j,\sigma_i)\be_{\sigma_j}) = E_{ij}y_j $, so we can see that exactly one endpoint of each of the line segments comprising $\eqref{eq:gzmink}$ is maximized by $\by$, and the resulting sum gives exactly $\bw_\ao$.
    So $\bw_\ao$ is the unique point in $Z_\Gamma^\sigma$ maximizing $\by$, and is hence a vertex.
    
    \end{proof}
    
    
    
    
    In general, a zonotope generated by line segments with direction vectors $\bs_1,\dots,\bs_k$ can be tiled by half-open parallelotopes that are in bijection with linearly independent subsets of $\{\bs_1,\dots,\bs_k\}$.
    In the case of $Z_\Gamma^\sigma$, the generating vectors are 
        \begin{equation}\label{eq: generators}
       \left \{\deg(\sigma_i,\sigma_j)\be_{\sigma_i}-\deg(\sigma_j,\sigma_i)\be_{\sigma_j}: \{\sigma_i,\sigma_j\} \text{ is an edge of } C_\Gamma(\sigma),\, i<j \right \}.
        \end{equation}
    We define a \emph{subforest} of $C_\Gamma(\sigma)$ to be any graph on the same vertex set $\{\sigma_1,\dots,\sigma_m\}$ as $C_\Gamma(\sigma)$ whose edge set is a subset of the edges of $C_\Gamma(\sigma)$, and that does not contain any cycles.
        
    \begin{proposition}
    The linearly independent subsets of the vectors in (\ref{eq: generators}) are in bijection with subforests of the $\sigma$-connectivity graph $C_\Gamma(\sigma)$.
    \end{proposition}
    
    \begin{proof}
    The bijection is as follows:
    Given a subgraph of $C_\Gamma(\sigma)$, for each edge $\{i,j\}$ where $i<j$, take the vector $\deg(\sigma_i,\sigma_j)\be_{\sigma_i} - \deg(\sigma_j,\sigma_i)\be_{\sigma_j}$ to be in the subset; similarly, given a subset of the generating set, take the corresponding edges of $C_\Gamma(\sigma)$.
    It remains to show that linearly independent subsets of the generators correspond exactly to subforests of $C_\Gamma(\sigma)$.
    By identity \eqref{eq: edges}, the generating vector corresponding to the edge $\{i,j\}$ is simply a positive multiple of the vector $\frac{\be_{\sigma_i}}{\ell_i} - \frac{\be_{\sigma_j}}{\ell_j}$.
    From here we can follow exactly the proof of \cite[Lemma 3.2]{ArdilaSchindlerVindas}, noting that our generators are a subset of theirs and therefore inherit the same linear independence properties.
    \end{proof}
    
    For a given subforest $F$ of $C_\Gamma(\sigma)$, we denote by $\hobox{F}$ the corresponding half-open parallelotope in the tiling of $Z_\Gamma^\sigma$. 
    
    \begin{proposition}
    The volume of $\hobox{F}$ is given by the formula
        \begin{equation}\label{eq:volume F}
        \vol{\hobox{F}} := \left(\prod_{\substack{\text{edge } \{\sigma_i,\sigma_j\}\\\text{ of } F} } E_{ij} \right)\cdot
        \left( \prod_{i=1}^m \frac{1}{\ell_i} \right)\cdot
        \left( \prod_{\substack{\text{conn. comp.}\\T \text{ of } F}} \gcd(\ell_i:\sigma_i\in T) \right).
        \end{equation}
    \end{proposition}
    
    \begin{proof}
    This follows directly from \cite[Lemma 3.3]{ArdilaSchindlerVindas}.
    \end{proof}
    
    The \emph{$2$-valuation} of a positive integer k, denoted $\val(k)$, is the highest power of $2$ dividing $k$.
    For example, $\val(24)=3$.
    
    \begin{definition}\label{def:compatibility}
    A subforest $F$ of $C_\Gamma(\sigma)$ is \emph{$(\sigma,\Gamma)$-compatible} if for all connected components $T$ of $F$, we have
        \begin{equation}\label{eq:compatibility}
        \min_{i\in T} \val(\ell_i) \le \val\left( \sum_{j\in T} E_{jj} \right).
        \end{equation}
    \end{definition}
    
    \begin{proposition}
    Let $\sigma\in\Aut(\Gamma)$ have cycle type $(\ell_1,\dots,\ell_m)$, and let $F$ be a subforest of $C_\Gamma(\sigma)$.
    Then $\aff(\hobox{F})$ contains lattice points if and only if $F$ is $(\sigma,\Gamma)$-compatible.
    \end{proposition}
    
    \begin{proof}
    The affine span of $\hobox{F}$ is the span of the vectors $\deg(\sigma_i,\sigma_j)\be_{\sigma_i}-\deg(\sigma_j,\sigma_i)\be_{\sigma_j}$ corresponding to to the edges $\{\sigma_i,\sigma_j\}$ of $F$, translated by the vector $$\sum_{j=1}^m\frac{1}{2}\deg(\sigma_j,\sigma_j)\be_{\sigma_j} + \sum_{\substack{\text{edge }\{\sigma_i,\sigma_j\}\\\text{of } F,\,i<j}} \deg(\sigma_j,\sigma_i)\be_{\sigma_j}.$$
    Let $\by=\sum_{i=1}^m y_i\be_{\sigma_i}$ be a point in $\aff(\hobox{F})$.
    Then the coordinates $y_i$ satisfy, for all connected components~$T$ of $F$,
        \begin{equation}\label{eq:affspan}
        \sum_{\sigma_j\in T} \ell_jy_j
        = \sum_{\sigma_j\in T} E_{jj} + \sum_{\substack{\text{edge }\{\sigma_i,\sigma_j\}\\\text{of }T}} E_{ij}.
        \end{equation}
    We want to determine when \eqref{eq:affspan} has integer solutions $\by$.
    It is a fact from elementary number theory that this occurs exactly when $\gcd(\ell_j:\sigma_j\in T)$ divides the right side of \eqref{eq:affspan}.
    Clearly each $E_{ij}$ is divisible by $\gcd(\ell_j:\sigma_j\in T)$, so we may focus our attention on the term $\sum_{\sigma_j\in T}E_{jj}$.
    If we multiply this term by~$2$, then we get $\sum_{\sigma_j\in T} \ell_j\deg(\sigma_j,\sigma_j)$ which is clearly divisible by $\gcd(\ell_j:\sigma_j\in T)$.
    Hence we need only consider $2$-valuations to determine whether divisibility holds.
    Finally, noticing that ${\val(\gcd(\ell_j:\sigma_j\in T)) = \min_{\sigma_j\in T} \val(\ell_j)}$ allows us to conclude that the condition of  $(\sigma,\Gamma)$-compatibility introduced in \Cref{def:compatibility} is exactly what we need to guarantee integer solutions.
    \end{proof}
    
    Let $c(F)$ be the number of connected components of a graph $F$, and let $m$ be the number of cycles of~$\sigma\in\Aut(\Gamma)$.
    Note that for a subforest $F$ of $C_\Gamma(\sigma)$, the number of edges of $F$ is $m-c(F)$.
    We omit proofs for the following two corollaries as they mirror the proofs of Theorem~1.1 and Proposition 5.1 in \cite{ASV}.
    
    \begin{corollary}\label{cor:gz quasipolynomial}
    The Ehrhart quasipolynomial of the fixed polytope $Z_\Gamma^\sigma$ is
        \[
        \L(Z_\Gamma^\sigma;t) = \begin{cases}
        \displaystyle{\sum_{\substack{\text{subforest}\\ F \text{ of } C_\Gamma(\sigma)}}} \vol{\hobox{F}}\cdot t^{m-c(F)}, & t \text{ even}\\
        \displaystyle{\sum_{\substack{(\sigma,\Gamma)\text{-compatible}\\ \text{subforest}\\ F \text{ of } C_\Gamma(\sigma)}}} \vol{\hobox{F}}\cdot t^{m-c(F)}, & t \text{ odd}.
        \end{cases}
        \]
    \end{corollary}
    
    \emph{Eulerian polynomials} commonly arise when computing the Ehrhart series of polytopes.
    The Eulerian polynomial $A_n(z)$ is defined by the identity
        \[
        \sum_{t\ge 0}t^n z^t = \frac{A_n(z)}{(1-z)^{n+1}}.
        \]
    Notice that $A_n(z)$ has no constant term; the coefficients of $\frac{1}{z}A_n(z)$ are called the \emph{Eulerian numbers}.
    We can describe them using ascents of permutations, where the position $i\in\{1,2,\dots,n-1\}$ is an \emph{ascent} of $\sigma\in S_n$ if $\sigma(i+1)>\sigma(i)$.
    The coefficient of $z^k$ in $\frac{1}{z}A_n(z)$ is the number of permutations in $S_n$ with exactly $k$ ascents, denoted $A_{n,k}$.
    
    \begin{corollary}\label{cor:graphic zonotope ehrhart series}
    The Ehrhart series of the fixed polytope $Z_\Gamma^\sigma$ has the following form:
        \[
        \Ehr(Z_\Gamma^\sigma;z) = \sum_{\substack{(\sigma,\Gamma)\text{-compatible}\\\text{subforests}\\F \text{ of } C_\Gamma(\sigma)}}\frac{\vol{\hobox{F}}\cdot A_{m-c(F)}(z)}{(1-z)^{m-c(F)+1}} + \sum_{\substack{(\sigma,\Gamma)\text{-incompatible}\\\text{subforests}\\F \text{ of } C_\Gamma(\sigma)}} \frac{2^{m-c(F)}\cdot \vol{\hobox{F}}\cdot A_{m-c(F)}(z^2)}{(1-z^2)^{m-c(F)+1}}
        \]
    \end{corollary}
    
    To assess whether $\Hstar(Z_\Gamma;z)$ is a polynomial, we need to examine for which $\sigma\in\Aut(\Gamma)$ the poles of $\Ehr(Z_\Gamma^\sigma;z)$ cancel with the zeros of $\det(I-\rho(\sigma)z) = \prod_{i=1}^m(1-z^{\ell_i})$, where $m$ is the number of cycles of $\sigma$ and $\ell_1,\dots,\ell_m$ are the cycle lengths.
    The series $\Ehr(Z_\Gamma^\sigma;z)$ always has a pole at $z=1$, may have a pole at $z=-1$, and never has a pole at any other value of $z$.
    We can see from \Cref{cor:graphic zonotope ehrhart series} that the pole at $z=1$ has order at most $m$, and since $\prod_{i=1}^m(1-z^{\ell_i})$ has a zero of order $m$ at $z=1$, the pole at $z=1$ will always cancel.
    Hence we may focus on the pole at $z=-1$.
    The following theorem tells us when this pole cancels.
    
    \begin{theorem}\label{thm:graphic zonotope polynomial H*}
    Let $\sigma\in\Aut(\Gamma)$ have cycle type $\ell_1,\dots,\ell_m$.
    The series $\Hstar(Z_\Gamma;z)(\sigma)$ is a polynomial if and only if one of the following conditions holds:
        \begin{enumerate}[(a)]
            \item\label{item:all odd} All $\ell_i$ are odd.
            \item\label{item:even even} All even $\ell_i$ have $\deg(\sigma_i,\sigma_i)$ even.
            \item\label{item:even odd} There exists an even $\ell_i$ with $\deg(\sigma_i,\sigma_i)$ odd, and
                \[
                \text{\# even }\ell_i > \max\{\text{\# of edges in a }(\sigma,\Gamma)\text{-incompatible subforest of }C_\Gamma(\sigma)\}.
                \]
        \end{enumerate}
    \end{theorem}
    
    \begin{proof}
    In case \eqref{item:all odd}, all subforests of $C_\Gamma(\sigma)$ are $(\sigma,\Gamma)$-compatible since $\val(\ell_i)=0$ for all $i$.
    So the Ehrhart series of $Z_\Gamma^\sigma$ does not have a pole at $z=-1$.
    Likewise in case \eqref{item:even even}; here we can also see that the condition of $(\sigma,\Gamma)$-compatibility from \Cref{def:compatibility} is always satisfied.
    
    Now assume \eqref{item:all odd} and \eqref{item:even even} both fail.
    Then we must have some cycle $\sigma_i$ with even length $\ell_i$ where $\deg(\sigma_i,\sigma_i)$ is odd.
    In this case, we are guaranteed to have at least one $(\sigma,\Gamma)$-incompatible subforest $F$ of $C_\Gamma(\sigma)$, since every forest where $\sigma_i$ is alone in a connected component will be $(\sigma,\Gamma)$-incompatible.
    From \Cref{cor:graphic zonotope ehrhart series} we see that the pole of $\Ehr(Z_\Gamma^\sigma;z)$ at $z=-1$ has order $\max_{(\sigma,\Gamma)\text{-incomp. }F}\{m-c(F)+1\}$, which is one more than the maximum number of edges in a $(\sigma,\Gamma)$-incompatible subforest of $C_\Gamma(\sigma)$.
    Since $\prod_{i=1}^m(1-z^{\ell_i})$ has a zero at $z=-1$ of order equal to the number of even $\ell_i$, condition \eqref{item:even odd} is exactly what is needed for cancellation to occur.
    \end{proof}
    
    \begin{corollary}
    $H^*(Z_\Gamma;z)$ is a polynomial if and only if every $\sigma\in\Aut(\Gamma)$ satisfies one of the conditions of \Cref{thm:graphic zonotope polynomial H*}.
    \end{corollary}
    
    We conclude this section by focusing in on a particularly simple family of graphs.   
    Consider the path graph $\path_n$ as introduced in \Cref{ex:path graph}.
    Then $\Aut(\path_n)=\Z/2\Z$ for all $n$, but the cycle type of the group generator depends on whether $n$ is even or odd.
    When $n$ is even, the generator of $\Aut(\path_n)$ written in cycle notation is $\sigma=(1\; n)(2\; n-1)\dots(\frac{n}{2}\;\frac{n}{2}+1)$, and when $n$ is odd, the generator is $\sigma=(1\;n)(2\; n-1)\dots (\frac{n-1}{2}\;\frac{n+3}{2})(\frac{n+1}{2})$.
    The corresponding connectivity graphs $C_{\path_n}(\sigma)$ are shown in \Cref{fig:even odd path connectivity}.
    
    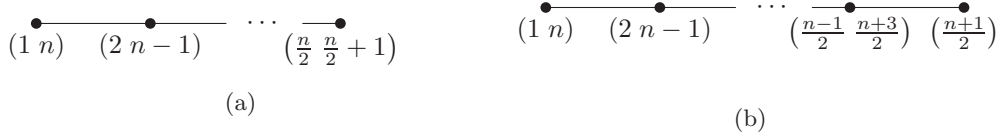
\begin{figure}
        \centering
            \begin{subfigure}{.4\linewidth}
            \input{img/path_graph_connectivity_even.tikz}
            \caption{}
            \label{fig:patheven}
            \end{subfigure}
            \hspace{1mm}
            \begin{subfigure}{.4\linewidth}
            \input{img/path_graph_connectivity_odd.tikz}
            \caption{}
            \label{fig:pathodd}
            \end{subfigure}
        
        \caption{The connectivity graphs $C_{\path_n}(\sigma)$ where $\sigma$ is the generator of $\Z/2\Z$ in the case where $n$ is even \eqref{fig:patheven} and odd \eqref{fig:pathodd}.}
        \label{fig:even odd path connectivity}
    \end{figure}
    
    We first consider the case when $n$ is even.
    Let $F$ be any subforest of $C_{\path_n}(\sigma)$.
    Using formula \eqref{eq:volume F}, we get
        \[
        \vol{\hobox F} = \left(2^{n/2-c(F)}\right)\left(\frac{1}{2^{n/2}}\right)\left(2^{c(F)}\right) = 1.
        \]
    We next assess $(\sigma,\path_n)$-compatibility of $F$.
    Since $\ell_i=2$ for all cycles $\sigma_i$, the left side of inequality \eqref{eq:compatibility} will always be $1$.
    However, whenever $T$ is the connected component of $F$ containing the cycle $(\frac{n}{2}\;\frac{n}{2}+1)$, the right side of inequality \eqref{eq:compatibility} will be $0$ since this particular cycle contains one edge of $\path_n$ and all other cycles have no edges within them.
    This means that no subforest $F$ of $C_{\path_n}(\sigma)$ is $(\sigma,\path_n)$-compatible.
    Hence when $n$ is even, the Ehrhart quasipolynomial of $Z_{\path_n}^\sigma$ is $0$ for odd dilations and $(t+1)^{n/2-1}$ (the Ehrhart polynomial of an integral parallelotope) for even dilations $t$:
        \[
        L_{Z_{\path_n}^\sigma}(t) = \begin{cases}
            (t+1)^{n/2-1}, & t \text{ even}\\
            0, & t \text{ odd}
        \end{cases}.
        \]
    Next, using \Cref{cor:graphic zonotope ehrhart series} we get
        \begin{align*}
        \Ehr_{Z_{\path_n}^\sigma}(z) &= \sum_{t=0}^\infty (2t+1)^{n/2-1}\cdot z^{2t}\\
        &= \sum_{\substack{\text{subforest } F\\\text{of } C_{\path_n}(\sigma)}} \frac{2^{n/2-c(F)}\cdot A_{n/2-c(F)}(z^2)}{(1-z^2)^{n/2-c(F)+1}}\\
        &= \frac{\displaystyle{\sum_{k=0}^{\frac{n}{2}-1} \binom{\frac{n}{2}-1}{k}\cdot 2^k\cdot (1-z^2)^{n/2-k-1}\cdot A_k(z^2) }}{(1-z^2)^{n/2}},
        \end{align*}
    and we can see directly from this Ehrhart series that $\Hstar(Z_{\path_n};z)$ is a polynomial when $n$ is even. 
    
    Now suppose $n$ is odd.
    As in the even case, we can use \eqref{eq:volume F} to find that $\vol{\hobox{F}}$ is again $1$ for all subforests $F$ of $C_{\path_n}(\sigma)$.
    However, unlike the even case, we have $\deg(\sigma_j,\sigma_j)=0$ for all cycles $\sigma_j$ of $\sigma$.
    This means that the half-integral shift of equation \eqref{eq:gzmink} disappears in this case, and we are left with an integral polytope $Z_{\path_n}^\sigma$.
    Since all the fixed polytopes are integral in this case, the $\Hstar$-series is guaranteed to be a polynomial.
    
    Let us go further and show that the $\Hstar$-polynomial is effective when $n$ is odd.
    The fixed polytopes $Z_{\path_n}^{\id}$ and $Z_{\path_n}^\sigma$ are both lattice parallelotopes with relative volume $1$ and dimensions $n-1$ and $\frac{n-1}{2}$, respectively.
    So we get
        \begin{align*}
        \EE(Z_{\path_n};z)(\id) &= \frac{\frac{1}{z}A_{n-1}(z)}{(1-z)^n} \\
        \EE(Z_{\path_n};z)(\sigma) &= \frac{\frac{1}{z}A_{(n-1)/2}(z)}{(1-z)^{(n+1)/2}} = \frac{(1+z)^{(n-1)/2}\frac{1}{z}A_{(n-1)/2}(z)}{(1-z^2)^{(n-1)/2}(1-z)},
        \end{align*}
    where we have arranged both series to have the desired denominator $\det(I-\rho\cdot z)$.
    Notice that the numerators of both rational functions are polynomials of degree $n-2$ with positive integer coefficients, and they are $\Hstar(Z_{\path_n};z)$ evaluated at the elements of $\Z/2\Z$.
    Write $\Hstar(Z_{\path_n};z) = \sum_{k=0}^{n-2}H_k^\ast z^k$ for virtual characters $H_k^\ast\in R(\Z/2\Z)$.
    For each $k$, there exist $a_k,b_k\in \R$ such that $H_k^\ast = a_k\chitriv + b_k\chialt$.
    For a polynomial $f(z)$, let $[z^k]f(z)$ denote the coefficient of $z^k$ in $f(z)$.
    Plugging in $\id$ and $\sigma$ to $\Hstar(Z_{\path_n};z)$ gives
        \begin{align*}
            a_k+b_k &= [z^k]\left(\frac{1}{z}A_{n-1}(z)\right)\\
            a_k-b_k &= [z^k]\left((1+z)^{(n-1)/2}\frac{1}{z}A_{\frac{n-1}{2}}(z)\right).
        \end{align*}
    Summing these two equations, it is clear to see that $a_k>0$, which supports \cite[Conjecture~12.4]{Stapledon}.
    Subtracting the second equation from the first, we see that $b_k\ge0$ if and only if the right side of the first equation is greater than or equal to the right side of the second equation.
    This is indeed true, and we prove it combinatorially.
    
    \begin{lemma}
    For $0\le k\le n-2$, the coefficient of $z^k$ in $\frac{1}{z}A_{n-1}(z)$ is greater than or equal to the coefficient of $z^k$ in $(1+z)^{(n-1)/2}\frac{1}{z}A_{(n-1)/2}(z)$.
    \end{lemma}
    \begin{proof}
    The coefficient of $z^k$ in $\frac{1}{z}A_{n-1}(z)$ is the Eulerian number $A_{n-1,k}$, the number of permutations of the set $\{1,2,\dots,n-1\}$ with exactly $k$ ascents.
    On the other hand, we have
        \begin{align}
            \nonumber[z^k]\left( (1+z)^{(n-1)/2}\frac{1}{z}A_{\frac{n-1}{2}}(z) \right) &= [z^k] \left(\left( \sum_{i=0}^{(n-1)/2} \binom{\frac{n-1}{2}}{i}z^i \right)\left( \sum_{j=0}^{(n-3)/2} A_{\frac{n-1}{2},j}z^j \right)\right)\\
            \nonumber&= \sum_{i=0}^k \binom{\frac{n-1}{2}}{k-i}A_{\frac{n-1}{2},i}\\
            \label{eq:binomial}&=
            \sum_{i=0}^k A_{\frac{n-1}{2},i}\binom{\frac{n-1}{2}-1}{k-i-1} + 
            \sum_{i=0}^k A_{\frac{n-1}{2},i}\binom{\frac{n-1}{2}-1}{k-i}\\
            \nonumber&\le A_{n-1,k}.
        \end{align}
        
    To see that \eqref{eq:binomial} $\le A_{n-1,k}$, we observe that \eqref{eq:binomial} counts some, but not all, of the permutations in $S_{n-1}$ with exactly $k$ ascents.
    The first summand of \eqref{eq:binomial} deals with the case where the elements $1,2,\dots,\frac{n-1}{2}$ come before the remaining elements in the permutation.
    The Eulerian number $A_{(n-1)/2,i}$ counts the number of permutations of $\{1,2,\dots,\frac{n-1}{2}\}$ with exactly $i$ ascents; we are then guaranteed an ascent at position $\frac{n-1}{2}$, and then the binomial coefficient $\binom{(n-1)/2-1}{k-i-1}$ gives the number of ways to choose the \emph{positions} of the remaining $k-i-1$ ascents from $\{(n+1)/2, (n+3)/2,\dots, n-2\}$ (which is less than or equal to the number of ways to complete the permutation with ascents in those positions).
    Likewise, the second summand of \eqref{eq:binomial} deals with the case where the elements $\frac{n+1}{2}, \frac{n+3}{2},\dots, n-1$ come before the remaining elements in the permutation; in this case we are guaranteed to have no ascent at position $\frac{n-1}{2}$.
    Using this reasoning, it is clear that \eqref{eq:binomial} is less than or equal to the number of permutations in $S_{n-1}$ with exactly $k$ ascents where the elements $\{1,2,\dots,\frac{n-1}{2}\}$ occur either in the first $\frac{n-1}{2}$ positions or in the last $\frac{n-1}{2}$ positions.
    This in turn is clearly less than or equal to $A_{n-1,k}$.
    \end{proof}
    
    Since $a_k,b_k\ge0$ for all $k$, the $\Hstar$-polynomial of $Z_{\path_n}$ is effective with respect to the action of $\Z/2\Z$ when $n$ is odd.

%% file: img/path_graph_zonotope.tikz
\begin{tikzpicture}[scale=1.5,
    back/.style={loosely dotted, thick},
	edge/.style={color=cyan},
	facet/.style={fill=cyan,fill opacity=0.300000},
	vertex/.style={fill=cyan}]

\coordinate (0111) at (0,1,1);
\coordinate[label=left:0120] (0120) at (0,1,2);
\coordinate[label=left:0201] (0201) at (0,2,0);
\coordinate[label=left:0210] (0210) at (0,2,1);
\coordinate[label=right:1011] (1011) at (1,0,1);
\coordinate[label=below:1020] (1020) at (1,0,2);
\coordinate[label=right:1101] (1101) at (1,1,0);
\coordinate (1110) at (1,1,1);

\draw[edge] (0120) -- (1020) -- (1011);
\draw[edge,back] (0120) -- (0111) -- (1011);
\draw[edge] (0201) -- (0210) -- (1110) -- (1101) -- cycle;
\draw[edge,back] (0111) -- (0201);
\draw[edge] (0120) -- (0210);
\draw[edge] (1020) -- (1110);
\draw[edge] (1011) -- (1101);

\fill[facet] (1110) -- (1020) -- (1011) -- (1101) -- cycle;
\fill[facet] (1101) -- (0201) -- (0210) -- (1110) -- cycle;
\fill[facet] (0210) -- (1110) -- (1020) -- (0120) -- cycle;

\coordinate[label=below:0111] (0111) at (0,1,1);
\coordinate[label=left:1110] (1110) at (1,1,1);

\fill[vertex] (0120) circle(1.2pt);
\fill[vertex] (1020) circle(1.2pt);
\fill[vertex] (1011) circle(1.2pt);
\fill[vertex] (0111) circle(1.2pt);
\fill[vertex] (1110) circle(1.2pt);
\fill[vertex] (1101) circle(1.2pt);
\fill[vertex] (0201) circle(1.2pt);
\fill[vertex] (0210) circle(1.2pt);

\end{tikzpicture}

%% file: img/path_graph.tikz
\begin{tikzpicture}
\newcommand\Square[1]{+(-#1,-#1) rectangle +(#1,#1)}

\coordinate[label=below:1] (1) at (0,0);
\coordinate[label=below:2] (2) at (1,0);
\coordinate[label=below:3] (3) at (2,0);
\coordinate[label=below:4] (4) at (3,0);
\draw (1) -- (4);

\fill[cyan] (1) circle(2pt);
\fill[magenta] (2) \Square{2pt};
\fill[magenta] (3) \Square{2pt};
\fill[cyan] (4) circle(2pt);

\end{tikzpicture}

%% file: img/path_graph_connectivity.tikz
\begin{tikzpicture}
\newcommand\Square[1]{+(-#1,-#1) rectangle +(#1,#1)}

\coordinate[label=below:(14)] (14) at (0,0);
\coordinate[label=below:(23)] (23) at (1,0);
\draw (14) -- (23);

\fill[cyan] (14) circle(2pt);
\fill[magenta] (23) \Square{2pt};

\end{tikzpicture}

%% file: img/path_graph_connectivity_even.tikz
\begin{tikzpicture}

\coordinate[label=below:$(1\;n)$] (1) at (0,0);
\coordinate[label=below:$(2\;n-1)$] (2) at (1.5,0);
\coordinate (3) at (2.5,0);
\node at (3,0) {$\cdots$};
\coordinate (4) at (3.5,0);
\coordinate[label=below:$\left(\frac{n}{2}\;\frac{n}{2}+1\right)$] (5) at (4,0);
\draw (1) -- (2) -- (3);
\draw (4)--(5);

\fill (1) circle(2pt);
\fill (2) circle(2pt);
\fill (5) circle(2pt);

\end{tikzpicture}

%% file: img/path_graph_connectivity_odd.tikz
\begin{tikzpicture}

\coordinate[label=below:$(1\;n)$] (1) at (0,0);
\coordinate[label=below:$(2\;n-1)$] (2) at (1.5,0);
\coordinate (3) at (2.5,0);
\node at (3,0) {$\cdots$};
\coordinate (4) at (3.5,0);
\coordinate[label=below:$\left(\frac{n-1}{2}\;\frac{n+3}{2}\right)$] (5) at (4,0);
\coordinate[label=below:$\left(\frac{n+1}{2}\right)$] (6) at (5.5,0);
\draw (1) -- (2) -- (3);
\draw (4)--(5)--(6);

\fill (1) circle(2pt);
\fill (2) circle(2pt);
\fill (5) circle(2pt);
\fill (6) circle(2pt);

\end{tikzpicture}

%% file: 3.2_invariant_decompositions.tex
    \subsection{Invariant Half-Open Decompositions}\label{sec:invariant decomps} 
    
    In this section, we use $G$-invariant triangulations to calculate the equivariant Ehrhart series. 
    More precisely, the triangulations are half-open decompositions. 
    For background definitions in discrete geometry see
    \cite{Ziegler}.
    
    \begin{definition}\label{def:halfopen}
    A \emph{half-open decomposition} of a lattice polytope $P$ with lattice triangulation $\mathcal T$ is a partition of the face poset of $\mathcal T$ into intervals such that the empty set is not in its own class.
    We use $\mathcal T ^{[\, )}$ to refer to the half-open decomposition, i.e., the triangulation together with the partition of the face poset.
    We identify a simplex $S$ in $\mathcal T$ with its set of vertices when convenient.
    \end{definition}
    
    \Cref{def:halfopen} is closely related to partitionability; the face poset of a pure simplicial complex is \emph{partitionable} if it can be partitioned into disjoint intervals such that each maximal element is a facet of the complex, see \cite[Chapter 2]{stanley_green}.
    In this case, it is easy to read off the $h$-vector of the simplicial complex by recording the heights of the minimal simplices in the intervals.
    \Cref{def:halfopen} generalizes the typical method for creating a half-open decomposition of (the cone over) a triangulated polytope by choosing a generic vector in the interior of one simplex and then taking ``sunny-side'' of each face
    ; one asks if it is possible to remain in a facet of the triangulation when walking in the direction of the vector
    , see 
     \cite[Secton 5.3]{BeckSanyal} or \cite[Section 3.4]{latticepolytopeslectures} for details.
Using a generic vector in the typical method creates a shelling (and more weakly a partition) of the simplicial complex.
However, this typical method can fail to create symmetric half-open decompositions with respect to the group action, and does not easily allow for simplices of varying dimension in the triangulation, highlighting the usefulness of the more general half-open decompositions from \Cref{def:halfopen}. 
    
    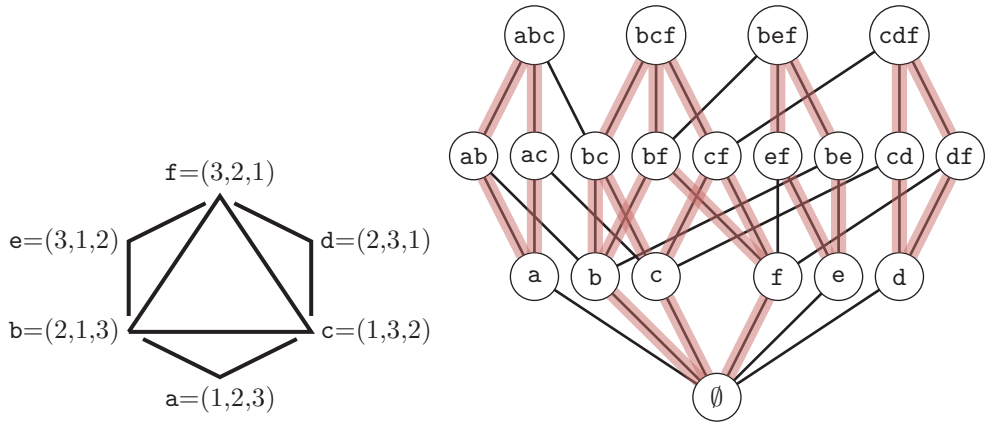
\begin{figure}[h]
        \begin{center}
        \input{img/Pi3_half_open_decomp.tikz}
        \end{center}
        \caption{A half-open decomposition of the permutahedron $\Pi_3$ determined by a triangulation (top image) and a partition of the face poset of the triangulation into intervals shown via its poset diagram (bottom image)}
        \label{fig:exampledecompositions}
    \end{figure}
    
    \begin{example}\label{example:pi3}
    The two-dimensional permutahedron $\Pi_3 \subset \R^3$ obtained as the convex hull of the permutations of the coordinates $(1,2,3)$ admits a half-open decomposition. Figure \ref{fig:exampledecompositions} shows a triangulation of $\Pi_3$ and a partition of the face poset of the triangulation. The broken edges in the triangulation indicate which faces are open and which are closed in the cone over $\Pi_3$ according to the partition.
    \end{example}
    
    \begin{example}\label{example:square}
    The 0/1 square admits a half-open decomposition.
    Triangulate the square into two triangles: 
    $\Delta(abc) = \conv(\{(0,1),(0,0),(1,1)\})$ and 
    $\Delta(bcd) = \conv(\{(0,0),(1,1),(1,0)\})$, as shown on the left in Figure \ref{fig:exampledecomposition2}.
    Partition the face poset into three intervals: $[a,abc],[\emptyset,bc],[d, bcd]$ to create the half-open decomposition, as shown on the right in Figure \ref{fig:exampledecomposition2}.
    The broken edges in the triangulation shown in Figure~\ref{fig:exampledecomposition2} indicate which faces in the cone over the triangulated square are open according to the partition of the face poset.
    The maximal simplex $bc$ in the interval $[\emptyset,bc]$ is not full-dimensional, which emphasizes the additional flexibility offered by these half-open decompositions.
    
    \begin{figure}
        \begin{center}
        \input{img/square_half_open_decomp.tikz}
        \end{center}
        \caption{A half-open decomposition of the 0/1 square determined by a triangulation (left image) and a partition of the face poset of the triangulation into intervals shown via its poset diagram (right image)}
        \label{fig:exampledecomposition2}
    \end{figure}
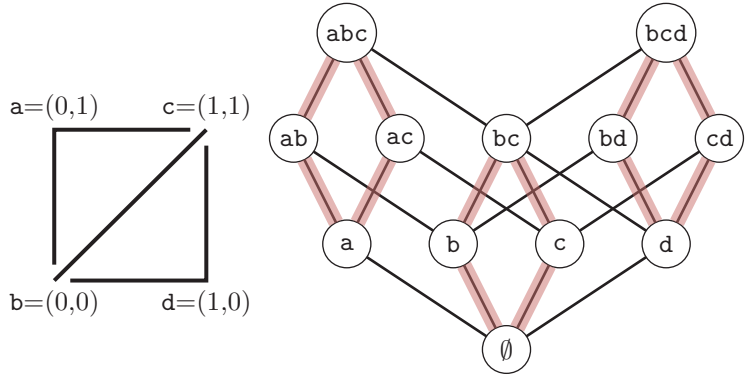
    
    \end{example}

    \begin{definition}
    Let $P \subset \R^d$ be a lattice polytope, and let $I = [\underline S, \overline S]$ be an interval in a half-open decomposition $\mathcal T^{[ \, )}$ of $P$.
    The \emph{half-open cone over the interval$~I$}, denoted $\Cone(I)$, is the set of points $\bx \in \R^{d +1}$ such that:
        \begin{equation*}
        \bx = \sum_{\bv_i \in \underline S} \lambda_i (\bv_i,1) + \sum_{\bv_j \in \overline S \setminus \underline S} \mu_j (\bv_j,1),  \text{ where }  \lambda_i \in \R_{>0} \text{ and } \mu_j \in \R_{\geq 0} 
        .
        \end{equation*}
    \end{definition}
    
    \noindent
    \Cref{prop:disjoint_union} and \Cref{prop:conedecomp} below are similar in spirit to \cite[Propositions 3.22, 3.27]{latticepolytopeslectures}.
    
    \begin{proposition}\label{prop:disjoint_union}
    The cone over a lattice polytope $P$ with half-open decomposition~$\mathcal T^{[\,)}$ is equal to the disjoint union of half-open cones over the intervals of $~\mathcal T^{[\,)}$:
        $$
        \Cone(P) = \bigsqcup_{I \in \mathcal T^{[\,)}} \Cone(I ).
        $$
    \end{proposition}
    \begin{proof}
    The inclusion $ \Cone(P) \supseteq \bigsqcup_{I \in \mathcal T^{[\,)}} \Cone(I )$ is clear.
    The origin is contained in the half-open cone, $\Cone(I)$, for the unique interval $I = [\emptyset, \overline S]$ that contains the empty set.
    Let $\bz \in \Cone(P)$, $\bz \neq 0$.  
    There exists a minimal simplex $S \in \mathcal T^{[\,)}$ (with respect to inclusion) such that $\bz \in \Cone(S)$ and
        $
        \bz= \sum_{\bv_i \in S}\lambda_i (\bv_i,1), \,\,\,\, \lambda_i \in \R_{>0}.
        $
    The simplex $S$ is contained in a unique interval $I = [\underline S, \overline S]$ of $\mathcal T^{[\,)}$.
    Thus,
        $$
        \bz= \sum_{\bv_i \in S}\lambda_i (\bv_i,1) = \sum_{\bv_i \in \underline{S}} \lambda_i (\bv_i,1) + \sum_{\bv_i \in S \setminus \underline{S}} \lambda_i (\bv_i , 1) + \sum_{\bv_j \in \overline S \setminus S} 0 (\bv_j , 1).
        $$
    This implies that all vertices of $ \underline {S}$ have a coefficient in $\R_{ > 0}$, 
    and all vertices of $\bv \in \overline{S} \setminus \underline{S}$ have a coefficient in $\R_{\geq 0}$.
    Thus, $\bz \in \Cone(I)$.
    It remains to show that the union $\bigsqcup_{I \in \mathcal T^{[\,)}} \Cone(I ) $ is disjoint.
    Suppose $\bz \neq 0$, $\bz \in \Cone(I)$, and $\bz \in \Cone(I')$ for two intervals $I =[\underline S,\overline S]$, $I' =[\underline{S}',\overline{S}']$ of$~\mathcal T^{[\,)}$.
    Then 
        \begin{equation}\label{eqn:zequals}
        \bz = \sum_{\bv_i \in T} \lambda_i (\bv_i,1) , \,\,\,\, \lambda_i \in \R_{>0},
        \end{equation}
    for a \emph{unique} face $T$ of the simplicial complex.
    As $\bz$ is in $\Cone(I)$,
    $$
    \bz = \sum_{\bv_i \in \underline{S}} \lambda_i (\bv_i,1) + \sum_{\bv_j \in \overline{S} \setminus \underline{S}} \mu_j (\bv_j,1),
    $$
    where $\lambda_i \in \R_{>0}$ and $\mu_j \in \R_{\geq 0}$.
    Restricting to summands with strictly positive coefficients expresses $\bz$ as a conical combination of vertices of $T$, since it is an expression of the form \eqref{eqn:zequals}.
    This also shows that $\underline{S} \subset T$.
    Likewise, $\underline {S}' \subseteq T \subseteq \overline {S}'$.
    Thus $T$ is contained in two intervals in the half-open decomposition, and they must be the same. 
    \end{proof}
    
    \begin{definition}
    Let $P \subseteq \R^d$ be a lattice polytope, and let $I = [\underline S, \overline S]$ be an interval in a half-open decomposition $\mathcal T^{[ \, )}$ of $P$. 
    The \emph{half-open fundamental parallelepiped of the interval $I$}, denoted 
    $\hobox (I)$, is the set of points $\bx \in \R^{d +1}$ such that:
        \begin{equation*}
        \bx = \sum_{\bv_i \in \underline S} \lambda_i (\bv_i,1) + \sum_{v_j \in \overline S \setminus \underline S} \mu_j (\bv_j,1),  \text{ where }  \lambda_i \in (0,1] \text{ and }\mu_j \in [0,1) 
        .
        \end{equation*} 
    \end{definition}
    
    For an interval $I$ of a half-open decomposition, $\Bx(I)$ is defined to be the set of lattice points in the half-open fundamental parallelepiped
    $~\hobox (I)$:
        $$
        \Bx(I) = \hobox (I) \cap \Z^{d+1}.
        $$
    Throughout, we use $\Bx(I)_k$ to denote the set of lattice points in the half-open fundamental parallelepiped of $I$ at height $k$, i.e. with last coordinate equal to $k$.
    
    \begin{proposition}\label{prop:conedecomp}
    Let $P$ be a lattice polytope with half-open decomposition $\mathcal T^{[\,)}$.
    Every lattice point $\bx \in \Cone(P)$ can be expressed uniquely as $\bx = \bw + \mathbf y$ for some $\bw = \sum_{\bv_i \in \overline S} \lambda_i (\bv_i,1),$ with $\lambda_i \in \Z_{\geq 0}$  and $\mathbf y \in \Bx(I)$ for some interval $I = [\underline S, \overline S] \in~\mathcal T^{[\,)}$.
    \end{proposition}
    \begin{proof}
    Let $\bx \in \Cone(P)$. 
    Then by Proposition \ref{prop:disjoint_union}, $\bx \in \Cone(I)$ for a unique interval $I = [\underline S, \overline S]$ of $\mathcal T^{[\,)}$, and
    \begin{flalign*}
    	\bx =& \quad \sum_{\bv_i \in \underline S} \lambda_i (\bv_i,1) + \sum_{\bv_j \in \overline S \setminus \underline S} \mu_j (\bv_j,1), \,\,\, \text{ where } \lambda_i \in \R_{>0} \text{ and }\mu_j \in \R_{\geq 0}, \\
    	    =& \quad \sum_{\bv_i \in \underline S}\bigg( \lceil \lambda_i \rceil - 1 + \lambda_i - (\lceil \lambda_i \rceil -1)\bigg)(\bv_i,1) + \sum_{\bv_j \in \overline S \setminus \underline S}\left( \lfloor \mu_j \rfloor + (\mu_j - \lfloor \mu_j \rfloor)\right) (\bv_j,1) \\
    	    =& \quad \left( \sum_{\bv_i \in \underline S}\left( \lceil \lambda_i \rceil - 1 \right )(\bv_i,1) +\sum_{\bv_j \in \overline S \setminus \underline S} \lfloor \mu_j \rfloor (\bv_j,1)  \right) \\
    	    & \quad+ \left ( \sum_{\bv_i \in \underline S} \left(1 - \lceil \lambda_i \rceil + \lambda_i\right) (\bv_i,1) + \sum_{\bv_j \in \overline S \setminus \underline S}\left(\mu_j - \lfloor \mu_j \rfloor\right) (\bv_j,1) \right).
    \end{flalign*}
    \end{proof}
    
    \begin{definition}
    Let $P$ be a lattice polytope invariant under the action of a group$~G$. A half-open decomposition $\mathcal T^{[\,)}$ of $P$ is called \emph{$G$-invariant} if 
        \begin{enumerate}
        
        \item Simplices are sent to simplices (the triangulation is $G$-invariant), and
        
        \item The action of $G$ induces an automorphism of the face poset of $\mathcal T^{[\,)}$ such that intervals of $\mathcal T^{[\,)}$ are sent to intervals.
        
        \end{enumerate}
    \end{definition}

We present a lemma on the representation theory of the symmetric algebra, which is useful in \Cref{thm:permrep}, to follow.
Through it, we can see one motivation for considering the denominator $\det(I -z \cdot \rho)$ to express the equivariant Ehrhart series.
Let $\{\be_1, \dots, \be_n\}$ be a basis for a vector space $V$.
The set $\left \{\be_{i_{\scriptstyle 1}}\odot  \, \cdots \, \odot \be_{i_{\scriptstyle t}} \mid i_\ell \leq i_{\ell+1} \right \}$ 
is a basis for $\sym^t(V)$, the \emph{$t$-th symmetric power}.
A basis for the \emph{$t$-th exterior power}, $\Lambda^t(V)$, is the set $\{\be_{i_{\scriptstyle 1}} \wedge \dots \wedge \be_{i_{\scriptstyle t}} \mid i_\ell < i_{\ell+1} \}. $
If $\rho$ is a representation of $G$ on $V$, then $G$ also acts (diagonally) on $\bigotimes^t V$ by
$${g(\bv_1 \otimes \dots \otimes \bv_t) = g\bv_1 \otimes \dots \otimes g\bv_t.}$$
The group $G$ acts diagonally on $\sym^t(V)$ and $\Lambda^t(V)$ in the same manner.
It is useful to look at the characters of these actions, which we denote by $\chi_{\sym^t(V)}$ and $\chi_{\Lambda^t(V)}$ respectively.
For $g \in G$, let $\{\be_1,\dots , \be_n \}$ be an orthonormal eigenbasis of the action of $g$ on $V$, with eigenvalues $\{\lambda_1,\dots,\lambda_n\}$. 
This is possible as $\rho(g)$ can be written as a unitary matrix, see for example \cite[Chapter 1]{Serre}.
Then 
$$g(\be_{i_1}\odot \be_{i_2}\odot \cdots \odot \be_{i_t}) = \lambda_{i_1}\lambda_{i_2}\cdots \lambda_{i_t}\be_{i_1}\odot \be_{i_2}\odot \cdots \odot \be_{i_t},$$
where $i_1 \leq i_2 \leq \dots \leq i_t$.
This shows that an eigenvector of $\sym^t(V)$ has an eigenvalue that is a monomial in the eigenvalues of $\rho(g)$ of degree $t$.
So the sum of the eigenvalues of $g$ acting on $\sym^t(V)$ is the sum of all monomials in the eigenvalues of degree $t$.
With respect to the eigenbasis of the action of $g$ on $V$, $$\det(I -z \cdot \rho(g)) = 
\det \left( \begin{bmatrix}1-z\lambda_1 & & \\
& \ddots & \\
& & 1-z\lambda_n 
\end{bmatrix} \right )
=\Pi_{i=1}^{n}(1-z\lambda_i).$$
In fact, $\det(I-z \cdot \rho (g))$ is independent of the choice of basis to express $\rho$.
Thus the generating function of characters of the symmetric powers has the rational form:
\begin{equation}\label{eqn:chisym}
\sum_{t \geq 0} \chi_{\sym^t(V)}z ^t = \frac{1}{\det(I - z \cdot \rho )}.
\end{equation}

For the $t$-th exterior power, ${g( \be_{i_1}\wedge \be_{i_2}\wedge \cdots \wedge \be_{i_t}) = \lambda_{i_1}\lambda_{i_2}\dots \lambda_{i_t}\be_{i_1}\wedge \be_{i_2}\wedge \dots \wedge \be_{i_t}}$, where $i_1 < i_2 < \dots <i_t$.
The character of the $t$-th exterior product evaluated at $g$ is the sum of all square-free homogeneous monomials in $\lambda_i$ of degree $t$.
As such, we can rewrite 
\begin{equation}\label{eqn:det}
\det(I - z \cdot \rho) = \sum_{i= 0 }^n (-1)^i \chi_{\Lambda^i(V)} z^i.
\end{equation}
We have now recovered Lemma 3.1 in \cite{Stapledon}, which states:
\begin{lemma}[{\cite[Lemma 3.1]{Stapledon}}]\label{lemma:det}
Let $G$ be a finite group and let $V$ be an $n$-dimensional representation. 
Then
\begin{equation}\label{sympowers}
\sum_{t \geq 0} \sym^t V z^t = \frac{1}{1 - Vz + \Lambda^2 V z^2 - \dots + (-1)^n \Lambda^n V z ^n}.
\end{equation}
Moreover, if an element $g \in G$ acts on $V$ via a matrix $A$, and if $I$ denotes the identity $n \times n$ matrix, then both sides equal $\frac{1}{\det(I - z \cdot A)}$ when the associated characters are evaluated at $g$.
\end{lemma}
    
For a $G$-invariant simplex $S$ with vertices $\{\bv_{0},\bv_{1},\dots ,\bv_{d} \} \in M \times \{1\}$ as in the \Cref{setup}, the infinite-dimensional $\C$-vector space 
$\C\left[\Z_{\geq 0}\verts(S)\right]$ with basis: $$ \left \{\bx = \sum_{i=0}^{d} c_i\bv_i \big \lvert c_i \in \Z_{\geq 0} \text{ for all  } i \in \{0,\dots,d \} \right\}$$ 
is isomorphic as a graded $G$-module to the vector space $\sum_{t \geq 0} \sym^t (M' \bigotimes_{\Z} \C)$, and their characters coincide. If $S$ is a unimodular simplex, then \eqref{eqn:chisym} expresses the equivariant Ehrhart series of $S$.
This idea is generalized to all simplices in \cite[Proposition 6.1]{Stapledon}, and we now generalize it further to polytopes admitting $G$-invariant half-open decompositions.

    As in the Setup \ref{setup}, let $P$ be a lattice polytope with vertices in $M \times \{1\}$, invariant under the linear action of a finite group $G$.
    For a half-open decomposition $\mathcal T^{[\,)}$ of~$P$, let $\Bx(\mathcal T^{[\,)})_i$ denote the union of the sets $\Bx(I)_i$ for intervals $I$ of $\mathcal T^{[\,)}$.
    The group $G$ permutes the lattice points in $\Bx(\mathcal T^{[\,)})_i$, and the corresponding permutation character is written $\chi_{\Bx(\mathcal T^{[\,)})_i}$.
    \begin{theorem}\label{thm:permrep}
    Let $\mathcal T^{[\,)}$ be a $G$-invariant half-open decomposition of a $d$-dimensional polytope $P$ as in the Setup \ref{setup} such that $\dim(\overline {S}) = d$ for all intervals $[\underline{S},\overline{S}]$ of $T^{[\,)}$ and such that all orbits of intervals of $T^{[\,)}$ have order $|G|$ except for a unique $G$-invariant interval. 
    Then the $\Hstar$-polynomial is the permutation representation on the union of box points in $\mathcal T^{[\,)}$, graded by height: 
    $$
    \Hstar (P;z) = \sum_{i=0}^d \chi_{\Bx\left(\mathcal T^{[\,)}\right)_i} z^i.
    $$
    \end{theorem}
    
    \begin{proof}
    Let $q$ be the number of orbits of intervals of $\mathcal T^{[\,)}$.
    There necessarily exists a $G$-invariant simplex $\overline{S_0}$ contained in an interval $[\underline{S_0}, \overline{S_0}]$ of $\mathcal T^{[\,)}$ that contains the $G$-fixed barycenter of the vertices of $P$. Label its vertices by $\{\bv_{0,0},\dots,\bv_{0,d} \}$ and label the orbit of the interval $[\underline{S_0}, \overline{ S_0}]$ by $\mathcal O_0$.
    Order the other orbits $\mathcal O_1, \dots, \mathcal O_{q-1}$ of intervals of $\mathcal T^{[\,)}$.
    For each $i \in [q-1]$, label representative simplices $\underline{S_i}$ and $\overline{ S_i}$ contained in an interval $[\underline{S_i} , \overline {S_i}]$ of $\mathcal O_i$ and label the vertices of $\overline {S_i}$ as $\bv_{i,0}, \dots ,\bv_{i,d}$.
    Identify the $G$-module $\C[\Z_{\geq 0} \verts(\overline{S_0})]$ with the polynomial ring and graded $G$-module, $\C[X_0, \dots, X_d]$ where $g(X_i) \coloneqq X_j$ if $g(\bv_{0,i}) = \bv_{0,j}$.
    Viewing $g \in G$ as a permutation on $\{0,1,\dots,d \}$,
    we write, $g(X_0^{c_0}\cdots X_d^{c_d}) = X_{g(0)}^{c_0} \cdots X_{g(d)}^{d}$.
    We define a map $f$ from
    \begin{equation}\label{eq:tensor}
    \C \left[ X_0, \dots, X_d\right] \bigotimes_{\C}
    \left(  \C\left[\Bx\left(\left[\underline {S_0},\overline{S_0}\right]\right)\right]  \bigoplus_{\substack{i: |\mathcal O_i|= |G| \\ g \in G}} \C \left[\Bx\left(\left[\underline {gS_i},\overline{gS_i}\right]\right)\right]
    \right)
    \end{equation}
    to
    $
    \C[\Cone(P) \cap M']
    $
    by defining $f$ on a basis and extending linearly.
   The map $f$ makes use of the expression for lattice points in $\Cone(P)$ given in \Cref{prop:conedecomp}.
    To shorten notation, for a lattice point $\mathbf y_0$ in $\Bx([\underline{S_0},\overline{S_0}])$ or $g\mathbf y_i$ in $\Bx([\underline{gS_i},\overline{gS_i}])$, we denote the corresponding basis vector of 
    \begin{equation}\label{eq:basis}
    \left(  \C\left[\Bx\left(\left[\underline {S_0},\overline{S_0}\right]\right)\right]  \bigoplus_{\substack{i: |\mathcal O_i|= |G| \\ g \in G}} \C \left[\Bx\left(\left[\underline {gS_i},\overline{gS_i}\right]\right)\right]
    \right)
    \end{equation}
    as $\mathbf y_0$ or $g\mathbf y_i$ respectively.
    The direct sum \eqref{eq:basis} is also a graded $G$-module.
    A basis for the tensor product \eqref{eq:tensor} is the set
    $$
    \left \{
    \begin{aligned}
    	(X_0^{c_0} \cdots X_d^{c_d}) \otimes \mathbf y_0 : \{c_0, \dots, c_d \} \in \Z_{\geq 0},\, &\mathbf y_0 \in \Bx\left(\left[\underline {S_0},\overline{S_0}\right]\right), \\
    	g(X_0^{c_0} \cdots X_d^{c_d}) \otimes g\mathbf y_i : \{c_0, \dots, c_d \} \in \Z_{\geq 0},\, &\mathbf y_i \in \Bx\left(\left[\underline{ S_i},\overline{S_i}\right]\right),\,\, g \in G, \,\,|\mathcal O _i| = |G|
    \end{aligned}
    \right \} .
    $$
    Define 
    $$f(g(X_0^{c_0}\cdots X_d^{c_d}) \otimes g \mathbf y_i )= f((X_{g(0)}^{c_0}\cdots X_{g(d)}^{c_d}) \otimes g \mathbf y_i ) := \sum_{j = 0}^d c_j g (\bv_{i,j}) + g\mathbf y_i.$$
     
    Suppose $g\mathbf y_0 = \bw_0$ for box points $\mathbf y_0$ and $\bw_0$ in $\Bx([\underline{S_0},\overline{S_0}])$ and therefore $$X_{g(0)}^{c_0} \cdots X_{g(d)}^{c_d} \otimes g\mathbf y_0 = X_{0}^{c_{g^{-1}(0)}} \cdots X_{d}^{c_{g^{-1}(d)}} \otimes \bw_0.$$ 
    To check $f$ is well-defined, we compute:
    \begin{align*}
    f \left( X_{g(0)}^{c_0} \cdots X_{g(d)}^{c_d} \otimes g\mathbf y_0 \right ) 
    =
    \sum_{j=0}^d c_j g(\bv_{i,j}) +g(\mathbf y_0) 
    = \sum_{j=0}^d c_j \bv_{i,g(j)} + \bw_0 
    &= \sum_{\gamma = 0}^d c_{g^{-1}(\gamma)} \bv_{i,\gamma} + \bw_0 \\
    &= f(X_{0}^{c_{g^{-1}(0)}} \cdots X_{d}^{c_{g^{-1}(d)}} \otimes \bw_0).
    \end{align*}
    The map $f$ is an isomorphism of vector spaces (\Cref{prop:conedecomp}) and a $G$-module isomorphism, as we verify on the basis. 
    For $h \in G$,
    \begin{align*}
    f\left( h \left(  g(X_0^{c_0} \cdots X_d^{c_d}) \otimes g\mathbf y_i  \right)  \right) &= f \left(  (X_{hg(0)}^{c_0} \cdots X_{hg(d)}^{c_d}) \otimes hg\mathbf y_i\right) \\
            &= \sum_j c_j hg(\bv_{i,j}) + hg \mathbf y_i \\
    	&= h \left(\sum_j c_j g (\bv_{i,j}) +g\mathbf y_i\right) \\
    	&= h(f(g(X_0^{c_0} \cdots X_d^{c_d}) \otimes g\mathbf y_i)).
    \end{align*}
    
    The $G$-module isomorphism also respects the grading and yields an equality among characters.
    Identify $ \C[X_0,\dots, X_d]$ and $\sum_{t \geq 0} \sym^t (M' \bigotimes_{\Z} \C)$ as graded $G$ modules as in the discussion around \Cref{lemma:det}. This yields
    \begin{align*}
     \sum_{t \geq 0}\chi_{tP}z^t = \frac{1}{\det(I - z \cdot \rho)}\sum_{i=0}^{d}\chi_{\Bx(\mathcal T^{[\,)})_i}z^i.
    \end{align*}
    \end{proof}
    
    \begin{example}\label{ex:pi3permrep}[\Cref{example:pi3} continued] 
    The two-dimensional permutahedron $\Pi_3 \subseteq \R^3$ under the action of the group $\Z/3\Z$ cyclically permuting the standard basis vectors admits a $G$-invariant half-open decomposition
    as dictated by Theorem~\ref{thm:permrep}.
    This half-open decomposition is described in Figure \ref{fig:exampledecompositions} through a triangulation $\mathcal T$ of $\Pi_3$ and a partition of the face poset of $\mathcal T$ into intervals.
    In this partition, each maximal simplex in an interval is a triangle and every orbit of triangles has order 1 or 3.
    With this half-open decomposition, the lattice points in the union of the fundamental parallelepipeds are
    \begin{align*}
    &\left\{(0,0,0)\right \} & \text{ at height 0,} \\ &\left\{(1,2,3),(2,3,1),(3,1,2),(2,2,2)\right\} & \text{ at height 1,}  \\ &\left \{(4,4,4)\right\} & \text{ at height 2.}
    \end{align*}
    At height 0 and 2, there is a unique $\Z/3\Z$-invariant box point, each giving a copy of the trivial representation.
    At height 1, the character of the permutation representation on the 4 lattice points evaluates to 4 at the identity element, and 1 at the other two group elements.
    Taking the inner product to express this representation in the basis of irreducible representations for $\Z/3\Z$: $\chi_{triv},\chi_{g},\chi_{g^2}$,
    yields the $\Hstar$-polynomial
    $$
    \Hstar(\Pi_3;z) = \chi_{triv} + (2 \chi_{triv} + \chi_{g} + \chi_{g^2})z + \chi_{triv}z^2.
    $$
    \end{example}
    
    The next theorem introduces a second way to compute the equivariant Ehrhart series using $G$-invariant half-open decompositions.
    The conditions required by the theorem ensure that the quasipolynomial $\chi_{tP}$ is actually polynomial in $t$. 
    For a $G$-invariant half-open decomposition $T^{[\,)}$ of a polytope, let
    $\mathcal O$ denote an orbit of 
    intervals of $\mathcal T^{[\,)}$, 
    $\dim(\mathcal O) \coloneqq \dim(\overline S) $ for any interval $I= [\underline S ,\overline S] \in \mathcal O$,
    and $\Bx(\mathcal O)_i \coloneqq \bigcup_{I \in \mathcal O} \Bx(I)_i$.
    Let $\chi_{\Bx(\mathcal O )_i}$ denote the permutation character on the lattice points in $\Bx(\mathcal O)_i$.
    
    \begin{theorem}\label{thm:eeseries2}
    Let $P$ be a $G$-invariant $d$-polytope as in the Setup \ref{setup} with a $G$-invariant half-open decomposition $\mathcal T^{[\,)}$ such that no interval consists of a single $d$-face of the triangulation.
    Furthermore, suppose that for each interval $I=[\underline S, \overline S]$ of $\mathcal T^{[\,)}$ and each $g \in G$, if a box point $\mathbf y\in \Bx(I)$ is fixed by $g$, then the simplex $\overline{S}$ is fixed pointwise by $g$.
    
    The equivariant Ehrhart series has the following rational generating function:
    $$
    \sum_{t \geq 0} \chi_{tP} z^t = \frac{\sum_{\mathcal O \in \mathcal T^{[\,)}} (1-z)^{d-\dim(\mathcal O)} \sum_{i=0}^{\dim(\mathcal O)}\chi_{\Bx(\mathcal O)_i}z^i }{(1-z)^{d+1}}.
    $$
    In this case, $\chi_{tP}$ is a polynomial in $t$ with coefficients in $R(G)$.
    \end{theorem}
    
    \begin{proof}
    We break $\Cone(P)$ into orbits and describe the permutation representation on each piece.
    Let $\mathcal O$ be an orbit of intervals of $\mathcal T^{[\,)}$, and let $g \in G$.
    Furthermore, let $\Cone(\mathcal O) \coloneqq \bigcup_{I \in \mathcal O} \Cone(I)$.
    
    \begin{claim}
    \begin{equation}\label{claim:orbits}
    	\sum_{\bu\, \in \, \Cone (\mathcal O)^g \cap \Z ^{d+1}} z^\bu  = \sum_{I=[\underline S,\overline{S}] \in \mathcal O} \frac{\sum_{\mathbf m \in \Bx(I)^g}z ^{\mathbf m}}{\Pi_{\bs \in \overline{S}} (1-z^\bs)}
    \end{equation}
    \end{claim}
    \begin{proof}[Proof of Claim]
    We first show that the left hand side of Equation \eqref{claim:orbits} is a subset of the right in terms of the lattice points appearing in the exponents of the series expansions.
    Suppose $\bu \in \Cone (\mathcal O)^g \cap \Z^{d+1}$.
    By Proposition \ref{prop:disjoint_union}, there exists a unique interval ${I = [\underline S, \overline S] \in \mathcal O}$ such that 
    $\bu \in \Cone(I)$. 
    Let $\{\bv_1,\dots,\bv_{n+1} \}$ be the vertices of $\overline S$.
    Then we may write $\bu$ uniquely as ${\bu = \sum_{i = 1}^{n+1} (c_i \bv_i)  + \mathbf y}$, with $c_i \in \Z_{\geq 0}\,$ for all $i \in [n+1]$, and $\mathbf y\in \Bx(I)$ by Proposition \ref{prop:conedecomp}.
    As the expression of $\bu$ is unique, $\bu = g(\bu) = \sum_{i = 1}^{n+1}(c_i g (\bv_i)) + g(\mathbf y)$ implies that $\mathbf y= g(\mathbf y)$.
    This implies that the coefficient of $z^{\bu}$ in the series expansion of the right side of Equation \eqref{claim:orbits} is 1.
    
    We now show the right hand side of Equation \eqref{claim:orbits} is a subset of the left. 
    By Proposition \ref{prop:conedecomp}, every lattice point $\bu$ in the series expansion of the right side of Equation \eqref{claim:orbits} as $z^{\bu}$ has coefficient one.
    Furthermore, there exists a unique interval $I = [\underline S, \overline S] \in \mathcal O$ with $\verts(\overline S) = \{\bv_1 , \dots \bv_{n+1} \}$ such that
    $\bu = \sum_{i=1}^{n+1}(c_i \bv_i) +\mathbf y,$
    with $c_i \in \Z_{\geq 0}$ for all $i \in [n+1]$ and $\mathbf y \in \Bx(I)^g$.
    Then,
    \begin{align*}
    	g\left (\sum_{i=1}^{n+1} c_i \bv_i + \mathbf y\right) &= \sum_{i=1}^{n+1} (c_i g(\bv_i)) + g (\mathbf y) \\ 
    			      &= \sum_{i=1}^{n+1}(c_i g(\bv_i)) + \mathbf y     \\
    			      &= \sum_{i=1}^{n+1}(c_i \bv_i ) + \mathbf y\text{\,\, (by assumption)}.  
    \end{align*}
    Thus $\bu \in  \Cone(\mathcal O)^g $ and Equation \eqref{claim:orbits} holds.
    \end{proof}
    Homogenize, sending $z \rightarrow (\mathbf{1} ,z_{d+1})$. 
    Let $n = \dim (\mathcal O)$.
    Then Equation \eqref{claim:orbits} becomes
    $$
    \sum_{t \geq 0}\left |\Cone(\mathcal O)^g \cap (\Z ^{d} \times t) \right| z^t  = \frac{\sum_{k = 0}^{n} \left|\Bx(\mathcal O)_k^g\right| z ^k}{(1-z)^{n+ 1}}.
    $$
    For all $t \in \Z_{\geq 0}$, $\left |\Cone(\mathcal O)^g \cap (\Z ^{d} \times t) \right|$ is equal to the permutation character on the lattice points at height $t$ in $\Cone(\mathcal O)$ evaluated at $g$, which we denote by $\chi_{t\mathcal O}(g)$.
    Likewise, $\left|\Bx(\mathcal O)_k^g\right| $ is equal to $\chi_{\Bx(\mathcal O)_k}(g) $ for all $k$ such that $0 \leq k \leq n$. 
    So in general,  $$
    \sum_{t \geq 0}\chi_{t\mathcal O} z^t  = \frac{\sum_{k = 0}^{n} \chi_{\Bx(\mathcal O)_k} z ^k}{(1-z)^{n+ 1}}.
    $$
    Summing over all the orbits yields,
    \begin{align*}
        \sum_{t \geq 0} \chi_{tP}z^t & =  \sum_{\mathcal O \in T^{[\,) }} \frac{\sum_{k=0}^{\dim(\mathcal O)} \chi_{\Bx(\mathcal O)_k}z^k}{(1-z)^{\dim(\mathcal O) + 1}}\\
        &= \frac{\sum_{\mathcal O \in T^{[\,) }} (1-z)^{d-\dim(\mathcal O)} \chi_{\Bx(\mathcal O)_k}z^k}{(1-z)^{ d + 1}}.
    \end{align*}
    \end{proof}
    
    \begin{example}[Example \ref{ex:pi3permrep} continued]\label{example:pi3eeseries}
    Again consider the two dimensional permutahedron $\Pi_3 \subseteq \R^3$ under the action of $\Z/3\Z$ cyclically permuting the standard basis vectors.
    It is necessary to use a different half-open decomposition from that in Example \ref{ex:pi3permrep} to compute the equivariant Ehrhart series.
    
    \begin{figure}
        \begin{center}
        \input{img/Pi3_triangulation.tikz}
        \end{center}
        \caption{A triangulation of the permutahedron $\Pi_3$ into six triangles.}
        \label{fig:pi3sixtriangles}
    \end{figure}
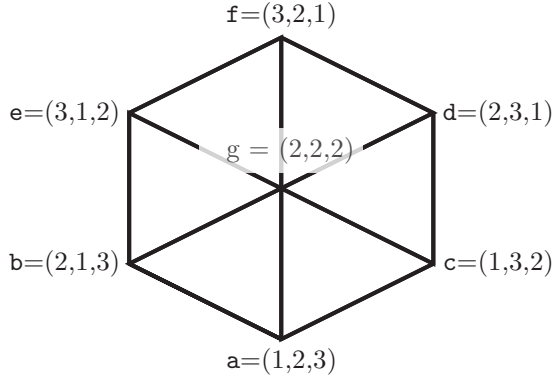
    
    As shown in Figure \ref{fig:pi3sixtriangles}, we triangulate $\Pi_3$ using a barycentric subdivision.
    The intervals we use for a $G$-invariant half-open decomposition are:
    $$[\emptyset,g],[a,abg],[b,beg],[c,acg],[d,cdg],[e,efg],[f,dfg].$$
    The only lattice points that can be box points must have a coordinate sum that is a multiple of 6.
    The box points for each of the respective intervals are:
    $$\{(0,0,0)\},\{(1,2,3) \},\{(2,1,3) \} ,\{ (1,3,2)\},\{(2,3,1) \},\{(3,1,2)\},\{(3,2,1) \}$$
    There are 3 orbits of box points.
    The character table of $\Z/3 \Z$ is as follows, with $\zeta$ a cube root of unity.
        \begin{center}
        \begin{tabular}{c|ccc}
        	$\Z/3\Z$   & $id$ & $g$ & $g^2$ \\
        	\hline
        	$\chi_{triv}$& 1 &  1 & 1 \\
        	$\chi_{g}$   & 1 & $\zeta$ & $\zeta^2$ \\
        	$\chi_{g^2}$ & 1 & $\zeta^2$ & $\zeta$ 
        \end{tabular}
        \end{center}
    We calculate the following rational generating function for the equivariant Ehrhart series using the formula from \Cref{thm:eeseries2}:
    $$
    \sum_{t\geq 0}\chi_{t\Pi_3}z^t = \frac{\chi_{triv}+ 2(\chi_{g}+\chi_{g^2})z+\chi_{triv}z^2   }{(1-z)^3}
    $$
    Evaluating at the group elements yields the Ehrhart series of the fixed subpolytopes.
    Transforming to $\chi_{tP}$ using the usual transformation gives: 
    \begin{align*}
    	\chi_{t\Pi_3} &= \chi_{triv}\binom{t+2}{2} + 2(\chi_{g}+\chi_{g^2})\binom{t+1}{2} + \chi_{triv}\binom{t}{2} \\
    		  &= (\chi_{triv}+\chi_{g}+\chi_{g^2})t^2 + (\chi_{triv}+\chi_{g}+\chi_{g^2})t + \chi_{triv}.
    \end{align*}
    Evaluating at the identity and $g$ yields the Ehrhart polynomials of the fixed subpolytopes, $3t^2+3t +1$ and $1$ respectively.
    Theorem \ref{thm:eeseries2} would need to be modified to allow for rational triangulations in order to be applied to $\Pi_3$ under the action of $S_3$ permuting the standard basis vectors.
    

    \end{example}

%% file: img/Pi3_half_open_decomp.tikz
\begin{tikzpicture}[scale = .6]
	\coordinate[label=below:{\texttt{a}=(1,2,3)}] (a) at (0,0) {};
    \coordinate[label=right:{\texttt{c}=(1,3,2)}] (c) at (2,1) {};
    \node[] (c') at (1.8,1.4) {};
    \node[] (c'') at (1.8,.9) {};
    \coordinate[label=left: {\texttt{b}=(2,1,3)}] (b) at (-2,1) {};
    \node[] (b') at (-1.8,1.4) {};
    \node[] (b'') at (-1.8,.9) {};
    \coordinate[label=left: {\texttt{e}=(3,1,2)}] (e) at (-2,3) {};
    \coordinate[label=right:{\texttt{d}=(2,3,1)}] (d) at (2,3) {};
    \coordinate[label=above:{\texttt{f}=(3,2,1)}] (f) at (0,4) {};
    \node[] (f') at (0.2,3.8) {};
    \node[] (f'') at (-0.2,3.8) {};
    \coordinate[] (g) at (0,2) {};
    
    \draw[color=white ] (b) -- (f) -- (e) -- (b);
    \draw[color=white] (c) -- (f) -- (d) -- (c);
    \draw[color=white ] (a) -- (b) --(c) -- (a);
    \draw[line width=1.5pt,shorten <= 6pt, shorten >= 6pt] (b) -- (a) -- (c) ;
    \draw[line width=1.5pt,shorten <= 6pt, shorten >= 6pt] (b) -- (e) -- (f) ;
    \draw[line width=1.5pt,shorten <= 6pt, shorten >= 6pt] (f) -- (d) -- (c) ;
    \draw[line width=1.5pt,color = black] (b) -- (f) -- (c) -- (b);
\end{tikzpicture}
\begin{tikzpicture} 
    [scale=.8]%
    \Vertex[style={minimum size=1.0cm,draw=black,fill=pink,text=clv0,shape=circle},LabelOut=false,L=\hbox{$\emptyset$},x=0cm,y=0cm]{v0}
    \Vertex[style={minimum size=1.0cm,draw=black,fill=pink,text=clv1,shape=circle},LabelOut=false,L=\hbox{$\text{\texttt{f}}$},x=1cm,y=2cm]{v1}
    \Vertex[style={minimum size=1.0cm,draw=black,fill=pink,text=clv2,shape=circle},LabelOut=false,L=\hbox{$\text{\texttt{e}}$},x=2cm,y=2cm]{v2}
    \Vertex[style={minimum size=1.0cm,draw=black,fill=pink,text=clv3,shape=circle},LabelOut=false,L=\hbox{$\text{\texttt{ef}}$},x=1cm,y=4cm]{v3}
    \Vertex[style={minimum size=1.0cm,draw=black,fill=pink,text=clv4,shape=circle},LabelOut=false,L=\hbox{$\text{\texttt{d}}$},x=3cm,y=2cm]{v4}
    \Vertex[style={minimum size=1.0cm,draw=black,fill=pink,text=clv5,shape=circle},LabelOut=false,L=\hbox{$\text{\texttt{df}}$},x=4cm,y=4cm]{v5}
    \Vertex[style={minimum size=1.0cm,draw=black,fill=pink,text=clv6,shape=circle},LabelOut=false,L=\hbox{$\text{\texttt{cf}}$},x=0cm,y=4cm]{v6}
    \Vertex[style={minimum size=1.0cm,draw=black,fill=pink,text=clv7,shape=circle},LabelOut=false,L=\hbox{$\text{\texttt{c}}$},x=-1cm,y=2cm]{v7}
    \Vertex[style={minimum size=1.0cm,draw=black,fill=pink,text=clv8,shape=circle},LabelOut=false,L=\hbox{$\text{\texttt{cd}}$},x=3cm,y=4cm]{v8}
    \Vertex[style={minimum size=1.0cm,draw=black,fill=pink,text=clv9,shape=circle},LabelOut=false,L=\hbox{$\text{\texttt{cdf}}$},x=3cm,y=6cm]{v9}
    \Vertex[style={minimum size=1.0cm,draw=black,fill=pink,text=clv10,shape=circle},LabelOut=false,L=\hbox{$\text{\texttt{b}}$},x=-2cm,y=2cm]{v10}
    \Vertex[style={minimum size=1.0cm,draw=black,fill=pink,text=clv11,shape=circle},LabelOut=false,L=\hbox{$\text{\texttt{bc}}$},x=-2cm,y=4cm]{v11}
    \Vertex[style={minimum size=1.0cm,draw=black,fill=pink,text=clv12,shape=circle},LabelOut=false,L=\hbox{$\text{\texttt{be}}$},x=2,y=4cm]{v12}
    \Vertex[style={minimum size=1.0cm,draw=black,fill=pink,text=clv13,shape=circle},LabelOut=false,L=\hbox{$\text{\texttt{bf}}$},x=-1,y=4cm]{v13}
    \Vertex[style={minimum size=1.0cm,draw=black,fill=pink,text=clv14,shape=circle},LabelOut=false,L=\hbox{$\text{\texttt{bcf}}$},x=-1cm,y=6cm]{v14}
    \Vertex[style={minimum size=1.0cm,draw=black,fill=pink,text=clv15,shape=circle},LabelOut=false,L=\hbox{$\text{\texttt{bef}}$},x=1cm,y=6cm]{v15}
    \Vertex[style={minimum size=1.0cm,draw=black,fill=pink,text=clv16,shape=circle},LabelOut=false,L=\hbox{$\text{\texttt{a}}$},x=-3cm,y=2cm]{v16}
    \Vertex[style={minimum size=1.0cm,draw=black,fill=pink,text=clv17,shape=circle},LabelOut=false,L=\hbox{$\text{\texttt{ab}}$},x=-4cm,y=4cm]{v17}
    \Vertex[style={minimum size=1.0cm,draw=black,fill=pink,text=clv18,shape=circle},LabelOut=false,L=\hbox{$\text{\texttt{ac}}$},x=-3cm,y=4cm]{v18}
    \Vertex[style={minimum size=1.0cm,draw=black,fill=pink,text=clv19,shape=circle},LabelOut=false,L=\hbox{$\text{\texttt{abc}}$},x=-3cm,y=6.0cm]{v19}
    \Edge[lw=1pt,style={color=black},](v0)(v16)
    \Edge[lw=1pt](v0)(v1)
    \Edge[lw=1pt](v0)(v2)
    \Edge[lw=1pt](v0)(v4)
    \Edge[lw=1pt](v0)(v7)
    \Edge[lw=1pt](v0)(v10)
    \Edge[lw=1pt](v1)(v3)
    \Edge[lw=1pt](v1)(v5)
    \Edge[lw=1pt](v1)(v6)
    \Edge[lw=1pt](v1)(v13)
    \Edge[lw=1pt](v2)(v3)
    \Edge[lw=1pt](v2)(v12)
    \Edge[lw=1pt](v3)(v15)
    \Edge[lw=1pt](v4)(v5)
    \Edge[lw=1pt](v4)(v8)
    \Edge[lw=1pt](v5)(v9)
    \Edge[lw=1pt](v7)(v6)
    \Edge[lw=1pt](v6)(v9)
    \Edge[lw=1pt](v6)(v14)
    \Edge[lw=1pt](v7)(v18)
    \Edge[lw=1pt](v7)(v8)
    \Edge[lw=1pt](v7)(v11)
    \Edge[lw=1pt](v8)(v9)
    \Edge[lw=1pt](v10)(v17)
    \Edge[lw=1pt](v10)(v11)
    \Edge[lw=1pt](v10)(v12)
    \Edge[lw=1pt](v10)(v13)
    \Edge[lw=1pt](v11)(v19)
    \Edge[lw=1pt](v11)(v14)
    \Edge[lw=1pt](v12)(v15)
    \Edge[lw=1pt](v13)(v14)
    \Edge[lw=1pt](v13)(v15)
    \Edge[lw=1pt](v16)(v17)
    \Edge[lw=1pt](v16)(v18)
    \Edge[lw=1pt](v17)(v19)
    \Edge[lw=1pt](v18)(v19)
    \Edge[lw=0.2cm,style={color=purple!80, opacity=.5}](v16)(v17)
    \Edge[lw=0.2cm,style={color=purple!80, opacity=.5}](v16)(v18)
    \Edge[lw=0.2cm,style={color=purple!80, opacity=.5}](v17)(v19)
    \Edge[lw=0.2cm,style={color=purple!80, opacity=.5}](v18)(v19)
    \Edge[lw=0.2cm,style={opacity=.5,color=purple!80},](v0)(v10)
    \Edge[lw=0.2cm,style={opacity=.5,color=purple!80},](v0)(v7)
    \Edge[lw=0.2cm,style={opacity=.5,color=purple!80},](v0)(v1)
    \Edge[lw=0.2cm,style={opacity=.5,color=purple!80},](v7)(v11)
    \Edge[lw=0.2cm,style={opacity=.5,color=purple!80},](v10)(v11)
    \Edge[lw=0.2cm,style={opacity=.5,color=purple!80},](v1)(v13)
    \Edge[lw=0.2cm,style={opacity=.5,color=purple!80},](v10)(v13)
    \Edge[lw=0.2cm,style={opacity=.5,color=purple!80},](v1)(v6)
    \Edge[lw=0.2cm,style={opacity=.5,color=purple!80},](v6)(v7)
    \Edge[lw=0.2cm,style={opacity=.5,color=purple!80},](v6)(v14)
    \Edge[lw=0.2cm,style={opacity=.5,color=purple!80},](v11)(v14)
    \Edge[lw=0.2cm,style={opacity=.5,color=purple!80},](v13)(v14)
    \Edge[lw=0.2cm,style={opacity=.5,color=purple!80},](v2)(v12)
    \Edge[lw=0.2cm,style={opacity=.5,color=purple!80},](v2)(v3)
    \Edge[lw=0.2cm,style={opacity=.5,color=purple!80},](v12)(v15)
    \Edge[lw=0.2cm,style={opacity=.5,color=purple!80},](v3)(v15)
    \Edge[lw=0.2cm,style={opacity=.5,color=purple!80},](v4)(v5)
    \Edge[lw=0.2cm,style={opacity=.5,color=purple!80},](v4)(v8)
    \Edge[lw=0.2cm,style={opacity=.5,color=purple!80},](v5)(v9)
    \Edge[lw=0.2cm,style={opacity=.5,color=purple!80},](v8)(v9)
\end{tikzpicture}

%% file: img/square_half_open_decomp.tikz
\begin{tikzpicture}[scale = 2]
    \coordinate[label=below:{\texttt{b}=(0,0)}] (b) at (0,0) {};
    \coordinate[label=above:{\texttt{a}=(0,1)}] (a) at (0,1) {};
    \coordinate[label=above: {\texttt{c}=(1,1)}] (c) at (1,1) {};
    \coordinate[label=below: {\texttt{d}=(1,0)}] (d) at (1,0) {};
    \coordinate[label=below: {$\,$}] (e) at (0, -.5) {};
    \draw[line width=1.5pt,color = black] (b) -- (c);
    \draw[line width=1.5pt,shorten <= 6pt, shorten >= 6pt] (b) -- (a) -- (c) ;
    \draw[line width=1.5pt,shorten <= 6pt, shorten >= 6pt] (b) -- (d) -- (c) ;
\end{tikzpicture}
\begin{tikzpicture}[scale = .7]
    \Vertex[style={minimum size=1.0cm,draw=black,fill=pink,text=clv0,shape=circle},LabelOut=false,L=\hbox{$\emptyset$},x=0cm,y=0cm]{v0}
    \Vertex[style={minimum size=1.0cm,draw=black,fill=pink,text=clv1,shape=circle},LabelOut=false,L=\hbox{$\text{\texttt{a}}$},x=-3cm,y=2cm]{a}
    \Vertex[style={minimum size=1.0cm,draw=black,fill=pink,text=clv2,shape=circle},LabelOut=false,L=\hbox{$\text{\texttt{b}}$},x=-1cm,y=2cm]{b}
    \Vertex[style={minimum size=1.0cm,draw=black,fill=pink,text=clv3,shape=circle},LabelOut=false,L=\hbox{$\text{\texttt{c}}$},x=1cm,y=2cm]{c}
    \Vertex[style={minimum size=1.0cm,draw=black,fill=pink,text=clv4,shape=circle},LabelOut=false,L=\hbox{$\text{\texttt{d}}$},x=3cm,y=2cm]{d}
    \Vertex[style={minimum size=1.0cm,draw=black,fill=pink,text=clv5,shape=circle},LabelOut=false,L=\hbox{$\text{\texttt{ab}}$},x=-4cm,y=4cm]{ab}
    \Vertex[style={minimum size=1.0cm,draw=black,fill=pink,text=clv6,shape=circle},LabelOut=false,L=\hbox{$\text{\texttt{ac}}$},x=-2cm,y=4cm]{ac}
    \Vertex[style={minimum size=1.0cm,draw=black,fill=pink,text=clv7,shape=circle},LabelOut=false,L=\hbox{$\text{\texttt{bc}}$},x=0cm,y=4cm]{bc}
    \Vertex[style={minimum size=1.0cm,draw=black,fill=pink,text=clv8,shape=circle},LabelOut=false,L=\hbox{$\text{\texttt{bd}}$},x=2cm,y=4cm]{bd}
    \Vertex[style={minimum size=1.0cm,draw=black,fill=pink,text=clv9,shape=circle},LabelOut=false,L=\hbox{$\text{\texttt{cd}}$},x=4cm,y=4cm]{cd}
    \Vertex[style={minimum size=1.0cm,draw=black,fill=pink,text=clv10,shape=circle},LabelOut=false,L=\hbox{$\text{\texttt{abc}}$},x=-3cm,y=6cm]{abc}
    \Vertex[style={minimum size=1.0cm,draw=black,fill=pink,text=clv11,shape=circle},LabelOut=false,L=\hbox{$\text{\texttt{bcd}}$},x=3cm,y=6cm]{bcd}
    \Edge[lw=1pt,style={color=black},](v0)(a)
    \Edge[lw=1pt](v0)(b)
    \Edge[lw=1pt](v0)(c)
    \Edge[lw=1pt](v0)(d)
    \Edge[lw=1pt](a)(ab)
    \Edge[lw=1pt](a)(ac)
    \Edge[lw=1pt](b)(ab)
    \Edge[lw=1pt](b)(bc)
    \Edge[lw=1pt](b)(bd)
    \Edge[lw=1pt](c)(ac)
    \Edge[lw=1pt](c)(bc)
    \Edge[lw=1pt](c)(cd)
    \Edge[lw=1pt](d)(bc)
    \Edge[lw=1pt](d)(bd)
    \Edge[lw=1pt](d)(cd)
    \Edge[lw=1pt](ab)(abc)
    \Edge[lw=1pt](ac)(abc)
    \Edge[lw=1pt](bc)(abc)
    \Edge[lw=1pt](bc)(bcd)
    \Edge[lw=1pt](bd)(bcd)
    \Edge[lw=1pt](cd)(bcd)
    \Edge[lw=0.2cm,style={color=purple!80, opacity=.5}](a)(ab)
    \Edge[lw=0.2cm,style={color=purple!80, opacity=.5}](a)(ac)
    \Edge[lw=0.2cm,style={color=purple!80, opacity=.5}](ab)(abc)
    \Edge[lw=0.2cm,style={color=purple!80, opacity=.5}](ac)(abc)
    \Edge[lw=0.2cm,style={opacity=.5,color=purple!80},](v0)(b)
    \Edge[lw=0.2cm,style={opacity=.5,color=purple!80},](v0)(c)
    \Edge[lw=0.2cm,style={opacity=.5,color=purple!80},](b)(bc)
    \Edge[lw=0.2cm,style={opacity=.5,color=purple!80},](c)(bc)
    \Edge[lw=0.2cm,style={opacity=.5,color=purple!80},](d)(bd)
    \Edge[lw=0.2cm,style={opacity=.5,color=purple!80},](d)(cd)
    \Edge[lw=0.2cm,style={opacity=.5,color=purple!80},](bd)(bcd)
    \Edge[lw=0.2cm,style={opacity=.5,color=purple!80},](cd)(bcd)
\end{tikzpicture}

%% file: img/Pi3_triangulation.tikz
\begin{tikzpicture}
    \coordinate[label=below:{\texttt{a}=(1,2,3)}] (a) at (0,0) {};
    \coordinate[label=right:{\texttt{c}=(1,3,2)}] (c) at (2,1) {};
    \node[] (c') at (1.8,1.4) {};
    \node[] (c'') at (1.8,.9) {};
    \coordinate[label=left: {\texttt{b}=(2,1,3)}] (b) at (-2,1) {};
    \node[] (b') at (-1.8,1.4) {};
    \node[] (b'') at (-1.8,.9) {};
    \coordinate[label=left: {\texttt{e}=(3,1,2)}] (e) at (-2,3) {};
    \coordinate[label=right:{\texttt{d}=(2,3,1)}] (d) at (2,3) {};
    \coordinate[label=above:{\texttt{f}=(3,2,1)}] (f) at (0,4) {};
    \node[] (f') at (0.2,3.8) {};
    \node[] (f'') at (-0.2,3.8) {};
    \coordinate[] (g) at (0,2) {};
    \draw[line width=1.5pt,color=black,  ] (a) -- (b) -- (g) -- (a);
    \draw[line width=1.5pt,color=black,  ] (b) -- (e) -- (g) -- (b);
    \draw[line width=1.5pt,color=black,  ] (e) -- (f) --(g) -- (e);
    \draw[line width=1.5pt,color = black, ] (f) -- (d) -- (g) -- (f);
    \draw[line width=1.5pt,color = black, ] (d) --(c) --(g) -- (d) ;
    \draw[line width=1.5pt,color =black, ] (c) -- (a) -- (g) -- (c) ;
    \draw[line width=1.5pt, color=black] (a) -- (b) -- (e) -- (f) -- (d) -- (c) -- (a) -- (b);
    
    \node[rectangle,fill = white,opacity=.8, text opacity=1] (g') at (0,2.5) {g = (2,2,2)};
\end{tikzpicture}

%% file: 3.3_breaking_hstar.tex
    \subsection{Breaking Down the H*-series}\label{sec:breaking Hstar} 
  
    In this section, we consider the hypersimplex and the (prime) permutahedra under the cyclic group action and provide an explicit formulas for the equivariant Ehrhart series.
        
          \subsubsection{Cyclic Group Action on Hypersimplex}
        We give a combinatorial interpretation (Theorem \ref{thm: hypersimplex cyclic group}) for the $\Hstar$-series of the hypersimplex $\Delta_{k,n}$ under the group $\Z/n\Z$ where the action is given by cyclically shifting coordinates. We build upon \cite{Katz} to compute the $\Hstar$-series and then follow the arguments in \cite{Kim} to give an equivariant generalization. Before stating the main theorem, we give several definitions.
        
         For two integers $0<k<n$, the $(k,n)$\emph{-th hypersimplex} is defined to be 
          $$\Delta_{k,n}=\{(x_1,\cdots,x_n) \in \mathbb{R}^n \mid 0 \leq x_i \leq 1,   x_1+\cdots+x_n=k\}.$$
         It is an $(n-1)$-dimensional polytope inside $\mathbb{R}^n$ whose vertices are (0,1)-vectors with exactly $k$ 1's. In particular, it has $\binom{n}{k}$ many vertices.

         A \emph{decorated ordered set partition} $((L_1)_{l_1},\cdots,(L_m)_{l_m})$ of type $(k,n)$ consists of an ordered partition $(L_1,\cdots,L_m)$ of $\{1,2,...,n\}$ and an $m$-tuple $(l_1,\cdots,l_m) \in \mathbb{Z}^m$ such that $l_1+\cdots+l_m=k$ and $l_i \geq 1$. We call each $L_i$ a $block$ and we place them on a circle in a clockwise fashion and then think of $l_i$ as the clockwise distance between adjacent blocks $L_i$ and $L_{i+1}$ where indices are considered modulo $m$ (see Figure \ref{fig1}). So the circumference of the circle is $l_1+\cdots+l_m=k$. We regard decorated ordered set partitions up to cyclic rotation of blocks (together with corresponding $l_i$). 
         For example, the decorated ordered set partition $(\{1,2,7\}_2,\{3,5\}_3,\{4,6\}_1)$ is the same as $(\{3,5\}_3,\{4,6\}_1,\{1,2,7\}_2)$. A decorated ordered set partition is called \emph{hypersimplicial} if it satisfies $1 \leq l_i \leq |L_i|-1$ for all $i$.

        Given a decorated ordered set partition, we define the \emph{winding vector} and the \emph{winding number}. To define the winding vector, let $w_i$ be the distance of the path starting from the block containing $i$ to the block containing $(i+1)$ moving clockwise (where $i$ and $(i+1)$ are considered modulo $n$). If $i$ and $(i+1)$ are in the same block then $w_i=0$. In Figure \ref{fig1}, the winding vector is $w=(0,2,3,3,3,1,0)$. The total length of the path is $(w_1+\cdots+w_n)$, which should be a multiple of $k$ as we started from 1 and came back to 1 moving clockwise. If $(w_1+\cdots+w_n)=kd$, then we define the winding number to be $d$. In Figure \ref{fig1}, the winding number is 2.
        
        We define a cyclic group $\Z/n\Z$ action on the set of decorated ordered set partitions of type $(k,n)$ by cyclically shifting elements in their parts. It is straightforward to check that the $\Z/n\Z$ action cyclically shifts the winding vectors so the action preserves the winding number.
        
        For some $\sigma\in \Z/n\Z$, a decorated ordered set partition is called $\sigma$-fixed if it is invariant under the action of $\sigma$. Note that a decorated ordered set partition is $\sigma$-fixed if and only if it has $\sigma$-fixed winding vector. For example, consider $\sigma\in \Z/4\Z$  defined by $\sigma(i)=i+1$. Then $\sigma$ sends $(\{1,3\}_1,\{2,4\}_1)$ to $(\{2,4\}_1,\{1,3\}_1)$ which is identified with $(\{1,3\}_1,\{2,4\}_1)$. Therefore $(\{1,3\}_1,\{2,4\}_1)$ is $\sigma$-fixed. However, $(\{1,3\}_2,\{2,4\}_1)$ is not $\sigma$-fixed as it goes to $(\{1,3\}_1,\{2,4\}_2)$. Note that $(\{1,3\}_2,\{2,4\}_1)$ is  $\sigma^2$-fixed. 
        
        \begin{example} \label{ex1} 
        Consider a decorated ordered set partition $(\{1,2,7\}_2,\{3,5\}_3,\{4,6\}_1)$ of type (6,7) and $\sigma\in\mathbb{Z}/7\mathbb{Z}$ defined by $\sigma(i)=i+1$. Then $\sigma$ sends $(\{1,2,7\}_2,\{3,5\}_3,\{4,6\}_1)$ to $(\{1,2,3\}_2,\{4,6\}_3,\{5,7\}_1)$. Both are not hypersimplicial as $3>|\{3,5\}|-1$.
        
        The winding vector of the decorated ordered set partition $(\{1,2,7\}_2,\{3,5\}_3,\{4,6\}_1)$  is $(0,2,3,3,3,1,0)$ and the winding vector of $(\{1,2,3\}_2,\{4,6\}_3,\{5,7\}_1)$  is $(0,0,2,3,3,3,1)$ which is obtained by applying $\sigma$ to $(0,2,3,3,3,1,0)$.
        \end{example}
        
        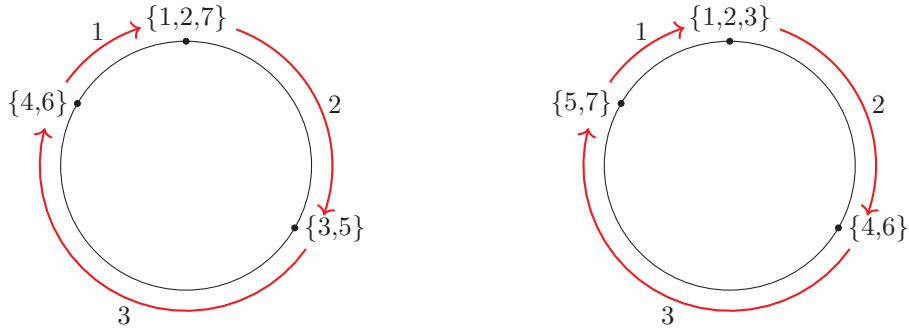
\begin{figure}[ht] 
            \centering
            \input{img/decorated_osp_example.tikz}
            \caption{The figure on the left is the picture associated to the decorated ordered set partition $(\{1,2,7\}_2,\{3,5\}_3,\{4,6\}_1)$ and the figure on the right is the picture after acting by $\sigma$. }
      
            \label{fig1}
        \end{figure}

        Now we state our main theorem in this section.

        \begin{theorem}\label{thm: hypersimplex cyclic group}
        The character $\Hstar_d(\Delta_{k,n};z)$ is the permutation character on the set of decorated ordered set partitions of type $(k,n)$ with winding number $d$.
         \end{theorem}

        The proof of \Cref{thm: hypersimplex cyclic group} will be given as follows.
        For a fixed element $\sigma$ in $\mathbb{Z}/n\mathbb{Z}$, we will first give an explicit formula for $\Hstar(\Delta_{k,n};z)(\sigma)$ (see \eqref{eq:Hstar hypersimplex}). 
        Then we enumerate $\sigma$-fixed hypersimplicial decorated ordered set partitions with a certain winding number.
        
        From now on we fix $\sigma \in \mathbb{Z}/n\mathbb{Z}$. We denote the order of $\sigma$ with $b$ and letting $a=\frac{n}{b}$, any $\sigma$-fixed vector in $\mathbb{R}^n$ is of the form 
        \begin{equation*}
            (x_1,\cdots,x_a,x_{1},\cdots,x_a,\cdots,x_{1},\cdots,x_a).
        \end{equation*}
        
        Note that 
       $$d\Delta_{k,n}=\left \{(x_1,\cdots,x_n) \in \mathbb{R}^n \mid 0 \leq x_i \leq d,   x_1+\cdots+x_n=d k\right \},$$
        so the number of $\sigma$-fixed lattice points of $d\Delta_{k,n}$ equals the number of integer solutions of 
        \begin{align}\label{eq: hypersimplex fixed points}
            0 \leq x_i \leq d, \qquad
            b\left(\sum_{i=1}^{a}x_i\right)=dk.
        \end{align}
        We denote 
        \begin{align}\label{eq:notation b1}\nonumber
            &h=gcd(b,k)\\ &b=hb_1, \qquad k=h k_1.
        \end{align} Then \eqref{eq: hypersimplex fixed points} has an integer solution if and only if $d$ is a multiple of $b_1$.
        Suppose so, and let $d=d'b_1$, then \eqref{eq: hypersimplex fixed points} becomes
        \begin{equation*}
             0 \leq x_i \leq d'b_1, \qquad
            \sum_{i=1}^{a}x_i=d' k_1. 
        \end{equation*}
       By \cite[Theorem 2.8]{Katz},
       
        \begin{equation} \label{eq: hilbert series veronese} \EE(\Delta_{k,n};z)(\sigma)=\frac{\sum\limits_{i \geq 0} (-1)^{i}\binom{a}{i} \left(\sum\limits_{j\geq0} \binom{i}{j} (z^{b_1}-1)^j \left(\sum\limits_{\ell \geq 0} \binom{a-j}{\ell(k_1-b_1i)}_{\hspace{-1mm}k_1-b_1i} (z^{b_1})^{\ell}\right)\right)}{(1-z^{b_1})^a} 
           \end{equation}
        where the notation $\binom{n}{k}_{p}$ means the coefficient of $x^k$ in $(1+x+\dots+x^{p-1})^n$. In \cite[Section 2.1]{Kim}, the second author showed that \eqref{eq: hilbert series veronese} can be simplified as follows:
        
        \begin{equation*}  \EE(\Delta_{k,n};z)[\sigma]= \frac{\sum\limits_{i \geq 0} (-1)^{i}\binom{a}{i} \left(\sum\limits_{\ell\geq0} \binom{a}{\ell(k_1-b_1 i)-i)}_{k_1-b_1 i}(z^{b_1})^{\ell} \right)} {(1-z^{b_1})^a}.
           \end{equation*}
           Since $\det (I - z \cdot \rho(\sigma))=(1-z^b)^a$, we conclude
          \begin{align}  \nonumber \Hstar(\Delta_{k,n};z)(\sigma)&=\det (I - z \cdot \rho(\sigma))\EE(\Delta_{k,n};z)(\sigma)\\
          \nonumber&= \left(\frac{1-z^{h b_1}}{1-z^{b1}}\right)^a\sum\limits_{i \geq 0} (-1)^{i}\binom{a}{i} \left(\sum\limits_{\ell\geq0} \binom{a}{\ell(k_1-b_1 i)-i)}_{k_1-b_1 i}(z^{b_1})^{\ell} \right)\\
          &=\left(\sum_{\ell\geq 0}\binom{a}{\ell}_{h}(z^{b_1})^{\ell}\right)\sum\limits_{i \geq 0} (-1)^{i}\binom{a}{i} \left(\sum\limits_{\ell\geq0} \binom{a}{\ell(k_1-b_1 i)-i)}_{k_1-b_1 i}(z^{b_1})^{\ell} \right)\label{eq: hilbert series veronese2}
           \end{align}
          
          \begin{remark}
          The sums in \eqref{eq: hilbert series veronese} and \eqref{eq: hilbert series veronese2} are finite sums since $\binom{n}{k}_{p}$ is zero if $k$ is greater than $n(p-1)$.
          \end{remark} 
           To simplify \eqref{eq: hilbert series veronese2} we need the following lemma.
          
           \begin{lemma}\label{lemma: binomial}
           For positive integers $a, h, i$ and $p$, we have
     
           \begin{equation*}
               \left(\sum_{\ell \geq 0} \binom{a}{\ell}_h z^{\ell}\right)\left(\sum_{\ell \geq 0} \binom{a}{\ell p-i}_{p} z^{\ell}\right)=\sum_{\ell \geq 0}  \binom{a}{\ell p-i}_{h p} z^{\ell}.
           \end{equation*}
           \end{lemma}
           \begin{proof}
           The number $\binom{a}{\ell p-i}_{h p}$ counts the number of integer solutions to equations 
           \begin{equation*}
              0\leq x_j\leq h p -1, \qquad x_1+\dots+x_a=\ell p -i.
           \end{equation*}
         
           Given a solution $x_1,\dots,x_a$, we write each $x_j$ as
           \begin{equation*}
               x_j=p y_j +z_j, \qquad    0\leq z_j\leq p-1.
           \end{equation*}
           Note that $0\leq y_j\leq h-1$. If $y_1+\dots+y_a=\ell'$ then we have
           \begin{equation*}
              z_1+\dots+z_a=(\ell-\ell')p-i.
           \end{equation*}
           For a fixed $\ell'$, the number of such integer solutions for $(y_1,\dots,y_a,z_1,\dots,z_a)$ equals $\binom{a}{\ell'}_h \binom{a}{(\ell-\ell')p-i}_p$. This gives
           \begin{equation*}
              \sum_{\ell'\geq 0} \binom{a}{\ell'}_h \binom{a}{(\ell-\ell')p-i}_p= \binom{a}{\ell p-i}_{h p}
           \end{equation*}
           which completes the proof.
           \end{proof}
       
        Applying Lemma \ref{lemma: binomial} to \eqref{eq: hilbert series veronese2}, we conclude
        \begin{equation}\label{eq:Hstar hypersimplex}
             \Hstar(\Delta_{k,n};z)(\sigma)=\sum\limits_{i \geq 0} (-1)^{i}\binom{a}{i} \left(\sum\limits_{\ell\geq0} \binom{a}{\ell(k_1-b_1 i)-i)}_{k-b i}(z^{b_1})^{\ell} \right).
        \end{equation}
        
        Now we will count $\sigma$-fixed hypersimplicial decorated ordered set partitions of type $(k,n)$ with winding number $d$.
        The argument follows the structure of the argument in \cite[Section 3]{Kim}, but now we must consider $\sigma$-fixed decorated ordered set partitions. Here we briefly summarize the argument. As before we fix an element $\sigma \in \mathbb{Z}/n\mathbb{Z}$ and denote the order of $\sigma$ with $b$ and set $a=\frac{n}{b}$.
      
        \begin{definition}
        For a decorated ordered set partition $L=((L_1)_{l_1},(L_2)_{l_2},...,(L_m)_{l_m})$, a block $L_i$ is \emph{bad} if $l_i \geq |L_i|$. Let $I (L)=\{L_i \mid L_i \text{ is bad}\}$.
        \end{definition}
        
        Since hypersimplicial decorated ordered set partitions satisfy $1 \leq l_i \leq |L_i|-1$ for all blocks, a decorated ordered set partition is hypersimplicial if and only if $I (L)$ is empty.
        
        \begin{definition}
        For a subset $T \subseteq \{1,2,\cdots,n\} $ that is $\sigma$-fixed, define $\up(T)$ to be a set of all (unordered) partitions of $T$ that are $\sigma$-fixed. For each $S\in \up(T)$, we define $S(\sigma)$ to be the number of orbits of blocks in $S$ under the action of $\sigma$. For example, consider $n=6$ and $\sigma\in \Z/6\Z$ defined by $\sigma(i)=i+3$. Then $T=\{1,2,4,5\}$ is $\sigma$-fixed and $\up(T)$ has $7$ elements which are $\{\{1,2,4,5\}\}$, $\{\{1,2\},\{4,5\}\}$, $\{\{1,4\},\{2,5\}\}$, $\{\{1,5\},\{2,4\}\}$, $\{\{1,4\},\{2\},\{5\}\}$, $\{\{1\},\{4\},\{2,5\}\}$, and $\{\{1\},\{2\},\{4\},\{5\}\}$. Note that
        \begin{align*}
            \{\{1\},\{2\},\{4\},\{5\}\}(\sigma) &= 2,\\
             \{\{1,2\},\{4,5\}\}(\sigma) &= 1.
        \end{align*}
        \end{definition}
        
        \begin{definition}
        For $\sigma$-fixed $T\subseteq \{1,2,\cdots,n\}$ and $S\in \up(T)$,
        define $K_{k,n}^{d} (S)$ to be the set of all $\sigma$-fixed decorated ordered set partitions of type $(k,n)$ with winding number $d$ such that $S \subseteq I (L)$.
        \end{definition}

        \begin{definition}
        For $\sigma$-fixed $T \subseteq \{1,2,\cdots,n\}$, let $H_{k,n}^{d} (T)=\sum\limits_{S \in \up(T)} (-1)^{S(
        \sigma)} |K_{k,n}^{d} (S)|.$
        \end{definition}

        \begin{proposition} \label{prop: hypersimplicial dop number} The number of $\sigma$-fixed hypersimplicial decorated ordered set partitions of type $(k,n)$ with winding number $d$
        is
        \begin{equation}\label{eq: dop counting}\sum_{\substack{T \subseteq \{1,2,...,n\}\\ T: \sigma-fixed}} H_{k,n}^{d} (T).\end{equation}
        \end{proposition}
        \begin{proof}
        We have
        \begin{equation*}
            \sum_{\substack{T \subseteq \{1,2,...,n\}\\ T: \sigma-fixed}} H_{k,n}^{d} (T)=\sum\limits_{\substack{T \subseteq \{1,2,...,n\}\\ T: \sigma-fixed}} \left(\sum\limits_{S \in \up(T)}(-1)^{S(\sigma)} \left |K_{k,n}^{d} (S)\right| \right).
        \end{equation*}
        A decorated ordered set partition $L$ belongs to $K^{d}_{k,n} (S)$ if and only if $S$ is a subset of $I(L)$.
        If $I (L)$ is empty then $L$ will be counted once when $S=\varnothing$.
        Now assume $I (L)$ is non empty. Note that $I(L)$ is $\sigma$-fixed; we denote the number of $\sigma$-orbits in $I(L)$ with $i$. The number of $\sigma$-fixed $S\subseteq I(L)$ such that $|S|=j$ equals $\binom{i}{j}$, therefore $L$ will be counted $\binom{i}{j}$ times with the sign $(-1)^j$ as $S$ ranges over all $\sigma$-fixed subsets of $I(L)$ with $S(\sigma)=j$. We conclude that the contribution of $L$ to \eqref{eq: dop counting} is $\sum_{j=0}^{m}(-1)^j \binom{i}{j}=0$. So \eqref{eq: dop counting} equals the number of $\sigma$-fixed hypersimplicial decorated ordered set partitions of type $(k,n)$ with winding number $d$.
        \end{proof}

        \begin{proposition} \label{proposition: second winding bector}
        Given a $\sigma$-fixed $T\subseteq\{1,2,\cdots,n\}$ such that $|T|=bi$, the value 
        $(-1)^{i}H_{k,n}^{d} (T)$ equals the number of $\sigma$-fixed (integer) vectors $(v_1,v_2,\cdots,v_n)$ satisfying
        
        \begin{align}\label{eq: condition1}
        0 \leq v_j \leq k-bi-1 \qquad &\text{if} \qquad j \notin T,\\
        \label{eq: condition2}
        1 \leq v_j \leq k-b i \qquad &\text{if} \qquad j \in T, \\
        \label{eq: condition3}
        v_1+\cdots+v_n=(k-b i)d.&
        \end{align}
         \end{proposition}
         \begin{proof}
         This can be proved by the same method in the proof of \cite[Proposition~2.22]{Kim}
         \end{proof}

        \begin{corollary}\label{cor: Hknd count}
        Given a $\sigma$-fixed $T\subseteq\{1,2,\cdots,n\}$ such that $|T|=bi$, we have 
        \begin{align*}
        H_{k,n}^{d} (T)=\begin{cases}(-1)^i\binom{a}{\frac{d}{b_1}(k_1-b_1 i)-i }_{k-bi} \qquad &\text{if} \qquad \text{$d$ is a multiple of $b_1$}\\
        0 \qquad &\text{otherwise}.
        \end{cases}
        \end{align*}
        Recall that $b_1=\frac{b}{gcd(b,k)}$ as in \eqref{eq:notation b1}.
        \end{corollary}
        
        \begin{proof}
        For a $\sigma$-fixed vector $v=(v_1,\cdots,v_n)$ satisfying \eqref{eq: condition1}, \eqref{eq: condition2} and \eqref{eq: condition3}, let $v'=(v'_1,\cdots,v'_n)$ be a vector such that $v'_j=v_j$ if $j \notin T$, and $v'_j=v_j-1$ if $j \in T$. Then $v'$ is also $\sigma$-fixed so we write \begin{equation*}
            v'=(v'_1,\cdots,v'_a,v'_{1},\cdots,v'_a,\cdots,v'_{1},\cdots,v'_a).
        \end{equation*}
        Property \eqref{eq: condition1} and \eqref{eq: condition2} becomes $0 \leq v'_j \leq (k-bi-1)$  for all $j$ and the property \eqref{eq: condition3} becomes 
        \begin{align*}
            b\left(\sum_{j=1}^{a}v'_j\right)=(k-bi)d-bi
            \qquad \rightarrow \qquad \sum_{j=1}^{a}v'_j=\frac{k_1 d}{b_1}-i(d+1).
        \end{align*}
        
        Thus the solution exists only if $d$ is a multiple of $b_1$ and in that case the number of such 
          $v'$ is $\binom{a}{\frac{d}{b_1}(k_1-b_1 i )-i }_{k-bi}$.

        \end{proof}

        \begin{proof}[Proof of \Cref{thm: hypersimplex cyclic group}]
        By \Cref{prop: hypersimplicial dop number} and \Cref{cor: Hknd count}, the number of $\sigma$-fixed hypersimplicial decorated ordered set partitions (of type $(k,n)$ with winding number $d$) is 
            \begin{align*}
            \sum_{\substack{T \subseteq \{1,2,...,n\}\\ T: \sigma-fixed}} H_{k,n}^{d} (T)&=\sum_{i\geq0}\left(\sum_{\substack{|T|=bi\\ T: \sigma-fixed}} H_{k,n}^{d} (T)\right)\\&=\begin{cases}
            \sum_{i\geq0} (-1)^i\binom{a}{i}\binom{a}{\frac{d}{b_1}(k_1-b_1 i)-i }_{k-bi}\qquad &\text{if } \text{$d$ is a multiple of $b_1$,}\\
        0 \qquad &\text{otherwise}.
            \end{cases}
            \end{align*}
            Comparing with \eqref{eq:Hstar hypersimplex} completes the proof.
        \end{proof}
        \begin{example}
        Consider hypersimplicial decorated ordered set partitions of type $(2,4)$:
        \begin{align*}
            &\text{winding number $0$}: \{\{1,2,3,4\}_2\}\\
            &\text{winding number $1$}:  \{\{1,2\}_1\{,3,4\}_1\}, \{\{1,4\}_1\{,2,3\}_1\}\\
            &\text{winding number $2$}: \{\{1,3\}_1\{,2,4\}_1\}.
        \end{align*}
        For $\sigma\in \mathbb{Z}/4\mathbb{Z}$ defined by $\sigma(i)=i+2$, fixed subpolytope $\Delta_{2,4}^{\sigma}$ is the line segment with vertices $(1,0,1,0)$ and $(0,1,0,1)$. The number of lattice points in $t \Delta_{2,4}^{\sigma}$ equals $(t+1)$, thus we conclude
        \begin{align*}\Ehr(\Delta_{2,4}^{\sigma};z)=\sum_{t\geq 0}(t+1)z^t=\frac{1}{(1-z)^2}=\frac{1+2z+z^2}{(1-z^2)^2}\\
        \rightarrow \Hstar(\Delta_{2,4};z)(\sigma)=1+2z+z^2.\end{align*}
       Every hypersimplicial decorated ordered set partition of type $(2,4)$ is $\sigma$-fixed, which gives $\Hstar(\Delta_{2,4};z)(\sigma)$ equals $(1+2z+z^2)$ by Theorem \ref{thm: hypersimplex cyclic group}.
        
        \end{example}

        \subsubsection{Cyclic Group Action on Prime Permutahedra}
   

        \begin{lemma}\label{lem:cyclicactionpipolynomial}
        The $\Hstar$-series of the permutahedron $\Pi_n$ under the action of $\Z/n\Z$ permuting the standard basis vectors is polynomial.
        \end{lemma}
        \begin{proof}
        By Lemma 5.2 of \cite{ASV},
        for $\sigma \in S_n$ with cycle type $\lambda = (\ell_1, \dots, \ell_m)$, $\Hstar(\Pi_n;z)$ is polynomial if and only if the number of even parts in $\lambda$ is 0, $m$, or $m-1$.
        For any element $\sigma$ of $\Z/n\Z$, all cycles of $\sigma$ have the same length, so $\Hstar(\Pi_n;z)$ is polynomial, and thus $\Hstar$ itself is a polynomial.
        \end{proof}

        We now focus on the case when $p$ is prime and $\Z/p\Z$ acts on $\Pi_p$.

    
        \begin{theorem}\label{thm:1modp}
        The $h^*$-polynomial of the permutahedron $\Pi_p$ for prime $p$ has coefficients equivalent to $1 \mod p$.
        \end{theorem}
        
        \begin{proof}
        If $p=2$, then $h^*(z) = 1$. Let $p > 2$.
        Let $\mathcal{I}$ be the set of all nonempty linearly independent subsets of $\{e_j-e_k:1\le j<k\le p\}$, and let $T$ be a linearly independent subset in $\mathcal{I}$.
        The elements of $\mathcal{I}$ correspond to forests on $p$ labeled vertices (excluding the graph with no edges) \cite[Lemma 9.6]{beck_robins}.
        Associate to each $T$ the half-open parallelepiped $\hobox T = \sum_{e_j - e_k \in T}$ ($0,e_j-e_k$].
        As discussed in \cite[Chapter 9]{beck_robins}, the $(p-1)$-dimensional permutahedron $\Pi_p$ is equal to the disjoint union of translates of the parallelotopes:
            $$
            \Pi_p = \{0\} \cup \bigcup_{T \in \mathcal I} \left( \sum_{e_j - e_k \in T} ( 0,e_j-e_k ] \right).
            $$
        Let $\sigma = (1\,2\,\dots\,p)$ be a generator of $\Z/p\Z$.
        The group $\Z/p\Z$ acts on the independent sets $\mathcal{I}$ by cyclically permuting the forest vertex labels.
        Each orbit has size $p$, as follows.
        For $1\leq k<p$, suppose $\sigma^k T = T \in \mathcal I$.
        Then for each edge $(i,j)$ in the forest corresponding to $T$, the edge $(|i+k|_p,|j+k|_p)$ is also in the forest.
        As $k$ is coprime to $p$, the forest thus contains $p$ edges and therefore a cycle, which is a contradiction. Thus $\sigma^k$ does not fix any $T \in \mathcal I$, and $|\Z/p\Z T | = p$. 
        Let $\mathcal S \subset \mathcal I$ be a transversal containing one representative from each orbit in $\mathcal{I}$.
        Each half-open parallelepiped in an orbit has the same Ehrhart series, and is translated by an integral vector. 
        Thus,
            \begin{align*} 
            \Ehr_{\Pi_p}(z) &= \frac{1}{1-z}+ p \sum_{T \in \mathcal S } \Ehr(\hobox T;z) \\
             &=\frac{(1-z)^{p-1}}{(1-z)^{p}}+ \sum_{T\in \mathcal S} \frac{p\cdot h^*(\hobox T;z)(1-z)^{p-|T|-1 }}{(1-z)^{p}}
             \end{align*}
    
        The coefficients of $(1-z)^{p-1}$ are equivalent to one mod $p$, because the coefficient of $z^i$ is $\binom{p-1}{i}(-1)^i$, and 
        \begin{equation*}
            \binom{p-1}{i}=\frac{(p-1)\cdots(p-i)}{i!} \equiv \frac{(-1)\cdots(-i)}{i!} \mod p \equiv \frac{(-1)^i i!}{i!} \mod p \equiv (-1)^i \mod p.
        \end{equation*}
        Hence the claim follows.
        \end{proof}
        
        \begin{example}
        For the permutahedron $\Pi_3$ we have $ \mathcal S = \{ \{e_1 - e_2\} , \{ e_1 - e_2, e_2 - e_3\}\}$.
        Thus 
            $$
            \Ehr_{\Pi_3}(z) = \frac{ 1- 2z+ z^2 }{(1-z)^3}+ \frac{3t(1-z)}{(1-z)^3} + \frac{3(z+z^2)}{(1-z)^3}
            = \frac{1 + 4z + z^2}{(1-z)^3}.
            $$
        \end{example}


     \begin{theorem}\label{thm:cyclicaction_trivialfixed}
        Let $P \subset \R^n$ be an $(n-1)$-dimensional lattice polytope invariant under the action of the cyclic group $\Z/n\Z $ permuting the coordinates of $\R^n$.
        Furthermore, suppose that for each non-identity element $g \in \Z/n\Z$, $P^g$ is a single lattice point.
        Then,
        $$
        \Hstar(P;z) = \sum_{i=0}^{n-1} \left( \chi_0 + \frac{h_i^* -1}{n} \sum \chi_{j} \right) z^i,
        $$
        where $h_i$ is the coefficient of $z^i$ in the $h^*$-polynomial of $P$, and $\{\chi_0,\dots,\chi_{n-1} \}$ are the irreducible representations of $\Z/n\Z$.
        \end{theorem}
        \begin{proof}
        It is enough to show that our formula for $\Hstar(P;z)$ specializes to the fixed Ehrhart series for each conjugacy class.
        Let $\zeta$ be an $n$-th root of unity. 
        Then 
        $$
        \sum_{i=0}^{n-1}\zeta ^ i = \frac{1-\zeta^n}{1-\zeta} = 
        \begin{cases} 
        n \text { if $\zeta$ =1 (L'Hospital's)} \\ 
        0 \text{ else}. 
        \end{cases}
        $$
        For the identity element, we have
            $$
        	\Hstar(P;z)(\id) = \sum_{i=0}^{n-1}\left(1 + \frac{h_i^*-1}{n} n \right)z^i = \sum_{i=0}^{n-1} h_i^*z^i.
            $$
        For $g\ne \id$, our formula gives $\Hstar(P;z)(g)= \sum_{i=0}^{n-1}(1 + 0 )z^i$.
        Combining this with $\det(\id - \rho(g)z) = (1-z^n)$ gives
            $$
        	\frac{\Hstar(P;z)(g)}{\det(\id -\rho(g)z)} = \frac{1+z+\dots+z^n}{1-z^n} = \frac{1}{1-z}.
            $$
        \end{proof}

        Let $\Pi_n$ be the permutahedron of dimension $n-1$ in $\mathbb R^n$ for $n$ prime.
        Let $\mathbb Z / n \mathbb Z$ act on $\Pi_n$ by cyclically permuting coordinates of $\mathbb R^n$.
        
        \begin{corollary}\label{cor:Hstarprime}
        For an odd, prime number $n$, the $\Hstar$-series of $\Pi_n$ under the action of $\mathbb Z / n\mathbb Z$ is
        	$$
        	\Hstar(\Pi_n;z) = \sum_{i=0}^{n-1} \left(\chi_0 + \frac{h_i^* -1}{n}\sum_{j=0}^{n-1} \chi_j\right) z^i,
        	$$
        where $h_i^*$ is the coefficient of $z^i$ in the $h^*$-polynomial of $\Pi_n$.
        \end{corollary}


        We conclude our discussion of prime permutahedra by applying Stapledon's \Cref{cor:seventen} to show the existence of a $\Z/p\Z$-invariant nondegenerate hypersurface.
        In \Cref{sec:hypersurfaces} we continue in a similar vein, establishing a criterion that guarantees the nonexistence of such a hypersurface under the action of the symmetric group.
        
        \begin{theorem}
        Let $p$ be prime.
        Then for the action of $\Z/p\Z$ on $\Pi_p$, there exists a $\Z/p\Z$-invariant nondegenerate hypersurface with Newton polytope $\Pi_p$.
        Consequently $\Hstar(\Pi_p;z)$ is effective and polynomial.
        \end{theorem}

        \begin{proof}
        If $p=2$, then $\Pi_p$ is a line segment and the condition of \Cref{cor:seventen} is trivially satisfied.
        Suppose $p$ is an odd prime.
        Then every proper face of $\Pi_p$ with dimension $>1$ corresponds to some ordered set partition of $p$ with at least $2$ parts, which do not all have the same size and hence cannot possibly be fixed under the action of any $\sigma\in \Z/p\Z$ other than the identity.
        So the stabilizer of all proper faces is trivial, and hence all vertices are fixed under this stabilizer.
        Now consider the face $\Pi_p$.
        Since $n$ is odd, $\Pi_p$ contains the lattice point $(\frac{p+1}{2},\dots,\frac{p+1}{2})$, which is fixed under the action of $\Z/p\Z$.
        Hence by \Cref{cor:seventen}, the $H^*$-series is effective.
        \end{proof}
        
        \begin{remark} \Cref{thm:1modp} combined with \Cref{cor:Hstarprime} gives an alternative proof of effectiveness. Since $p$ is prime, $\frac{h_i^* -1}{p}$ is a nonnegative integer, so we can explicitly see $\Hstar(\Pi_p;z)$ is effective.
        \end{remark}

%% file: img/decorated_osp_example.tikz
\begin{tikzpicture}[scale=0.55]
    \filldraw[black] (0,3) circle (2pt) node[anchor=south] {\{1,2,7\}};
    \filldraw[black] (2.598,-1.5) circle (2pt) node[anchor=west] {\{3,5\}};
    \filldraw[black] (-2.598,1.5) circle (2pt) node[anchor=east] {\{4,6\}};
    \draw (0,0) circle (3);
    
    \filldraw[black] (5+8,3) circle (2pt) node[anchor=south] {\{1,2,3\}};
    \filldraw[black] (5+10.598,-1.5) circle (2pt) node[anchor=west] {\{4,6\}};
    \filldraw[black] (5+8-2.598,1.5) circle (2pt) node[anchor=east] {\{5,7\}};
    \draw (5+8,0) circle (3);
    
    \draw [red,thick,<-,domain=-20:70] plot ({3.5*cos(\x)}, {3.5*sin(\x)});
    \draw [red,thick,->,domain=145:108] plot ({3.5*cos(\x)}, {3.5*sin(\x)});
    \draw [red,thick,<-,domain=165:325] plot ({3.5*cos(\x)}, {3.5*sin(\x)});
    \filldraw[black] (3.672-0.5,9-7.520) circle (0.00001pt) node[anchor=west] {2};
    \filldraw[black] (-8.5+7.0208360839,9-12.172077255
    ) circle (0.00001pt) node[anchor=north] {3};
    \filldraw[black] (-8.5+6.393647419
    ,9-6.2047757148
    ) circle (0.00001pt) node[anchor=south] {1};

   \begin{scope}[shift={(13,0)}]
   
    \draw [red,thick,<-,domain=-20:70] plot ({3.5*cos(\x)}, {3.5*sin(\x)});
    \draw [red,thick,->,domain=145:108] plot ({3.5*cos(\x)}, {3.5*sin(\x)});
    \draw [red,thick,<-,domain=165:325] plot ({3.5*cos(\x)}, {3.5*sin(\x)});
    \filldraw[black] (3.672-0.5,9-7.520) circle (0.00001pt) node[anchor=west] {2};
    \filldraw[black] (-8.5+7.0208360839,9-12.172077255
    ) circle (0.00001pt) node[anchor=north] {3};
    \filldraw[black] (-8.5+6.393647419
    ,9-6.2047757148
    ) circle (0.00001pt) node[anchor=south] {1};

    \end{scope}

\end{tikzpicture}

%% file: 3.4_describing_invariant_hypersurfaces.tex
    \subsection{Describing Nondegenerate Hypersurfaces}\label{sec:hypersurfaces} 

    
    
    Up until now, we have mainly focused our attention on criteria \eqref{item:effective} and \eqref{item:polynomial} from \Cref{conj:Stapledon}.
    We now turn to \eqref{item:nondegen}, detailing what can be said about the existence of a $S_n$-invariant nondegenerate hypersurface in the case of orbit polytopes.

    \subsubsection{Orbit polytopes}
    
    Weight polytopes arise in representation theory of semisimple Lie algebras, and the combinatorics of these polytopes have been studied a great deal in recent years \cite{Postnikov,KhareRidenour,Khare,ACEP,Supina}.
    Here we specifically focus on weight polytopes of the Type A root system, we refer to these as \emph{orbit polytopes} (also sometimes known as \emph{permutohedra}).
    
    \begin{definition}
    An \emph{orbit polytope} is a polytope of the form $\conv\{\sigma\cdot \bw:\sigma\in S_n\}$ where $\bw\in\Z^n$ is any integer point, and $S_n$ acts by permuting its coordinates.
    \end{definition}
    
    Using the machinery of Hopf monoids, it is shown in \cite[Proposition 5.4]{AguiarArdila} that the faces of generalized permutahedra are products of generalized permutahedra.
    It is easily shown as a corollary (see \cite[Example 3.5]{Supina}) that the faces of orbit polytopes are products of orbit polytopes.
    Here, we are particularly interested in orbit polytopes which have some rectangular $2$-dimensional faces.
    Rectangles themselves are not orbit polytopes, so the only way for a rectangle to arise as a product of orbit polytopes is as the product of exactly two line segments and any number of points.
    
    
    
    
    \begin{example}\label{ex: max face}
    Let $\bw=(0,1,2,3)\in\R^4$.
    Then the orbit polytope $P$ of $\bw$ is the $3$-dimensional standard permutahedron.
    Consider the linear functional $\by = (0,0,1,1)\in(\R^4)^\ast$.
    Then the $\by$-maximal face of $P$ is the convex hull of the points in the $S_4$-orbit of $\bw$ that maximize $\by$.
    These are exactly the points where the $0$ and $1$ are in the first two positions and  the $2$ and $3$ are in the last two positions: $(0,1,2,3)$, $(1,0,2,3)$, $(0,1,3,2)$, and $(1,0,3,2)$.
    The convex hull of these four points gives a square face of $P$; it can be viewed as the product of a line segment living in a $2$-dimensional subspace of $\R^4$ with vertices $(0,1)$ and $(1,0)$, and another line segment living in the $2$-dimensional orthogonal complement with vertices $(2,3)$ and $(3,2)$.
    \end{example}
    
    \begin{theorem}\label{thm: op no hypersurface}
    Let $P$ be an orbit polytope which has a rectangular $2$-dimensional face $F$ with odd side lengths, and let $G$ be the symmetric group acting by the standard representation.
    Then the toric variety of $P$ does not admit a $G$-invariant non-degenerate hypersurface.
    \end{theorem}
    \begin{proof}
    Let $F$ be a rectangular $2$-face of $P$ with odd side lengths.
    We have established that $F$ can be written as the product of two line segments and any number of points.
    When we restrict a hypersurface in the toric variety of $P$ to the face $F$, the points will simply contribute a monomial factor to the expression for the hypersurface, which will not affect whether it is smooth at $F$.
    Hence we are free to ignore the points and focus our attention on the two line segments.
    
    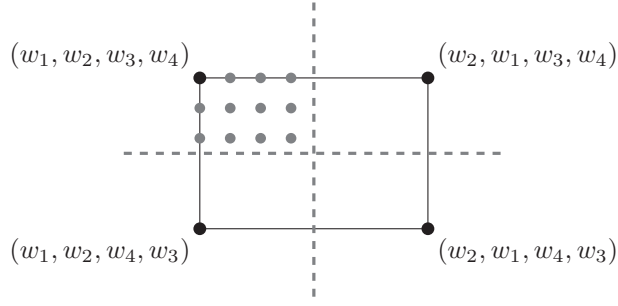
\begin{figure}
        \centering
        \input{img/odd_rectangle_face.tikz}
        \caption{A rectangular face arising as the product of two line segments.
        We are interested in the case where $w_1<w_2\le w_3<w_4$ and $w_2-w_1$ and $w_4-w_3$ are odd.
        Then the $S_n$-orbits of integer points in the rectangle are in correspondence with the integer points in the upper left quadrant of the figure, and each orbit consists of precisely one point in each quadrant.}
        \label{fig:oddrect}
    \end{figure}
    
    By analyzing the linear functionals that are maximized at the face $F$, one can show that the two line segments must have the form $\conv\{(w_1,w_2),(w_2,w_1)\}$ and $\conv\{(w_3,w_4),(w_4,w_3)\}$ where $w_1<w_2\le w_3<w_4$ (see \Cref{ex: max face}).
    Since the segments must have odd length, we further know that $w_2-w_1$ and $w_4-w_3$ are odd.
    If we look at the orbits of the lattice points in $F$ under the action of $S_n$, we will find that each orbit contains exactly four points, with exactly one in each quadrant of $F$ (see \Cref{fig:oddrect}).
    So the parameters appearing in the equation for our hypersurface restricted to $F$ can be identified with points in one quadrant of $F$.
    We will denote them by $c_{i,j}$ where $0\le i\le \lfloor\frac{w_2-w_1}{2}\rfloor$ and $0\le j \le \lfloor\frac{w_4-w_3}{2}\rfloor$.
    Using this, a generic $S_n$-invariant hypersurface restricted to the face $F$ of $P$ has the following equation:
        \[
        \sum_{j=0}^{\lfloor\frac{w_4-w_3}{2}\rfloor}\left(\left(x_3^{w_3+j}x_4^{w_4-j}+x_3^{w_4-j}x_4^{w_3+j}\right)\left(\sum_{i=0}^{\lfloor\frac{w_2-w_1}{2}\rfloor} c_{i,j}\left( x_1^{w_1+i}x_2^{w_2-i} + x_1^{w_2-i}x_2^{w_1+i} \right) \right)\right)=0
        \]
    
    Since $w_2-w_1$ is odd, then for each $i$ one of $w_1+i$ and $w_2-i$ will be odd and the other even.
    Likewise, for each $j$ one of $w_3+j$ and $w_4-j$ will be odd and the other even.
    Hence each term $ x_1^{w_1+i}x_2^{w_2-i} + x_1^{w_2-i}x_2^{w_1+i} $ vanishes when $x_2=-x_1$, regardless of the choices for $c_{i,j}$, and each term $x_3^{w_3+j}x_4^{w_4-j}+x_3^{w_4-j}x_4^{w_3+j}$ vanishes when $x_4=-x_3$. 
    Because the equation has a zero of multiplicity $2$ whenever $x_2=-x_1$ and $x_4 = -x_3$, the gradient vanishes at this locus, so no $S_n$-invariant hypersurface in the toric variety of $P$ is smooth at $F$.
    \end{proof}
    
    We will now characterize exactly which orbit polytopes have rectangular $2$-dimensional faces with odd side lengths.
    To do this, we will associate any given $S_n$-orbit polytope with a composition of $n$ using the following setup.
    
    \begin{setup}\label{setup: orbit polytopes}
    Given a point $\bw'\in \Z_n$, let $\bw$ be the representative from the $S_n$-orbit of $\bw'$ with coordinates in increasing order; $\bw = (w_1,\dots,w_1,w_2,\dots,w_2,\cdots,w_k,\dots,w_k)$ where $w_1<w_2<\dots<w_k$.
    From this, construct a composition $\alpha = (\alpha_1,\dots,\alpha_k)$ of $n$ where $\alpha_i$ is the number of $w_i$'s that appear in the coordinates of $\bw$.
    For example, if $\bw = (0,1,1,3,3,3,4,5)$ then the corresponding composition of $8$ would be $(1,2,3,1,1)$.
    \end{setup}
    
    Given two compositions $\alpha= (\alpha_1,\dots,\alpha_k)$ and $\beta = (\beta_1,\dots,\beta_\ell)$, we define their \emph{concatenation} to be $\alpha\concat\beta:= (\alpha_1,\dots,\alpha_k,\beta_1,\dots,\beta_\ell)$ and their \emph{near-concatenation} to be $\alpha\nconcat\beta:=(\alpha_1,\dots,\alpha_k+\beta_1,\dots,\beta_\ell)$.
    
    \begin{lemma}\label{lem:rectface}
    An orbit polytope has a rectangular $2$-dimensional face if and only if its corresponding composition $\alpha$ satisfies one of the following criteria:
        \begin{enumerate}[(i)]
            \item $\alpha$ has four or more parts, or
            \item $\alpha$ has three parts and the middle part is $\ge 2$.
        \end{enumerate}
    \end{lemma}
    \begin{proof}
    Proposition~4.19 of \cite{Supina} implies that for an orbit polytope with composition $\alpha$, any face can be described as a product of orbit polytopes with compositions $\beta^{(1)},\dots,\beta^{(r)}$ from which $\alpha$ can be obtained by some combination of concatenations and near-concatenations.
    Moreover, every such collection of compositions $\beta^{(1)},\dots,\beta^{(r)}$ gives a face of the orbit polytope.
    In order for this face to be a rectangle, we need two of the compositions $\beta$ to be $(1,1)$ (the composition of a line segment) and the remaining $\beta$'s to be $(1)$ (the composition of a point).
    When such a collection of compositions is combined together with any sequence of concatenations and near-concatenations, the resulting $\alpha$ is as described.
    \end{proof}
    
    \begin{corollary}\label{cor:oddrectface}
    An orbit polytope with composition $\alpha$ arising from the point $\bw$ (as in \Cref{setup: orbit polytopes}), that satisfies the condition of \Cref{lem:rectface}, has a rectangular $2$-face with odd side lengths if and only if one of the following criteria is satisfied:
        \begin{enumerate}[(i)]
            \item There exist $i$ and $j$ with $1\le i<i+1<j<j+1\le k$ such that the differences $w_{i+1}-w_i$ and $w_{j+1}-w_j$ are both odd.
            \item There exists an $i$ with $1\le i\le k-2$ such that $\alpha_{i+1}\ge 2$ and the differences $w_{i+1}-w_i$ and $w_{i+2}-w_{i+1}$ are both odd.
        \end{enumerate}
    \end{corollary}
    
    \subsubsection{Hypersimplices}
    
    We let the symmetric group $S_n$ acts on the hypersimplex $\Delta_{k,n}$ by permuting coordinates. The following theorem implies that $\Hstar$-series of the hypersimplex $\Delta_{k,n}$ is effective.
    
    \begin{theorem}\label{thm: hypersimplex hypersurface}
    The toric variety of the hypersimplex $\Delta_{k,n}$ admits an $S_n$-invariant non degenerate hypersurface.
    \end{theorem}
    \begin{proof}
    Let $e_k (x_1, x_2,\ldots ,x_n)$ be the $k$-th elementary symmetric polynomial in variables $x_1,\ldots,x_n$ i.e: 
    \begin{equation*}
        e_k (x_1, x_2,\cdots ,x_n)=\sum_{\substack{I\subseteq\{1,2,\ldots,n\}\\|I|=k}}\prod_{i\in I}x_i. 
    \end{equation*}
    Consider the hypersurface $H\subset (\mathbb{C}^{*})^{n}$ defined by the equation $e_k (x_1, x_2,\cdots ,x_n)=0$. It is straightforward to check that $H$ is $S_n$-invariant.
    
    We claim that $H$ is smooth using the induction on $k$. The base case $k=1$ is trivial as $\frac{\partial e_1 (x_1, x_2,\cdots x_n)}{\partial x_i}=1\neq0$. Assume the claim is true for $k-1$ and consider the case $k$. Suppose there exists a non smooth point $(a_1,a_2,\cdots,a_n) \in H$, then we have
    \begin{align*}
        &e_k(a_1,a_2,\ldots,a_n)=0\\
        &\frac{\partial e_k}{\partial x_i}(a_1,a_2,\ldots,a_n)=e_{k-1}(a_1,a_2,\ldots,\hat{a_i},\cdots,a_n)=0 \qquad \text{for $1\leq i\leq n$}.
    \end{align*}
   Since $\sum\limits_{i=1}^{n} e_{k-1}(a_1,a_2,\cdots,\hat{a_i},\cdots,a_n)=(n-k+1)e_{k-1}(a_1,a_2,\cdots,a_n)$, we obtain   \begin{equation}\label{eq: ek-1 zero}
       e_{k-1}(a_1,a_2,\ldots,a_n)=0.
   \end{equation}

From $e_{k-1}(a_1,a_2,\ldots,a_n)=e_{k-1}(a_1,a_2,\ldots,\hat{a_i},\ldots,a_n)+a_i e_{k-2}(a_1,a_2,\ldots,\hat{a_i},\ldots,a_n)$ and using \eqref{eq: ek-1 zero} gives 
\begin{equation}\label{eq: ek-2 zero}
    e_{k-2}(a_1,a_2,\ldots,\hat{a_i},\ldots,a_n)=0 \qquad \text{for $1\leq i\leq n$}.
\end{equation}
By the induction hypothesis, there is no $(a_1,a_2,\cdots,a_n)\in (\mathbb{C}^{*})^{n}$ satisfying \eqref{eq: ek-1 zero} and \eqref{eq: ek-2 zero}.

It remains to show $e_k(x_1,x_2,\ldots,x_n)\vert_{F}$ defines a smooth hypersurface inside a torus for every face $F$. Let $F$ be the face that maximizes the linear functional $L(x_1,x_2,\ldots,x_n)=\sum_{i=1}^{n}c_i x_i.$ We partition the set $\{1,2,\ldots,n\}$ with $S_1,S_2,\ldots,S_\ell$ so that:
\begin{itemize}
    \item For $i,j\in S_p$, $c_i=c_j$.
    \item For $p<q$ and $i\in S_p, j\in S_q$, $c_i>c_j$.
\end{itemize}
Let $h$ be the integer such that $\sum_{i=1}^{h}|S_i|\leq k$ and $\sum_{i=1}^{h+1}|S_i|> k$. Then vertices of $\Delta_{k,n}$ that maximizes $L$ are given by $(x_1,x_2,\ldots,x_n)$  such that
\begin{itemize}
    \item For $i\in S_p$ with $p\leq h$, $x_i=1$.
    \item For $i\in S_p$ with $p> h+1$, $x_i=0$.
    \item The number of $i$'s such that $x_i=1$ for $i\in S_{h+1}$ equals $(k-\sum_{i=1}^{h}|S_i|)$. For other $i$'s, $x_i=0$.
\end{itemize}
We conclude that $e_k(x_1,x_2,\ldots,x_n)\vert_{F}$ equals monomial times elementary symmetric functions with variables $\{x_i\}_{i \in S_{h+1}}$. Using the same argument, this defines a smooth hypersurface inside a torus. 
    \end{proof}

%% file: img/odd_rectangle_face.tikz
\begin{tikzpicture}[scale=2]

\coordinate (1234) at (0,1);
\coordinate (2134) at (1.5,1);
\coordinate (1243) at (0,0);
\coordinate (2143) at (1.5,0);
\node[anchor=south east] at (1234) {$(w_1,w_2,w_3,w_4)$};
\node[anchor=south west] at (2134) {$(w_2,w_1,w_3,w_4)$};
\node[anchor=north east] at (1243) {$(w_1,w_2,w_4,w_3)$};
\node[anchor=north west] at (2143) {$(w_2,w_1,w_4,w_3)$};

\draw (1234)--(2134)--(2143)--(1243)--cycle;

\draw[dashed, very thick, black!60] (.75,1.5) -- (.75, -.5);
\draw[dashed, very thick, black!60] (-.5,.5) -- (2,.5);

\foreach \i in {0,.2,...,.6}{
    \foreach \j in {.6,.8,1}{
        \fill[black!60] (\i,\j) circle(1pt);
    }
}

\fill (1234) circle(1.2pt);
\fill (2134) circle(1.2pt);
\fill (1243) circle(1.2pt);
\fill (2143) circle(1.2pt);






\end{tikzpicture}

%% file: 4_further_questions.tex
\section{Further Questions}\label{sec:further questions}
    
    
    \begin{question}
    \Cref{sec:zonotopes} uses zonotopal decompositions to compute the equivariant Ehrhart series of zonotopes of the Type A root system.
    How can this technique be adapted to general root systems?
    See \cite{ArdilaBeckMcWhirter} for some progress in this direction.
    \end{question}
    
    \begin{question}
    Can we characterize graphs $\Gamma$ for which every $\sigma\in\Aut(\Gamma)$ satisfies one of the conditions of \Cref{thm:graphic zonotope polynomial H*}?
    \end{question}
    
    \begin{question}
    Theorem 1.4 of \cite{BeckJochemkoMcCullough} gives a formula for the $\hstar$-polynomial of a zonotope in terms of the basis activity of its corresponding matroid.
    Is there an equivariant analogue of this formula that describes the $\Hstar$-series of a zonotope in terms relating to the $G$-action on its matroid?
    \end{question}
    
    \begin{question}
    For which $n$ is there a $G$-invariant triangulation as in Theorem \ref{thm:permrep} or Theorem \ref{thm:eeseries2} of the permutahedron $\Pi_n \in \R^n$ under the cyclic group action permuting the standard basis vectors?
    Can Theorem \ref{thm:eeseries2} be extended to compute equivariant Ehrhart theory using symmetric zonotopal tilings?
    \end{question}
    
    \begin{question}
    Can we find large families of polytopes exhibiting the $G$-invariant half-open triangulations described in Theorems \ref{thm:permrep} and \ref{thm:eeseries2}?
    \end{question}
    
    \begin{question}
    \Cref{thm: hypersimplex cyclic group} gives a combinatorial description of the $\Hstari$ characters for the action of $\Z/n\Z$ on the hypersimplex $\Delta_{k,n}$.
    However, we also know from \Cref{thm: hypersimplex hypersurface} that the $\Hstar$-series for the $S_n$ action on $\Delta_{k,n}$ is effective.
    Can we find a combinatorial interpretation of these characters $\Hstari$ of $S_n$?
    \end{question}
    
    \begin{question}
    \Cref{thm: op no hypersurface} shows that any orbit polytope containing a rectangular $2$-face with odd side lengths does not admit an $S_n$-invariant non-degenerate hypersurface.
    Is the converse true as well?
    In other words, are rectangular $2$-faces with odd length sides the only obstruction to the existence of an $S_n$-invariant nondegenerate hypersurface in the toric variety of an orbit polytope?
    \end{question}
    
    \begin{question}
    Suppose that the $G$-invariant polytope $P$ admits an invariant non-degenerate hypersurface, or has polynomial/effective $\Hstar$-series.
    Must the same then be true for the action of $G$ on positive integer dilations of $P$?
    \Cref{tab:dilates} summarizes what we know about the action of $S_4$ on the dilates of $\Pi_4$.
    \end{question}
    
    \begin{table}
    \centering
    \input{tab/Pi4_dilates.tex}
    \caption{Conditions of \Cref{conj:Stapledon} when $S_4$ acts on dilates of $\Pi_4$.}
    \label{tab:dilates}
    \end{table}
    
    \begin{question}
    Although condition \eqref{item:nondegen} of \Cref{conj:Stapledon} was proven to not be implied by conditions \eqref{item:effective} and \eqref{item:polynomial} (\Cref{thm:counterex}), it may still be the case that \eqref{item:effective} and \eqref{item:polynomial} are equivalent.
    Does polynomiality of the $\Hstar$-series imply effectiveness in general?
    Furthermore, what additional hypotheses need to be added to \Cref{conj:Stapledon} in order to make it true?
    \end{question}
    
    \begin{question}
    A topic of great interest in classical Ehrhart theory is determining which polytopes have $\hstar$-vectors that are \emph{unimodal}; i.e., each coefficient is weakly larger the previous, until a peak is reached, after which point each coefficient is weakly smaller than the previous.
    In the equivariant setting, we can interpret unimodality by looking at the differences $\Hstari-\Hstarimo$ and asking that these differences form a sequence consisting of effective characters up to a certain point, after which all the differences are negatives of effective characters.
    Which polytopes exhibit this equivariant unimodality property, and what parallels can be drawn with the classical concept of $\hstar$-unimodality?
    \end{question}
    
    \begin{question}
    The Ehrhart $f^\ast$-vector is a change-of-basis transformation of the $h^\ast$-vector, and $f^\ast$-positivity for polyhedral complexes is a weaker property than $\hstar$-positivity.
    What is the equivariant analogue of the $f^\ast$-vector, and what can we say about its effectiveness?
    This question was suggested by Katharina Jochemko.
    \end{question}

%% file: tab/Pi4_dilates.tex
\renewcommand*{\arraystretch}{2.5}
\begin{tabular}{|c|c|c|c|}
    \hline
         & $\Hstar$ polynomial? & $\Hstar$ effective? & $S_4$-inv.~non-deg.~hypersurface?\\
   \hline
        $\Pi_4$ & No\footnote{\label{foot:ASV}\cite{ASV}} & No\cref{foot:ASV} & No\footnote{\label{foot:square}\Cref{thm: op no hypersurface}}\\
    \hline
        $2\Pi_4$ & Yes\footnote{\label{foot:integral}All fixed polytopes are integral.} & Yes\footnote{\label{foot:sage}Verified in \texttt{Sagemath}.} & ?\\ 
    \hline
        $3\Pi_4$ & No\cref{foot:sage} & No\cref{foot:sage} & No\cref{foot:square}\\ 
    \hline
        $4\Pi_4$ & Yes\cref{foot:integral} & Yes\cref{foot:sage} & ?\\
    \hline
        $5\Pi_4$ & No\cref{foot:sage} & No\cref{foot:sage} & No\cref{foot:square}\\
    \hline
        Even dilates & Yes\cref{foot:integral} & ? & ?\\
    \hline
        Odd dilates & ? & ? & No\cref{foot:square}\\
    \hline
\end{tabular}

%% file: 5_acknowledgements.tex
\section{Acknowledgements}

We are grateful to Paco Santos and Alan Stapledon for sharing with us the counterexample of \Cref{thm:counterex}, and for a great deal of helpful email correspondence.
We also thank Vic Reiner for proposing this problem and Federico Ardila, Matt Beck, Christian Haase, Katharina Jochemko, and Kris Shaw for fruitful conversations.
Jean-Philippe Labb{\'e} provided useful advice in the development of the \texttt{Sagemath} package of \Cref{appendix:sage}.
Katharina Jochemko, as well as the Discrete Geometry group at FU Berlin, contributed travel funding for research visits that were instrumental in this project.

%% file: 6_appendix_Sagemath.tex
\appendix
\section{Calculating the \texorpdfstring{$\Hstar$}{H*}-series in \texttt{Sagemath}}\label{appendix:sage}

Elia implemented \texttt{Sagemath} functionality for calculating the $\Hstar$-series. For more detailed documentation, see \cite{Elia2022}.
It is open-source and staged for release with \texttt{Sagemath} version 9.6.
We show how to compute the $\Hstar$-series of the permutahedron $\Pi_4$ under the action of the symmetric group. First, we create the permutahedron, and set its backend to be \texttt{normaliz}, a necessary step for the usage of these methods.
Furthermore we create $\Pi_4$'s \texttt{restricted\_automorphism\_group}, which is the group of linear transformations mapping the polytope to itself and such that $d$-dimensional faces are mapped to $d$-dimensional faces. 
The output of the \texttt{restricted\_automorphism\_group} must be set to \texttt{permutation}. This means that every group element is expressed as a permutation of the vertices of the polytope. 
Later, we will create $S_4$ as a subgroup of the \texttt{restricted\_automorphism\_group}.
\small
\begin{lstlisting}
    sage: Pi4 = polytopes.permutahedron(4, backend='normaliz')
    sage: Pi4
    A 3-dimensional polyhedron in ZZ^4 defined as the convex hull of 24 vertices
    sage: G = Pi4.restricted_automorphism_group(output='permutation')
    sage: G
    Permutation Group with generators 
    [(1,6)(3,12)(4,8)(5,14)(9,18)(11,20)(15,19)(17,22),
    (0,1)(2,3)(4,5)(6,7)(8,9)(10,11)(12,13)(14,15)(16,17)(18,19)(20,21)(22,23),
    (0,2)(1,4)(3,5)(6,8)(7,10)(9,11)(12,14)(13,16)(15,17)(18,20)(19,22)(21,23),
    (0,23)(1,17)(2,21)(3,11)(4,15)(5,9)(6,22)(7,16)(8,19)(10,13)(12,20)(14,18)]
    sage: G.order()
    48
\end{lstlisting}
\normalsize 
The order of $G$ is 48; we must create $S4$ as a subgroup. Using our method \texttt{permutations\_to\_matrices}, which returns the matrix representation of elements of the restricted automorphism group, we see that the generator $(1,6)(3,12)(4,8)(5,14)(9,18)(11,20)(15,19)(17,22)$ does not correspond to the action of the symmetric group on $\Pi_4$:
\small
\begin{lstlisting}
    sage: conj_reps = G.conjugacy_classes_representatives()
    sage: Pi4.permutations_to_matrices(conj_class_reps=conj_reps, acting_group=G)
    {(1,6)(3,12)(4,8)(5,14)(9,18)(11,20)(15,19)(17,22): 
     [ 1/2  1/2  1/2 -1/2    0]
     [ 1/2  1/2 -1/2  1/2    0]
     [ 1/2 -1/2  1/2  1/2    0]
     [-1/2  1/2  1/2  1/2    0]
     [   0    0    0    0    1],
     ...
\end{lstlisting}
\normalsize
We create $S4$ as a subgroup of $G$ and compute the $\Hstar$-series using the method \texttt{Hstar\_function}:
\small
\begin{lstlisting}
    sage: S4 = G.subgroup(gens=[G.gens()[i] for i in range(1,4)])
    sage: S4.order()
    24
    sage: Pi4.Hstar_function(acting_group=S4, output='complete')
    {'Hstar': ((chi_0 + chi_2 + chi_4)*t^4 
     + (6*chi_0 + chi_1 + 6*chi_2 + 5*chi_3 + 9*chi_4)*t^3
     + (9*chi_0 + chi_1 + 8*chi_2 + 7*chi_3 + 14*chi_4)*t^2
     + (4*chi_0 + chi_1 + 3*chi_2 + 3*chi_3 + 5*chi_4)*t + chi_0)/(t + 1),
     'Hstar_as_lin_comb': ((t^4 + 6*t^3 + 9*t^2 + 4*t + 1)/(t + 1), 
     (t^3 + t^2 + t)/(t + 1), 
     t^3 + 5*t^2 + 3*t, (5*t^3 + 7*t^2 + 3*t)/(t + 1),
     (t^4 + 9*t^3 + 14*t^2 + 5*t)/(t + 1)),
     'conjugacy_class_reps': [(),
      (0,1)(2,3)(4,5)(6,7)(8,9)(10,11)(12,13)(14,15)(16,17)(18,19)(20,21)(22,23),
      (0,3,4)(1,5,2)(6,9,10)(7,11,8)(12,15,16)(13,17,14)(18,21,22)(19,23,20),
      (0,7)(1,6)(2,13)(3,12)(4,19)(5,18)(8,15)(9,14)(10,21)(11,20)(16,23)(17,22),
      (0,9,16,18)(1,11,22,12)(2,15,10,19)(3,17,20,6)(4,21,8,13)(5,23,14,7)],
     'character_table': [ 1  1  1  1  1]
     [ 1 -1  1  1 -1]
     [ 2  0 -1  2  0]
     [ 3 -1  0 -1  1]
     [ 3  1  0 -1 -1],
     'is_effective': False}
\end{lstlisting}
\normalsize
To reiterate, we have found, 
\begin{align*}
\Hstar(\Pi_4;z) =& \frac{(\chi_0 + \chi_2 + \chi_4)z^4  + (6\chi_0 + \chi_1 + 6\chi_2 + 5\chi_3 + 9\chi_4)z^3
 + (9\chi_0 + \chi_1 + 8\chi_2 + 7\chi_3 + 14\chi_4)z^2}{1+z} \\ 
 &+\frac{(4\chi_0 + \chi_1 + 3\chi_2 + 3\chi_3 + 5\chi_4)z + \chi_0}{1+z},
 \end{align*}
 where the characters are labeled according to the rows in the character table.
 
To see the documentation of the \texttt{Hstar\_function}, or of the related supporting methods, \texttt{fixed\_subpolytopes}, \texttt{permutations\_to\_matrices}, or indeed of any function in \texttt{Sagemath}, one can type \texttt{?} after the function:
\begin{lstlisting}
    sage: P = polytopes.cube(backend='normaliz')                                    
    sage: P.fixed_subpolytopes?
\end{lstlisting}
To see the both source code and the documentation simultaneously, type \texttt{??} after the function:
\begin{lstlisting}
    sage: P = polytopes.cube(backend='normaliz')                                    
    sage: P.Hstar_function??
\end{lstlisting}
